\documentclass[smallextended]{svjour3}

\smartqed

\usepackage{amssymb}

\usepackage[english]{babel}
\usepackage{amsmath}
\usepackage{amsfonts}
\usepackage{color}
\usepackage{graphicx}
\usepackage{cite}
\usepackage[linktocpage,pagebackref,colorlinks,linkcolor=blue,anchorcolor=blue,citecolor=blue,urlcolor=blue,hypertexnames=false]{hyperref}
\usepackage{algorithm}
\usepackage{algorithmicx}
\usepackage{algpseudocode} 
\usepackage{cancel}

\usepackage{multirow}

\newtheorem{assumption}{Assumption}

\numberwithin{equation}{section}

\newcommand{\bv}{\mathbf{v}} 
\newcommand{\bm}{\mathbf{m}}  
\newcommand{\ba}{\mathbf{a}}

\newcommand{\bc}{\mathbf{c}} 
\newcommand{\bx}{\mathbf{x}}
\newcommand{\by}{\mathbf{y}}
\newcommand{\bz}{\mathbf{z}}
\newcommand{\bg}{\mathbf{g}} 
  
\newcommand{\bff}{\mathbf{f}} 

\newcommand{\bu}{\mathbf{u}}
\newcommand{\bw}{\mathbf{w}}
\newcommand{\0}{\mathbf{0}}
\newcommand{\1}{\mathbf{1}}
\newcommand{\diag}{\text{diag}}
\newcommand{\sign}{\mathrm{sign}} 
\newcommand{\proj}{\mathbf{Proj}} 
\newcommand{\prob}{\mathbf{Prob}} 
\newcommand{\one}{\iota} 

\newcommand{\hB}{\mathcal{B}} 

\newcommand{\hL}{\mathcal{L}}
\newcommand{\hH}{\mathcal{H}}
\newcommand{\hN}{\mathcal{N}}
\newcommand{\xopt}{\bx_\mathbf{opt}}

\newcommand{\st}{\mbox{s.t.}}
\newcommand{\bU}{\mathbf{U}} 
\newcommand{\bbE}{\mathbb{E}}
\newcommand{\bbR}{\mathbb{R}}

\newcommand{\quotes}[1]{``#1''}
 
\begin{document}

\title{Adaptive Primal-Dual Stochastic Gradient Method for Expectation-constrained Convex Stochastic Programs}

\titlerunning{Adaptive Primal-Dual Stochastic Gradient Method}

\author{Yonggui Yan \and Yangyang Xu} 

\institute{Y. Yan, Y. Xu \at  Department of Mathematical Sciences, Rensselaer Polytechnic Institute, Troy, NY 12180\\
\email{\{yany4, xuy21\}@rpi.edu}}

\date{\today}

\maketitle
 
 \begin{abstract}
Stochastic gradient methods (SGMs) have been widely used for solving stochastic optimization problems. A majority of existing works assume no constraints or easy-to-project constraints.  
In this  paper, we  consider  convex stochastic optimization problems with expectation constraints. For these problems, it is often extremely expensive to perform projection onto the feasible set. Several SGMs in the literature can be applied to solve the expectation-constrained stochastic problems. We propose a novel primal-dual type SGM based on the Lagrangian function. Different from existing methods, our method incorporates an adaptiveness technique to speed up convergence. At each iteration, our method inquires an unbiased stochastic subgradient of the  Lagrangian function, and then it renews the primal variables by an adaptive-SGM update and the dual variables by a vanilla-SGM update. We show that the proposed method has a convergence rate of $O(1/\sqrt{k})$ in terms of the objective error and the constraint violation. Although the convergence rate is the same as those of existing SGMs, we observe its significantly faster convergence than an existing non-adaptive primal-dual SGM and a primal SGM on solving the Neyman-Pearson classification and quadratically constrained quadratic programs. Furthermore, we modify the proposed method to solve convex-concave stochastic minimax problems, for which we perform adaptive-SGM updates to both primal and dual variables. A convergence rate of $O(1/\sqrt{k})$ is also established to the modified method for solving minimax problems in terms of primal-dual gap. Our code has been released at \url{https://github.com/RPI-OPT/APriD}.


\end{abstract}

\section{Introduction}\label{sec1}
The stochastic approximation method can trace back to \cite{robbins1951stochastic}, where a root-finding problem is considered. Stochastic gradient methods (SGMs), as first-order stochastic approximation methods, have attracted a lot of research interests from both theory and applications. In the literature, a majority of works on SGMs focus on problems without constraints or with easy-to-project constraints.

In this work, we consider expectation-constrained convex stochastic programs, which are formulated as 
\begin{align}\label{eq:problem}
\min_{\bx\in X} \,&\, f_0(\bx) \equiv \bbE_{\xi_0}[F_0(\bx;\xi_0)], \nonumber \\
 \st \,&\,  f_i(\bx) \equiv  \bbE_{\xi_i}[F_i(\bx;\xi_i)] \leq 0,\, i \in [M].
\end{align}
Here, $[M]$ denotes $\{1,\ldots, M\}$, $X$ is a compact convex set in $\bbR^n$ that admits an easy projection, $\xi_i$ is a random variable and $f_i$ is a convex function on $X$ for each $i = 0, 1, ..., M$, and $\bbE_{\xi_i}$ takes the expectation about $\xi_i$. When $\xi_i$ is distributed uniformly on a finite set $\{\xi_{i1},\ldots,\xi_{iN_i}\}$ for each $i=0,1,\ldots,M$, \eqref{eq:problem} reduces to the finite-sum structured function-constrained problem
\begin{align}\label{eq:fs-problem}
\min_{\bx\in X} \,&\,  f_0(\bx) \equiv \frac{1}{N_0}\sum_{j=1}^{N_0}F_0(\bx;\xi_{0j})\nonumber\\
\st \,&\,  f_i(\bx) \equiv  \frac{1}{N_i}\sum_{j=1}^{N_i}F_i(\bx;\xi_{ij})\leq 0,\, i \in [M],
\end{align}
where each $N_i$ can be a big integer number.

Due to the possible nonlinearity of $\{f_i\}_{i=1}^M$, the expectation format, and the possibly big numbers $\{N_i\}_{i=1}^M$ and/or $M$, it can be extremely expensive to perform projection onto the feasible set of \eqref{eq:problem} or \eqref{eq:fs-problem}. Toward finding a solution of \eqref{eq:problem} or \eqref{eq:fs-problem} with a desired accuracy, we will design a primal-dual SGM that does not require the expensive projection. Our method will only require unbiased stochastic subgradients of the Lagrangian function of \eqref{eq:problem} or \eqref{eq:fs-problem} and the weighted projection onto $X$. To have fast convergence, we will incorporate into the designed method an adaptiveness technique, which has been extensively used in solving unconstrained problems, such as training deep learning models.

\subsection{Motivating applications}
Many  applications can be formulated into \eqref{eq:problem} or \eqref{eq:fs-problem}. We give a few concrete examples below.

 \vspace{0.2cm} 
\noindent\textbf{Neyman-Pearson classification.}~~One example  of  \eqref{eq:problem} is the Neyman-Pearson classification (NPC) \cite{rigollet2011neyman, scott2005neyman}.   
%A classifier parameterized by $\x$ is denoted as $h(\x;\ba)$. A sample $\ba$ is classified %the sample 
%based on the output of $h(\x;\ba)$. 
For a binary classification problem, a classifier $h(\bx;\ba)$, parameterized by $\bx$, is learned by finding an optimal $\bx$ from a set of training data. Each training data point $\ba$ has a label %$\{\ba\}$ labeled by $\{b\}$ where 
$b\in\{+1,-1\}$.  The classifier categorizes a new sample $\ba$ to the ``$+1$'' class  if $h(\bx,\ba)\geq 0$ and  ``$-1$'' otherwise.  The prediction incurs two types of errors: type-I error and type-II error. The former is also called false-positive error, defined as  
\[
R^{-}(\bx) = \prob(h(\bx;\ba)\geq 0 \mid b=-1 ) = \bbE[\one_{\geq0}(h(\bx;\ba))\mid b=-1]
\]
and the latter is also called false-negative error, defined as 
\[R^{+}(\bx) =  \prob(-h(\bx;\ba)> 0\mid b=+1)=\bbE[\one_{\geq0}(-h(\bx;\ba))\mid b=+1].\] 
Here, $b$ denotes the true label of $\ba$, and $\one_{\geq 0}$ denotes the indicator function of nonnegative numbers, i.e., $\one_{\geq 0}(z) = 1$ if $z\geq 0$ and $\one_{\geq 0}(z) = 0$ if $z<0$. 
In certain applications, the cost of making these two types of error could be severely different. For example, in medical diagnosis, missing a malignant tumor is much more concerned than diagnosing a benign tumor to a malignant one. Hence, for these applications, it would be beneficial to minimize the more concerned type of error while controlling another type of error at an acceptable level. 
NPC aims at minimizing the type-II error by controlling the type-I error.  It searches for the best parameter $\bx$  
in a given space $X$ 
by solving the error-constrained problem:
\begin{equation} \label{eq:NPproblem}
\min_{\bx\in X} \,  R^+(\bx), \quad \st \quad  R^-(\bx) \leq c,
\end{equation}
where $c>0$ is a user-specified level of error. 
Due to the nonconvexity of the indicator function $\one_{\geq0}(z)$, the problem in \eqref{eq:NPproblem} is often computationally intractable. Hence, it is common (cf. \cite{rigollet2011neyman}) to relax the indicator function $\one_{\geq0}(z)$ 
 by a convex  non-decreasing surrogate $\varphi(z):\bbR\rightarrow \bbR^+$.  
 This way,  if $h(\bx;\ba)$ is affine with respect to $\bx$ in a convex set  $X$, we can obtain a relaxed convex stochastic program in the form of \eqref{eq:problem} by relaxing $\one_{\geq0}(z)$ to $\varphi(z)$ with $z=h(\bx;\ba)$ and  $z=-h(\bx;\ba)$ respectively in the formulations of $R^-$ and $R^+$.
 
 \vspace{0.2cm} 
\noindent\textbf{Fairness-constrained classification.}~~
The data involved in a classification problem may contain certain sensitive features, such as gender and race. Fairness-constrained classification \cite{zafar2017fairness, zafar2019fairness} aims at finding a classifier that is fair to the samples with different sensitive features. For example, consider  the binary classification  problem.  
Let $(\ba, z)$ denote a data point, following a certain distribution and having a label $b\in\{+1, -1\}$.
Here, $z\in\{0,1\}$  represents a unidimensional binary sensitive feature. 
Suppose $h(\bx; \ba, z)$ is a classifier, parameterized  by $\bx$, such that a given sample $(\ba, z)$ is predicted to class $\widehat{b}=\sign(h(\bx; \ba, z))$. 
One can learn a fair classifier $h(\bx;\ba,z)$ by solving the problem:
\begin{align}
\min_{\bx\in X} \, L(\bx), \quad \st\quad   \prob(\widehat{b}\neq b \mid z=0) =  \prob(\widehat{b}\neq  b \mid z=1),  \label{eq:fairness}
\end{align}
where $L(\bx)$ measures a certain expected loss, e.g.,
$L(\bx)=\bbE\big[ \log\big(1+ \exp(-b \cdot h(\bx;\ba,z))\big) \big]$, and $X$ is a convex set. 
The constraint in \eqref{eq:fairness} is one kind of fairness constraint. More kinds of fairness constraints can be found in  \cite{zafar2019fairness}. Roughly speaking, a fairness constraint enforces that the classifier should not be correlated to the sensitive features. 
In general, the constraint in \eqref{eq:fairness} is nonconvex.  \cite{zafar2019fairness} introduces a relaxation  to \eqref{eq:fairness} by using  an  empirical  estimation  of the covariance between $z$ and $h(\bx;\ba,z)$, which results in the following problem:
\begin{align}\label{eq:fairness_app}
\min_{\bx \in X}\, \tilde L(\bx), \quad \st \quad   -c \le \frac{1}{n} \sum_{i=1}^n  (z_i-\bar{z})h(\bx; \ba_i,z_i) \leq c. 
\end{align}
Here, $\tilde L$ is an empirical loss, $\bar{z} =  \bbE[z]$, $\{(\ba_i, z_i)\}_{i=1}^n$ is a set of training samples, and $c>0$  is a given threshold  that trades off accuracy and unfairness.  If $h(\bx;\ba,z)$ is affine about $\bx$, then the constraint in \eqref{eq:fairness_app} will be convex about $\bx$, and in this case, \eqref{eq:fairness_app} with a convex loss function $\tilde L(\bx)$ becomes a convex problem in the form of \eqref{eq:fs-problem}.  

\vspace{0.2cm} 
\noindent\textbf{Chance-constrained problems.}~~
Another example is the chance-constrained problem. It can be modeled as (cf. \cite{shapiro2014lectures})   
\begin{equation}\label{eq:chance-prob}
\min_{\bx\in X} \,  f(\bx),\quad
\st\quad  \prob\{F(\bx,\xi)\leq 0\} 
\geq p,
\end{equation}
where $\xi\in \Xi$ is  a random vector, $X$ is  a convex  set, and $p\in(0,1)$ is  a user-specified probability value.  
Notice $ \prob\{F(\bx,\xi)\leq 0\}=1- \prob\{F(\bx,\xi)> 0\} =1-\bbE [\one_{> 0}(F(\bx,\xi))]$, where $\one_{> 0}$ denotes the indicator function of positive numbers. We can rewrite \eqref{eq:chance-prob} as  
\begin{equation}\label{eq:chance-prob2}
\min_{\bx\in X} \,  f(\bx),\quad
\st\quad \bbE[\one_{> 0}(F(\bx,\xi))]\leq 1- p.
\end{equation}
One approach to solve \eqref{eq:chance-prob2} is the sample average approximation method \cite{luedtke2008sample, pagnoncelli2009sample}. It approximates  the expectation  in  the  constraint by the empirical mean and forms the approximation problem:
\begin{align}
\min_{\bx\in X} \,  f(\bx),\quad \st\quad \frac{1}{N}\sum_{i=1}^N \one_{> 0} (F(\bx,\xi_i))\leq 1-\widehat{p},\label{eq:Chance_SAA}
\end{align}
where $\widehat{p}$ is carefully selected such that the solution of \eqref{eq:Chance_SAA} can satisfy certain desired property. When $f$ is convex, $F$ is linear about $\bx$, and the indicator function $\one_{>0}$  is relaxed by a convex surrogate, \eqref{eq:chance-prob2} and \eqref{eq:Chance_SAA} are convex in the form of  \eqref{eq:problem} and \eqref{eq:fs-problem} respectively. 

Another approach to solve \eqref{eq:chance-prob} is the scenario approximation  \cite{calafiore2006scenario, calafiore2005uncertain}, which takes $M$  independent samples $\xi_1,\xi_2,\xi_3,...,\xi_M$ of $\xi$ and forms the problem
\begin{align}
\min_{\bx\in X} \,  f(\bx),\quad \st\quad F(\bx,\xi_i)\leq 0, i \in [M]. \label{eq:Chance_scenario}
\end{align}
When $f(\bx)$ and  $F(\bx,\xi)$ are convex about $\bx\in X$  for all $\xi$,  \eqref{eq:Chance_scenario} becomes one instance of \eqref{eq:problem}  with $f_i(\bx)=F(\bx,\xi_i)$ for each $i$.  

\vspace{0.2cm}
\noindent\textbf{More.}~~More examples of \eqref{eq:problem} include the risk averse optimization (cf. \cite[Chapter 6]{shapiro2014lectures}), asset-allocation problems (cf. \cite{rockafellar2000optimization}), and the monotonicity constrained machine learning \cite{gupta2016monotonic}.    
 
\subsection{Related Work} \label{sec:literature} 
Several existing methods can be applied to solve \eqref{eq:problem} or \eqref{eq:fs-problem}. The method that we will propose is of primal-dual type and adopts an adaptiveness technique to have faster convergence. Below we review the existing works that are closely related to ours.

\subsubsection{Existing methods for solving \eqref{eq:problem} or \eqref{eq:fs-problem}}
 
Lan  and  Zhou  \cite{lan2020algorithms} proposes a cooperative stochastic approximation (CSA) method  for stochastic programs that have  one function  constraint. Let $g(\bx)  = \sum_{i=1}^M [f_i(\bx)]_+$. The constraint $\{\bx: f_i(\bx)  \leq 0, \forall\, i\in [M]\}$ can be reformulated to $g(\bx) \leq 0$. Hence, CSA can be applied to \eqref{eq:problem} with this reformulation.  
At the $k$-th  iterate $\bx^k$ for all $k\ge0$, CSA  first  calculates an unbiased estimator $\widehat{G}_k$ of $g(\bx^k) $. Then depending  on  whether  $\widehat{G}_k\leq  \eta_k$ or not for  some tolerance $\eta_k>0$, it moves either along the negative direction of an unbiased stochastic approximation of the subgradient $\tilde\nabla f_0(\bx^k)$ 
or $\tilde\nabla g(\bx^k)$.  Given a maximum iteration $K$, CSA takes a weighted average of all iterates indexed by $\hB^K = \{s\leq k \leq K \mid\widehat{G}_k\leq  \eta_k  \}$ for some $1\leq  s\leq K$. An ergodic convergence of $O(1/\sqrt{K})$ can be shown for CSA under the convexity assumption. 

When $M$ is big in \eqref{eq:problem} and $f_i$ is a deterministic function for each $i\in [M]$, one can apply the primal-dual SGM in \cite{xu2020primal} or its adaptive variant.  
The method in \cite{xu2020primal} is derived  by using  the augmented Lagrangian 
%(AL)
 function.  At the $k$-th iteration, it inquires an unbiased  stochastic subgradient $(\bu^k,\bw^k)$ of the %AL
  augmented Lagrangian function and then performs the updates 
\begin{equation}\label{eq:xu_x_update}
\bx^{k+1} = \proj_{X, D_k}(\bx^k-D_k^{-1}\bu^k),\quad
\bz^{k+1} = \bz^k+\rho_k\bw^k, 
\end{equation} 
where $\rho_k>0$ is a dual step size, and $D_k$ is a positive definite matrix. Two settings of $D_k$ are provided and analyzed in \cite{xu2020primal}. One is a non-adaptive setting with $D_k = \frac{I}{\alpha_k}$, where $I$ is the identity matrix and $\alpha_k>0$. 
The other is an adaptive setting with  $D_k = \diag(\bv^k)+\frac{I}{\alpha_k}$ and  $\bv^k = \eta\sqrt{\sum_{t=1}^k \frac{(\bu^t)^2}{\gamma_t^2}}$, where $\alpha_k>0, \eta > 0$ and $\gamma_t = \max\{1,\|\bu^t\|\}$. 
Under both settings, the method has an ergodic convergence rate of $O(1/\sqrt k)$ for the convex case. Notice that 
for the expectation or finite-sum constrained problems, the %AL
 augmented Lagrangian function will have the compositional expectation or finite-sum terms, which prohibit an easy access to an unbiased stochastic subgradient \cite{wang2017stochastic, xu2019katyusha}. Hence, the method in \cite{xu2020primal} will not apply to \eqref{eq:problem} with expectation constraints or \eqref{eq:fs-problem} with big numbers $\{N_i\}_{i=1}^M$.

If an easy projection can be performed onto $X_i = \{\bx\in \bbR^n: f_i(\bx)\leq 0\}$ for each $i \in [M]$, one can also apply the random projection methods \cite{wang2016stochastic, wang2015random} to \eqref{eq:problem}, and in addition, if the exact gradient of $f_0$ can be obtained, the stochastic proximal-proximal method in \cite{ryu2017proximal} can also be applied. However, if $X_i, i \in[M]$ are not simple sets  such  as  the logistic loss function induced constraint  set in  NPC \cite{rigollet2011neyman}, these random projection methods will be inefficient.
For the finite-sum constrained problem \eqref{eq:fs-problem}, level-set based methods are studied in \cite{lin2018level-fs}. For the special case of \eqref{eq:problem} where exact gradients of all the functions $\{f_i\}_{i=0}^M$ can be accessed, deterministic first-order methods have been studied in several works, e.g., \cite{xu2017first, xu2019iteration, yu2016primal, lin2018level-SIOPT, aybat2013augmented, lu2018iteration}. However, for \eqref{eq:fs-problem} with big numbers $\{N_i\}_{i=0}^N$, these deterministic methods would be inefficient, as computing exact gradients at each iteration is very expensive. 

Deterministic or stochastic gradient methods (e.g., \cite{nemirovski2009robust, hien2017inexact, hamedani2018iteration, hamedani2018primal, zhao2019optimal}) have also been studied for saddle-point problems, which can include the function-constrained problems as special cases if certain constraint qualifications hold. For example, one can solve the problem in \eqref{eq:problem} by the  saddle-point mirror stochastic approximation (MSA), which is proposed by Nemirovski et al.  \cite{nemirovski2009robust} to solve convex-concave stochastic saddle-point problems. Assuming the strong duality, we can reformulate \eqref{eq:problem} as a saddle-point problem by using the Lagrangian function $\hL(\bx,\bz)$; see \eqref{eq:lagrangian} below. Applying MSA to the reformulation will yield the following iterative update:
\begin{align}
\begin{pmatrix} \bx^{k+1} \\ 
\bz^{k+1} \end{pmatrix} =  \proj_{X \times Z}  
\left[\begin{pmatrix} \bx^{k} \\ \bz^{k} \end{pmatrix} - \gamma_k 
\begin{pmatrix} \bu^k \\  -\bw^k\end{pmatrix}\right],\label{eq:MSA}
\end{align} 
where $(\bu^k,\bw^k)$ is an unbiased stochastic subgradient of $\hL(\bx^k,\bz^k)$, $Z$ is an estimated compact set on the dual variable, and $\gamma_k>0$ is a step size. Given a maximum iteration $K$, MSA with $\gamma_k=\gamma/\sqrt{K}, \forall\, k$ for some $\gamma>0$ can achieve an ergodic convergence rate $O(1/\sqrt{K})$ in terms of the primal-dual gap. 

\subsubsection{Adaptive SGMs } 
A key difference between the existing methods reviewed above and our method for solving \eqref{eq:problem} or \eqref{eq:fs-problem} is the use of an adaptiveness technique, which has been widely used to improve the empirical convergence speed of SGMs for machine learning, especially deep learning. 
The vanilla SGM is easy to implement, but its convergence is often slow. 
Many adaptive variants have been proposed to improve the empirical convergence speed, such as %SGD with the momentum, 
Adagrad  \cite{duchi2011adaptive}, RMSprop \cite{tieleman2012lecture}, Adam \cite{kingma2014adam}, Adadelta \cite{zeiler2012adadelta}, Nadam \cite{dozat2016incorporating}, and  Amsgrad \cite{sashank2018convergence}. 
In the vanilla SGM, the same learning rate is used to all coordinates of the variable at each update, and this often results in poor performance for solving a problem with sparse training data \cite{duchi2011adaptive}.
In contrast, adaptive methods adaptively scale  the learning rate for different coordinates of the variable, based on a certain average of all the stochastic gradients that have been generated. The arithmetic average is used by Adagrad  \cite{duchi2011adaptive} and the exponential moving average by others \cite{tieleman2012lecture,kingma2014adam,zeiler2012adadelta,dozat2016incorporating,sashank2018convergence}. 
Besides the adaptive learning rate, momentum-based search direction is adopted in adaptive SGMs.  
These adaptive methods have been used a lot but mainly for 
problems without constraints or with easy-to-project constraints.
 
The adaptiveness technique that will be used in  our method is closely related to the technique used by Amsgrad \cite{sashank2018convergence}.  For a stochastic problem  
$\min_{\bx\in X} \, f_0(\bx) = \bbE_{\xi_0}[F_0(\bx;\xi_0)]$,  
Amsgrad iteratively performs the updates: 
\begin{align}
    \bm^k & = \beta_{1}\bm^{k-1} + (1-\beta_{1})\bu^k,   \nonumber\\
    \bv^k  & = \beta_2 \bv^{k-1} + (1-\beta_2)(\bu^k)^2, \nonumber \\
   \widehat{\bv}^k & = \max(\widehat{\bv}^{k-1}, \bv^k),  \label{eq:max_v} \\ 
   \bx^{k+1} &= \proj_{X,(\widehat{\bv}^k)^{1/2}}(\bx^k - \alpha_k \bm^k/(\widehat{\bv}^k)^{1/2}), \nonumber
\end{align}
where $\bu^k$ is an unbiased stochastic gradient of $f_0$ at $\bx^k$. Slightly different from Amsgrad, Adam \cite{kingma2014adam} sets $\widehat{\bv}^k=\bv^k$ for all $k$. It turns out that the setting in \eqref{eq:max_v} is important, as otherwise the method is not guaranteed to converge even for convex problems.   
The update  in \eqref{eq:max_v} ensures that each component of $ \widehat{\bv}^k$ is nondecreasing with respect to $k$. This addresses the divergence issue of Adam. However, one drawback is that once an  unusual $\bu^k$ with ``big'' components appears in early iterations,  the components of $ \widehat{\bv}^k$ will remain ``big'', which results small effective learning rates and can slow the convergence. We address this issue by using the clipping technique in \cite{abadi2016deep} to modify the update of $ \widehat{\bv}^k$. 

\subsection{Contributions}
The main contributions are listed below.
\begin{itemize}
\item We propose a new  adaptive primal-dual stochastic gradient method (APriD) for solving expectation-constrained convex  stochastic optimization problems. The method is derived based on  the ordinary Lagrangian function of the underlying problem. During each iteration, APriD first inquires an unbiased stochastic subgradient of the Lagrangian function about both primal and dual variables, and then it performs a stochastic subgradient descent update to the primal variable and a stochastic gradient ascent update to the dual variable. For the primal update, an adaptive learning rate is applied, in order to have faster convergence.   
\item We analyze the convergence rate of APriD. Assuming the existence of a Karush-Kuhn-Tucker (KKT) solution, we prove the boundedness of  the dual iterate in expectation. Based on that, we then establish the $O(1/\sqrt{k})$ ergodic convergence rate, where  $k$ is the number of subgradient inquiries. The  convergence rate results are in terms of primal and dual objective value and primal constraint violation.
\item We also extend APriD to solve a convex-concave minimax problem. The extended method (named as APriAD), during each iteration, first inquires an unbiased estimation of a subgradient of the minimax problem and then updates the primal  and dual variables by performing an adaptive update. An ergodic convergence rate of $O(1/\sqrt k)$ is also established in terms of the primal-dual gap.
\item We  test APriD on NPC \cite{scott2005neyman, rigollet2011neyman} and the quadratically constrained quadratic program (QCQP) in the expectation form and the scenario approximation form. We compare the proposed method to CSA \cite{lan2020algorithms} and MSA \cite{nemirovski2009robust}. Although all three methods have the same-order convergence rate, the numerical results demonstrate that APriD can decrease the objective error and the constraint violations (sometimes significantly) faster than CSA and MSA, in terms of iteration number and the running time. 
\end{itemize}

\subsection{Notation}
We use bold lower-case letters $\bx,\bz$ for vectors and $x_i,z_i$ for their $i$-th components. The bold number $\0,\1$ denote the all-zero vector and all-one vector, respectively. 
We use the convection $0/0 = 0$.
For any positive integer $M$, $[M]$ is short for the set $\{1,2,...,M\}$. 
With slight abuse of notation, for any two vectors $\bx$ and $\by$, $\lvert \bx\rvert $ takes element-wise absolute value, 
$\bx^p$ takes the element-wise $p$-{th} power, $\bx/\by$  takes the element-wise division, and $\max\{\bx,\by\}$ takes the element-wise maximum. We define $[\bx]_+ = \max\{\bx,\0\}$ and  $[\bx]_- = \max\{-\bx,\0\}$. $\bx\leq\by$ and $\bx\geq \by$ mean that the inequalities hold element-wisely.  

For a given symmetric positive semidefinite matrix $D$, we define the weighted product between $\bx, \bz$ as $\langle \bx,\bz\rangle_D = \bx^\top D\bz$.  Specially, when $D$ is  the  identity  matrix $I$, we briefly  denote $\langle \bx,\bz\rangle=\langle \bx,\bz\rangle_I  = \bx^\top\bz$.
For any vector $\bv$, $\mbox{diag}(\bv)$ denotes a diagonal matrix with $\bv$ on the diagonal.  For a nonnegative vector $\bv\geq \0$,
 $\langle \bx,\bz\rangle_{\bv} = \bx^\top \mbox{diag}(\bv)\bz$.
For any vector $\bx$, we denote  $\|\bx\| = \sqrt{\langle \bx,\bx\rangle} $,   $\|\bx\|_D = \sqrt{\langle \bx,\bx\rangle_D}$, $\|\bx\|_{\bv} = \sqrt{\langle \bx,\bx\rangle_{\bv}} $, $\|\bx\|_1 = \sum_{i}\lvert x_i\rvert $, and $\|\bx\|_\infty = \max_{i}\lvert x_i\rvert $. 
Similarly, for a closed convex set $X$, we denote the (weighted) projection onto $X$ as 
$\proj_{X}(\by) = \text{arg}\min_{\bx\in X} \, \|\bx-\by\|^2$, 
$\proj_{X,D}(\by) = \text{arg}\min_{\bx\in X} \, \|\bx-\by\|_D^2$
and $\proj_{X,\bv}(\by) = \text{arg}\min_{\bx\in X} \, \|\bx-\by\|_{\bv}^2$.

For a convex function $f(\bx)$, we denote $\tilde{\nabla} f(\bx)$ as one subgradient and $\partial f(\bx)$ as the set of all subgradients of $f$ at $\bx$. If $f$ is differentiable, we simply use $\nabla f(\bx)$
as the gradient. For a function $f(\bx,\by)$,  we denote the partial  gradient or partial subgradient about $\bx$ by $\nabla_\bx f(\bx,\by)$ or $\tilde{\nabla}_\bx f(\bx,\by)$ and the set of all partial subgradients by  $\partial_\bx f(\bx,\by)$.
In the analysis of our algorithm, we denote $\hH^k$ as all history information until the $k$-th iteration, i.e., $\hH^k  =  \{\bx^1,\bz^1,  \bx^2,  \bz^2,..., \bx^k,\bz^k\}$.

\section{Adaptive primal-dual stochastic gradient method} 

In this section, we give the details of our adaptive primal-dual stochastic gradient method. Since the algorithm only requires the stochastic subgradient of the Lagrangian function of the underlying problem, it can be applied to both \eqref{eq:problem}  and \eqref{eq:fs-problem} by specifying how to obtain the stochastic subgradient. Hence, we focus on the more general formulation in \eqref{eq:problem}.

Let $ \bff(\bx) = [f_1(\bx),f_2(\bx),...,f_M(\bx) ]^\top $ be the vector function that concatenates all the constraint functions in \eqref{eq:problem}. 
Then the Lagrangian function of \eqref{eq:problem} can be written as 
\begin{equation}\label{eq:lagrangian}
\hL (\bx,\bz) = f_0(\bx) + \bz^\top \bff(\bx), 
\end{equation}
where  $\bz = [z_1,z_2,...,z_M]^\top $ is the Lagrange multiplier or dual variable. The Lagrangian dual problem is  
\begin{align}
\max_{\bz\ge\0}\, \big\{d(\bz) \equiv \min_{\bx\in X} \,  \hL(\bx,\bz)\big\}.  \label{eq:dual_function}
\end{align}

By  the Lagrangian function in \eqref{eq:lagrangian}, we design an {\bf a}daptive  {\bf pri}mal-{\bf d}ual stochastic gradient method for solving \eqref{eq:problem}, named as APriD.  
The pseudocode is shown in  Algorithm~\ref{alg:APriD}. We assume an \emph{oracle} that can return an unbiased stochastic subgradient of $\hL$ at any inquiry point $(\bx,\bz)$. 
At each  iteration $k$, APriD first calls the oracle to obtain a stochastic subgradient $(\bu^k,\bw^k)$ of $\hL$ at $(\bx^k,\bz^k)$.  Then it updates the primal  variable by performing an adaptive-SGM step and the  dual  variable by a vanilla-SGM step. Our adaptive update is similar to Amsgrad \cite{sashank2018convergence}, and the difference is that we clip the primal stochastic gradient $\bu^k$ if its norm is greater than a user-specified parameter $\theta$. The clipping technique has been used in existing works about adaptive SGMs, e.g., \cite{luo2019adaptive}, to avoid potentially too-small effective learning rate.

\begin{algorithm}[htbp]
\caption{Adaptive primal-dual stochastic gradient (APriD) method for \eqref{eq:problem}}
\label{alg:APriD}
    \begin{algorithmic}[1]
    \State \textbf{Initialization:}  choose $\bx^1\in X$ and $\bz^1 \geq \0$, set $\bm^0 = \0, \bv^0 = \0, \widehat{\bv}^0 = \0$; 
   \State \textbf{Parameter setting:}  set the maximum number $K$ of iterations; choose $\beta_1, \beta_2 \in (0,1)$, $\theta>0$, non-increasing step sizes $\{\alpha_k\}_{k= 1}^K$  and $\rho_1>0$; if $\alpha_k=\alpha_1$ for $k\in[K]$,  i.e., $\alpha_k$ is a constant,  let $\eta_1 = \sum_{i=1}^\infty \alpha_i\beta_1^{i-1} = \frac{\alpha_1}{1-\beta_1}$; otherwise, let $\eta_1 = \sum_{i=1}^K \alpha_i\beta_1^{i-1}$.
      \For {$k = 1,2,...,K$}
      \State Call the oracle to obtain a stochastic subgradient $(\bu^k,\bw^k)$ of $\hL$ at $(\bx^k,\bz^k)$.
      \State Update the primal variable $\bx$ by 
              \begin{align} 
               \bm^k & = \beta_{1}\bm^{k-1} + (1-\beta_{1})\bu^k, \label{eq:m_update} \\
               \widehat{\bu}^k & =  \frac{\bu^k}{\max\big\{1, \frac{\|\bu^k\|}{\theta}\big\}}, \label{eq:u_hat}\\
               \bv^k  & = \beta_2 \bv^{k-1} + (1-\beta_2) (\widehat{\bu}^k)^2,   \label{eq:v_update1}\\
               \widehat{\bv}^k & = \max\big\{\widehat{\bv}^{k-1}, \bv^k\big\},\label{eq:v_update2} \\ 
               \bx^{k+1} &= \proj_{X,(\widehat{\bv}^k)^{1/2}}\big(\bx^k - \alpha_k \bm^k/(\widehat{\bv}^k)^{1/2}\big). \label{eq:x_update} 
              \end{align} 
                      
        \State If $k\geq 2$, update the step size $\rho_k$ by    
                	  \begin{align} 
			    \rho_k &= \frac{\rho_{k-1}}{\beta_1 + \frac{\alpha_{k-1} }{\eta_k}}, \text{ with } \eta_k = \frac{\eta_{k-1}-\alpha_{k-1}}{\beta_1}. \label{eq:rho_from_eta_update}
		   \end{align}    
	\State Update the dual variable $\bz$ by
        		   \begin{align} 
			\bz^{k+1} = [\bz^k + {\rho_k}\bw^k]_+.   \label{eq:z_update} 			            \end{align} 
       	      \EndFor  
    \end{algorithmic}
\end{algorithm}   

Below, we give a few facts about the step-size parameters of Algorithm~\ref{alg:APriD}. 
About the primal step size, the inequality in the following lemma will be used many  times in our analysis. 
\begin{lemma}\label{lem:beta_ineq} For any $1\leq j\leq t \leq K$,  we have 
\begin{align}
\alpha_j\leq \sum_{k=j}^t \alpha_k  \beta_1^{k-j}  \leq  \frac{\alpha_j}{1-\beta_1}.\label{eq:beta_ineq}
\end{align} 
\end{lemma}
\begin{proof} 
The first inequality is trivial from the positivity of $\{\alpha_j\}_{j= 1}^K$ and $\beta_1$. The second inequality follows from the non-increasing monotonicity  of $\{\alpha_j\}_{j=1}^K$ and by noting $\sum_{k=j}^t \alpha_k  \beta_1^{k-j} \leq \alpha_j\sum_{k=j}^{t} \beta_1^{k-j} = \alpha_j\sum_{k=0}^{t-j} \beta_1^{k} \leq \alpha_j\sum_{k=0}^\infty \beta_1^{k}$. 
The proof is finished by $\beta_1 \in (0,1)$.
\end{proof} 

About the  dual step size,  we show the next lemma, which will be used to analyze  the summation  of  $\big\|\bz^j-\bz\big\|^2$. % in  \eqref{eq:sum_z1}. 
The  proof of  the lemma is given in Appendix~\ref{app:rho}.
\begin{lemma}\label{lem:rho} The dual step size sequence $\{\rho_j\}_{j= 1}^K$ is non-increasing and
\begin{align}
&\frac{\sum_{k=j-1}^t \alpha_k  \beta_1^{k-(j-1)}}{\rho_{j-1}} - \frac{\sum_{k=j}^t \alpha_k  \beta_1^{k-j}}{\rho_j} \geq 0,  \quad \mbox{for any  }  2\leq j\leq t\leq K, \label{eq:rho_1}\\
&\rho_j\leq  \frac{ \rho_{1} \alpha_j}{\alpha_1(1\!-\!\beta_1)}, \quad \mbox{for any  }j \in [K]. \label{eq:rho_2}
\end{align}
Furthermore, if a constant primal step size is set, i.e.,  
$\alpha_j = \alpha_1,$ for all $j \in [K]$ for some $\alpha_1>0$, then $\eta_j = \frac{\alpha_1}{1-\beta_1}$ and $\rho_j = \rho_1$  for all $j \in [K]$, i.e.  the dual step size is also a constant.  
 \end{lemma} 

 \section{Convergence analysis}
In this section, we analyze the convergence of Algorithm~\ref{alg:APriD}.  
We first  give  the assumptions  in subsection~\ref{subsection:assu} and then show some preparatory lemmas in subsection  \ref{subsection:lem}. In subsection~\ref{subsection:conv}, we give the main convergence results, including a uniform bound on $\bbE[\bz^k]$.   
  
\subsection{Technical assumptions}\label{subsection:assu}
Throughout our analysis in this section, we  make the following three assumptions.
\begin{assumption}\label{assu:compact}
$X$ is a compact convex set in $\bbR^n$, i.e. there exists a constant $B$  such that 
\[
\|\bx_1-\bx_2\|_\infty \leq B, \quad \forall\, \bx_1,\bx_2\in X. 
\]
\end{assumption} 
For any $\bx\in \bbR^n$, $\|\bx\|\leq \sqrt{n}\|\bx\|_\infty$, and from Assumption~\ref{assu:compact},  it holds
\begin{equation}
\|\bx_1-\bx_2\|^2 \leq nB^2, \quad\forall\, \bx_1,\bx_2\in X.  \label{eq:nB2}
\end{equation}

\begin{assumption}\label{assu:bound}
There are constants $P, Q, F$ such that for any $ k\in[K]$,
\[\bbE\big[\bu^k\mid \hH^k\big] \in \partial_\bx \hL(\bx^k,\bz^k),  \quad \bbE\big[\|\bu^k\|^2\big] \leq P\cdot \bbE\big[\big\|\bz^k\big\|^2\big]+Q, \]
\[\bbE\big[\bw^k\mid\hH^k\big] = \bff(\bx^k) = \nabla_\bz \hL(\bx^k,\bz^k) , 
\quad  \bbE\big[\|\bw^k\|^2\big] \leq F^2.  \]
\end{assumption}
By Assumption~\ref{assu:bound}, we have for any $k\in[K]$,
\begin{equation}\label{eq:P_hat}
\max_{j\in [k]}  \bbE\big[\|\bu^j\|^2\big] \leq P\cdot\left(\max_{j\in [k]}\bbE\big[\|\bz^j\|^2\big]\right)+Q\equiv\widehat{P}_\bz^k.
\end{equation}
\begin{assumption}\label{assu:kkt}
There exists a primal-dual solution $(\bx^*,\bz^*)$ satisfying the KKT conditions of \eqref{eq:problem}:
\begin{align}
& \0 \in \partial f_0(\bx^*) + \hN_X(\bx^*) + \sum_{i=1}^{M}z_i^*\partial f_i(\bx^*), \label{eq:kkt1}\\
& \bx^* \in X,\quad  f_i(\bx^*)\leq 0, \quad \forall i \in [M], \label{eq:kkt2}\\
& z^*_i \geq 0,\quad  z^*_if_i(\bx^*) = 0, \quad \forall i \in [M], \label{eq:kkt3}
\end{align}
where  $\hN_X(\bx)$ denotes the normal cone of $X$ at $\bx$.
\end{assumption}

In Assumption~\ref{assu:bound}, the unbiasedness condition on the stochastic gradients is standard in the literature of SGMs. As $X$ is compact, the uniform boundedness on $\bbE\big[\|\bw^k\|^2\big]$ can hold if $\bff$ is continuous. However, the bound on $\bbE\big[\|\bu^k\|^2\big]$ often depends on the $\bz$-variable. To see this, 
we suppose that $\bu^{k}_{i}$ is an unbiased stochastic subgradient of $f_i$ at $\bx^k$ for each $i=0,1,\ldots, M$ and let
\begin{equation*}%\label{eq:unbiased_u}
\bu^k = \bu^{k}_{0} +  \sum_{i=1}^M z_i^k\bu^{k}_{i}.
\end{equation*}
Then $\bu^k$ will be an unbiased stochastic subgradient of $\hL(\cdot, \bz^k)$ at $\bx^k$. Denote $\bU^k=[\bu^{k}_{1},\ldots,\bu^{k}_{M}]$. We have
\begin{align*}
\bbE\big[\|\bu^k\|^2\big]\leq & 2\bbE\big[\|\bu^{k}_{0}\|^2\big]+2\bbE\big[\|\bz^k\|^2\|\bU^{k}\|^2\big] \\
=& 2\bbE\Big[\bbE\big[\|\bu^{k}_{0}\|^2\mid\hH^k\big]\Big]+2\bbE\Big[\|\bz^k\|^2 \bbE\big[\|\bU^{k}\|^2 \mid \hH^k\big]\Big]
\end{align*}
As $X$ is compact, $\bbE\big[\|\bu^{k}_{0}\|^2\mid \hH^k\big]$ and $\bbE\big[\|\bU^{k}\|^2\mid\hH^k\big]$ can be uniformly bounded by some constants $P$ and $Q$, and thus the boundedness condition on $\bbE\big[\|\bu^k\|^2\big]$ holds.   

Assumption~\ref{assu:kkt} is satisfied if a certain constraint qualification holds such as the Slater's condition.    
Under Assumption~\ref{assu:kkt}, the strong duality holds, i.e.
\begin{align}
d(\bz^*) =  f_0(\bx^*). \label{eq:strong_dual}
\end{align} 
where $d$ is defined in \eqref{eq:dual_function}.

\subsection{Preparatory lemmas}\label{subsection:lem}

In this subsection, we establish a few lemmas, whose proofs are given in Appendixes~\ref{app:3.1}-\ref{app:cov}.
First, we have the following lemma under  Assumption~\ref{assu:kkt}.
\begin{lemma}\label{lem:3.1} 
For any primal-dual solution $(\bx^*,\bz^*)$ satisfying the KKT conditions \eqref{eq:kkt1}-\eqref{eq:kkt3}, we have that for any $\bx\in X$,
\begin{align} \label{eq:2.41}
f_0(\bx)-f_0(\bx^*)  + \langle \bz^*,\bff(\bx)\rangle \geq 0.
\end{align}
Furthermore, for any $\bz\geq 0$, 
\begin{align} \label{eq:2.4}
f_0(\bx)-f_0(\bx^*) - \langle \bz, \bff(\bx^*)\rangle  + \langle \bz^*,\bff(\bx)\rangle \geq 0.
\end{align}
\end{lemma}

Under Assumption~\ref{assu:bound} and by the updates in Algorithm \ref{alg:APriD}, we are able to upper bound $\bbE \big[\|(\widehat{\bv}^k)^{1/2}\|_1\big]$ and $ \bbE [\|\bm^{k}\|_{(\widehat\bv^k)^{-{1/2}}}^2]$, as summarized below.
\begin{lemma}\label{lem:bound_v}
 For any integer $k\in [K]$,
 \begin{align*} 
   \bbE \big[\|(\widehat{\bv}^k)^{1/2}\|_1\big] \leq n\theta, \quad  \bbE \big[\|\bm^{k}\|_{(\widehat\bv^k)^{-{1/2}}}^2\big] \leq   \frac{\sqrt{n}(\theta + \frac{\widehat{P}_\bz^k}{\theta})}{(1-\beta_2)^{1/2}},
\end{align*}  
where $\widehat{P}_\bz^k$ is  defined in  \eqref{eq:P_hat}.
\end{lemma}

From  the  update in \eqref{eq:x_update} about the primal variable, we have  the  following result. 
\begin{lemma} \label{lem:primal_x} 
For  any  $ \bx\in X$ and $t\in [K]$,
\begin{align}\label{eq:primal_x_bbE}
  (1-\beta_{1}) \sum_{j=1}^t \bbE\left[\big\langle\bx^{j} - \bx, \bu^{j} \big\rangle\right] \sum_{k=j}^t \alpha_k  \beta_1^{k-j}
\leq \frac{n\theta B^2}{2}   +  \frac{ \sqrt{n}(\theta + \frac{\widehat{P}_\bz^t}{\theta}) \sum_{j=1}^t \alpha_j^2 }{2(1\!-\!\beta_1)^2(1\!-\!\beta_2)^{1/2}}.
\end{align}
\end{lemma}
   
From  the  update in \eqref{eq:z_update} about the dual  variable,  we have the follows.  
\begin{lemma} \label{lem:dual_z}
For any $\bz\geq \0$, $j \in [K]$, it holds
\begin{align}
 \big\langle\bz^{j} ,   \bff(\bx^j)\big\rangle \geq &  \big\langle \bz,   \bff(\bx^j)\big\rangle  +\frac{1}{2\rho_j}\Big(\big\|\bz^{j+1} - \bz\big\|^2   - \big\|\bz^j  -\bz\big\|^2\Big)  \nonumber\\
 & - \frac{{\rho_j}}{2} \big\|\bw^j\big\|^2 - \big\langle\bz^{j}-\bz,  \bw^j - \bff(\bx^j)\big\rangle. \label{eq:dual_z}
\end{align}
\end{lemma}

Lemma~\ref{lem:primal_x} gives an upper bound of  $\sum_{j=1}^t \bbE\big[\big\langle\bx^{j} - \bx, \bu^{j} \big\rangle\big] \sum_{k=j}^t \alpha_k\beta_1^{k-j}$, the next  lemma gives a lower bound on this summation.  
\begin{lemma}\label{lem:sum_u}
For $\bz\geq \0$ and $t\in [K]$ we have
\begin{align} 
&  \sum_{j=1}^t \bbE\big[\big\langle\bx^{j} - \bx, \bu^{j}\big\rangle\big] \sum_{k=j}^t \alpha_k\beta_1^{k-j}\nonumber\\ 
\geq&  \sum_{j=1}^t \bbE\big[f_0(\bx^j)-f_0(\bx) - \big\langle\bz^j, \bff(\bx)\big\rangle  +  \big\langle\bz, \bff(\bx^j)\big\rangle\big] \sum_{k=j}^t \alpha_k\beta_1^{k-j}\nonumber\\
& + \sum_{j=1}^t \bbE\Big[-\big\langle\bz^{j}-\bz, \bw^j - \bff(\bx^j)\big\rangle +\big\langle\bx^{j} - \bx, \bu^{j}-\tilde{\nabla}_\bx \hL (\bx^j,\bz^j) \big\rangle  \Big] \sum_{k=j}^t \alpha_k\beta_1^{k-j} 
 \nonumber\\
& -\frac{ \alpha_1\bbE\big[\big\|\bz^1-\bz\big\|^2\big]}{2\rho_1(1-\beta_{1})}  - \frac{ \rho_1 F^2\sum_{j=1}^t \alpha_j^2}{2\alpha_1(1\!-\!\beta_1)^2} 
+ \frac{\alpha_t }{2\rho_t}\bbE\big[\big\|\bz^{t+1}-\bz\big\|^2 \big], \label{eq:sum_u}
\end{align}  
where $ \tilde{\nabla}_\bx \hL (\bx^j,\bz^j)  = \bbE\big[\bu^j\mid \hH^j\big] $.
\end{lemma}

%For the last term  in  \eqref{eq:sum_u}, we  can bound it  after taking expectation and obtain the follows.
The second term  in  the right side of \eqref{eq:sum_u} can be bounded as follows.
\begin{lemma}\label{lem:cov} For any deterministic or stochastic vector $(\bx,\bz)$ with $\bx\in X$ and $\bz\geq \0$, it holds for any positive number sequence $\{\gamma_j\}_{j= 1}^K$, and $t\in [K]$ that 
\begin{align}
&~ \sum_{j=1}^t \gamma_j\bbE\left[\big\langle\bz^{j}-\bz,  \bw^j - \bff(\bx^j)\big\rangle - \big\langle\bx^{j} - \bx, \bu^{j}-\tilde{\nabla}_\bx \hL (\bx^j,\bz^j) \big\rangle\right] \nonumber\\
 \leq &~ \frac{1}{2}\bigg(\bbE\big[\big\|  \bz^1-\bz\big\|^2\big]+ n B^2 +  \big( F^2 + \widehat{P}_\bz^t\big) \sum_{j=1}^t \gamma_j^2\bigg),\label{eq:stochastic_x_z}
\end{align}
where $ \tilde{\nabla}_\bx \hL (\bx^j,\bz^j)  = \bbE\big[\bu^j\mid\hH^j\big]$ and  $\widehat{P}_\bz^t$ is defined in  \eqref{eq:P_hat}. Furthermore, if $\bx$ and $\bz$ are deterministic, then
 \begin{align}\label{eq:stochastic_x_z2}
 \sum_{j=1}^t \gamma_j\bbE\left[\big\langle\bz^{j}-\bz,  \bw^j - \bff(\bx^j)\big\rangle - \big\langle\bx^{j} - \bx, \bu^{j}-\tilde{\nabla}_\bx \hL (\bx^j,\bz^j) \big\rangle\right] = 0.
\end{align}
\end{lemma} 

\subsection{Main convergence results}\label{subsection:conv}

In this subsection, we use the previous established lemmas to show our main convergence results. First, we obtain the following theorem by combining the inequalities in \eqref{eq:primal_x_bbE} and \eqref{eq:sum_u}. 

\begin{theorem} 
\label{thm:converge}
Under Assumptions~\ref{assu:compact}-\ref{assu:kkt} and with the choice of $\bz^1=\0$, for any $\bx\in X$, $\bz\geq \0$, $t\in [K]$,
\begin{align}
 \bbE\big[f_0(\bar{\bx}^t)-f_0(\bx) - \big\langle\bar{\bz}^t, \bff(\bx)\big\rangle +\big\langle\bz, \bff(\bar{\bx}^t)\big\rangle\big]
 \leq \epsilon_\bz+\epsilon_0\bbE[\|\bz\|^2],\label{eq:converge}
\end{align}
where $\bar{\bx}^t = \frac{ \sum_{j=1}^t\sum_{k=j}^t \alpha_k  \beta_1^{k-j}   \bx^j}{\sum_{j=1}^t \sum_{k=j}^t \alpha_k  \beta_1^{k-j}  }$, $\bar{\bz}^t = \frac{  \sum_{j=1}^t\sum_{k=j}^t \alpha_k  \beta_1^{k-j}  \bz^j}{\sum_{j=1}^t \sum_{k=j}^t \alpha_k  \beta_1^{k-j}}$,  and 
\begin{align*}
 \epsilon_\bz =&  \frac{1}{2(1\!-\!\beta_1) \sum_{k=1}^t \alpha_k } \bigg(   n(\theta+1) B^2  +   \frac{ \sqrt{n}(\theta + \frac{\widehat{P}_\bz^t}{\theta}) \sum_{j=1}^t \alpha_j^2}{(1\!-\!\beta_1)^2(1\!-\!\beta_2)^{1/2}}   \\
 &  \hspace{2cm} + \frac{\rho_1F^2 \sum_{j=1}^t \alpha_j^2}{\alpha_1(1\!-\!\beta_1)}  + \big( F^2 + \widehat{P}_\bz^t\big) \sum_{j=1}^t \alpha_j^2\bigg),\\
 \epsilon_0 =&
 \frac{1}{2(1-\beta_{1})  \sum_{k=1}^t \alpha_k} \Big(\frac{\alpha_1}{\rho_1} + 1   \Big).
\end{align*}
\end{theorem} 

\begin{proof} 
From \eqref{eq:primal_x_bbE}, \eqref{eq:sum_u}, \eqref{eq:stochastic_x_z}, and $\bz^1=\0$, we have
\begin{align}
& \frac{n\theta B^2}{2 }  +   \frac{ \sqrt{n}(\theta + \frac{\widehat{P}_\bz^t}{\theta}) \sum_{j=1}^t \alpha_j^2 }{2(1\!-\!\beta_1)^2(1\!-\!\beta_2)^{1/2}}\nonumber\\ 
\overset{\eqref{eq:primal_x_bbE}}{\geq} & (1-\beta_{1}) \sum_{j=1}^t \bbE\big[\big\langle\bx^{j} - \bx, \bu^{j}\big\rangle\big] \sum_{k=j}^t \alpha_k\beta_1^{k-j}\nonumber\\ 
\overset{\eqref{eq:sum_u}, \eqref{eq:stochastic_x_z}}\geq&  (1-\beta_{1}) \sum_{j=1}^t \bbE\Big[f_0(\bx^j)-f_0(\bx) - \big\langle\bz^j, \bff(\bx)\big\rangle  +  \big\langle\bz, \bff(\bx^j)\big\rangle\Big] \sum_{k=j}^t \alpha_k\beta_1^{k-j} \nonumber\\ 
&  -\frac{ \alpha_1\bbE\big[\big\| \bz\big\|^2\big] }{2\rho_1}  - \frac{ \rho_1 F^2\sum_{j=1}^t \alpha_j^2}{2\alpha_1(1\!-\!\beta_1)} 
+ \frac{(1\!-\!\beta_1)\alpha_t}{2\rho_t} \bbE\big[\big\|\bz^{t+1}-\bz\big\|^2\big] \nonumber\\ 
& -\frac{1}{2}\bigg(\bbE\big[\big\| \bz\big\|^2\big]+ n B^2 +  \big( F^2 + \widehat{P}_\bz^t\big) \sum_{j=1}^t \alpha_j^2\bigg)  
, \label{eq:3.16}
\end{align}  
where in the second inequality, we have used \eqref{eq:stochastic_x_z} with $\gamma_j= \sum_{k=j}^t \alpha_k\beta_1^{k-j} (1-\beta_{1})\leq \alpha_j$  by  \eqref{eq:beta_ineq}. % in last inequality.

By the convexity of $\{f_i\}_{i\geq 0}$, we have from the  above inequality  that for any $\bz\geq \0$,
\begin{align*}
& \bbE\big[f_0(\bar{\bx}^t)-f_0(\bx) - \big\langle\bar{\bz}^t, \bff(\bx)\big\rangle + \big\langle\bz, \bff(\bar{\bx}^t)\big\rangle\big]\\
\leq& \bbE\bigg[\sum_{j=1}^t \Big(f_0(\bx^j)-f_0(\bx) - \big\langle\bz^j, \bff(\bx)\big\rangle  +  \big\langle\bz, \bff(\bx^j)\big\rangle\Big)\frac{\sum_{k=j}^t \alpha_k\beta_1^{k-j} }{\sum_{j=1}^t\sum_{k=j}^t \alpha_k  \beta_1^{k-j} } \bigg] \\
=&  \frac{1}{\sum_{j=1}^t\sum_{k=j}^t \alpha_k  \beta_1^{k-j} }  \sum_{j=1}^t \bbE\Big[f_0(\bx^j)-f_0(\bx) - \big\langle\bz^j, \bff(\bx)\big\rangle  +  \big\langle\bz, \bff(\bx^j)\big\rangle\Big] \sum_{k=j}^t \alpha_k\beta_1^{k-j} \\ 
\leq&  \frac{1}{2(1\!-\!\beta_1) \sum_{k=1}^t \alpha_k } \bigg( n(\theta+1) B^2  +  \frac{ \sqrt{n}(\theta + \frac{\widehat{P}_\bz^t}{\theta}) \sum_{j=1}^t \alpha_j^2}{(1\!-\!\beta_1)^2(1\!-\!\beta_2)^{1/2}}  + \frac{\rho_1F^2  \sum_{j=1}^t \alpha_j^2}{\alpha_1(1\!-\!\beta_1)}  \\
&  \hspace{2cm} + (1+\frac{\alpha_1}{\rho_1})\bbE\big[\big\| \bz\big\|^2\big]
 + \big( F^2 + \widehat{P}_\bz^t\big) \sum_{j=1}^t \alpha_j^2\bigg),
\end{align*} 
where in the last inequality, we have  used \eqref{eq:3.16}  for the numerator and the fact 
$\sum_{j=1}^t\sum_{k=j}^t \alpha_k  \beta_1^{k-j}  
 \geq \sum_{j=1}^t \alpha_j$ by  \eqref{eq:beta_ineq} for the denominator.
 Rearranging the above inequality gives the result in \eqref{eq:converge}.  \end{proof} 

In \eqref{eq:converge}, $\epsilon_0$ is a constant. In order to  bound   $\epsilon_\bz$ that depends on $\widehat{P}_\bz^t$ defined in \eqref{eq:P_hat}, we  need to bound  $\bbE[\|\bz^j\|^2]$, for all $ j \in [t]$. This  is shown in the  following theorem. 
\begin{theorem} 
\label{thm:bound_z_t}
Suppose Assumptions~\ref{assu:compact}-\ref{assu:kkt}  hold and $\bz^1=\0$. Assume that for a given positive number $K$, 
 $\sum_{j=1}^K \alpha_j^2\leq \widehat\alpha$, and $\rho_1 <\frac{\alpha_1(1\!-\!\beta_1)^4(1\!-\!\beta_2)^{1/2}\theta}{2\sqrt{n} P \widehat\alpha}$.
Then
$\bbE\big[\big\|\bz^{t}\big\|^2\big] $ for all $t\in [K+1]$ has a uniform bound,  i.e., 
\[
\bbE\big[\big\|\bz^{t}\big\|^2\big] \leq \frac{C_1}{1-C_2},
\]
where 
$C_1 = \frac{2\rho_{1}   }{ \alpha_1(1\!-\!\beta_1)^2}\left( n \theta B^2 +\frac{\sqrt{n}(\theta^2 + Q)\widehat\alpha}{(1\!-\!\beta_1)^2(1\!-\!\beta_2)^{1/2}  \theta}+ \frac{ \rho_1 F^2 \widehat\alpha}{\alpha_1(1\!-\!\beta_1)} \right) + \big(2 +  \frac{2}{(1\!-\!\beta_1)^2} \big)\big\|\bz^*\big\|^2$ and 
$C_2 = \frac{2\rho_1\sqrt{n} P\widehat\alpha}{\alpha_1(1\!-\!\beta_1)^4(1\!-\!\beta_2)^{1/2}\theta} $.
\end{theorem} 
\begin{proof}  
As the first step in the proof of Theorem~\ref{thm:converge}, we obtain from  \eqref{eq:primal_x_bbE}   and  \eqref{eq:sum_u} with $\bx = \bx^*,\bz = \bz^*$ that  
\begin{align*}
& \frac{n\theta B^2}{2}  +  \frac{ \sqrt{n}(\theta + \frac{\widehat{P}_\bz^t}{\theta}) \sum_{j=1}^t \alpha_j^2 }{2(1\!-\!\beta_1)^2(1\!-\!\beta_2)^{1/2}}\overset{\eqref{eq:primal_x_bbE}}\geq (1-\beta_{1}) \sum_{j=1}^t \bbE\big[\big\langle\bx^{j} - \bx^*, \bu^{j}\big\rangle\big] \sum_{k=j}^t \alpha_k\beta_1^{k-j}\\
{\geq}& (1-\beta_{1})  \sum_{j=1}^t \bbE\big[f_0(\bx^j)-f_0(\bx^*) - \big\langle\bz^j, \bff(\bx^*)\big\rangle  +  \big\langle\bz^*, \bff(\bx^j)\big\rangle\big] \sum_{k=j}^t \alpha_k\beta_1^{k-j}
\nonumber\\
&-\frac{ \alpha_1\big\| \bz^*\big\|^2}{2\rho_1} - \frac{ \rho_1 F^2\sum_{j=1}^t \alpha_j^2}{2\alpha_1(1\!-\!\beta_1)}  + \frac{\alpha_t (1\!-\!\beta_1)}{2\rho_t}\bbE\big[\big\|\bz^{t+1}-\bz^*\big\|^2 \big] \nonumber\\
& + (1\!-\!\beta_{1}) \sum_{j=1}^t \bbE\Big[\!-\!\big\langle\bz^{j}\!-\!\bz^*,  \bw^j \!-\! \bff(\bx^j)\big\rangle  \!+\!\big\langle\bx^{j} \!-\! \bx^*, \bu^{j} \!-\! \tilde{\nabla}_\bx \hL (\bx^j,\bz^j) \big\rangle \Big]\! \sum_{k=j}^t \alpha_k\beta_1^{k-j}.
\end{align*} 
Since $\bx^*,\bz^*$  are  deterministic, the last term on the right side of the above inequality vanishes  
by Lemma~\ref{lem:cov}. In addition,  
%inequality \eqref{eq:beta_ineq}, 
we use \eqref{eq:2.4} and have from the above inequality that
\begin{align*}
  &\frac{\alpha_t (1-\beta_{1})}{2\rho_t}\bbE\big[\big\|\bz^{t+1}-\bz^*\big\|^2\big] \\
 \leq &   
   \frac{n\theta B^2}{2 } + \frac{ \sqrt{n}(\theta + \frac{\widehat{P}_\bz^t}{\theta}) \sum_{j=1}^t \alpha_j^2 }{2(1\!-\!\beta_1)^2(1\!-\!\beta_2)^{1/2}}
+\frac{ \alpha_1\big\| \bz^*\big\|^2}{2\rho_1} + \frac{\rho_1F^2\sum_{j=1}^t \alpha_j^2}{2\alpha_1(1\!-\!\beta_1)} .
\end{align*} 
By \eqref{eq:rho_2},  
$ \frac{2\rho_t}{\alpha_t(1\!-\!\beta_1)} \leq \frac{2 \rho_{1}\alpha_t }{\alpha_t \alpha_1(1\!-\!\beta_1)^2} = \frac{ 2\rho_{1} }{ \alpha_1(1\!-\!\beta_1)^2}$. 
Thus we have 
\begin{align*}
& \bbE\big[\big\|\bz^{t+1}-\bz^*\big\|^2\big]  \\
 \leq &   \frac{ \rho_{1}}{ \alpha_1(1\!-\!\beta_1)^2} \Big( 
 n\theta B^2  +   \frac{ \sqrt{n}(\theta + \frac{\widehat{P}_\bz^t}{\theta}) \sum_{j=1}^t \alpha_j^2 }{(1\!-\!\beta_1)^2(1\!-\!\beta_2)^{1/2}}
+\frac{ \alpha_1\big\| \bz^*\big\|^2}{\rho_1} + \frac{\rho_1F^2\sum_{j=1}^t \alpha_j^2}{\alpha_1(1\!-\!\beta_1)} 
\Big).
\end{align*} 
By $\big\|\bz^{t+1}\big\|^2 \leq 2(\big\|\bz^{t+1}-\bz^*\big\|^2+\big\|\bz^*\big\|^2)$, we obtain
\begin{align*}
\bbE\big[\big\|\bz^{t+1}\big\|^2\big]  \leq& \frac{2 \rho_{1}}{ \alpha_1(1\!-\!\beta_1)^2} \Big(n \theta B^2+ \frac{ \sqrt{n}(\theta + \frac{\widehat{P}_\bz^t}{\theta}) \sum_{j=1}^t \alpha_j^2 }{(1\!-\!\beta_1)^2(1\!-\!\beta_2)^{1/2}}+ \frac{\rho_1 F^2\sum_{j=1}^t \alpha_j^2 }{\alpha_1(1\!-\!\beta_1)} \Big) \\
&+\big(2 +  \frac{2}{(1\!-\!\beta_1)^2} \big)\big\|\bz^*\big\|^2\\
 \leq & C_1+C_2 \max_{k\in [t]}  \left(\bbE\big[\big\|\bz^k\big\|^2\big]\right),
\end{align*} 
where in the last inequality, we have used \eqref{eq:P_hat} and $\sum_{j=1}^t \alpha_j^2\leq\sum_{j=1}^K \alpha_j^2\leq \widehat\alpha$ for all $t\in[K]$.
Notice that $C_2<1$ from the selection of $\rho_1$. Below we show $\bbE\big[\big\|\bz^t\big\|^2\big] \leq \frac{C_1}{1-C_2}$ for all $t\in [K+1]$ by induction.
\begin{itemize}
 \item When $t = 1$, $\bz^1 = \0$, thus $\bbE[\|\bz^1\|^2] \leq \frac{C_1}{1-C_2}$ holds trivially.
 \item Assume that $\bbE\big[\big\|\bz^k\big\|^2\big] \leq \frac{C_1}{1-C_2}$ holds for $k\leq t$, then
\[\bbE\big[\big\|\bz^{t+1}\big\|^2\big] \leq C_1 +  C_2 \left(   \max_{k\in [t]}\bbE\big[\big\|\bz^{k}\big\|^2\big] \right)  \leq C_1+C_2\frac{C_1}{1-C_2} = \frac{C_1}{1-C_2}.\]
 \end{itemize}
 Therefore, we complete the induction and obtain the desired result.  
\end{proof} 

When  the conditions in Theorem~\ref{thm:bound_z_t} hold, we  have that for $k\in [K+1]$,
\begin{equation*}%\label{eq:constant_P}
\widehat{P}_\bz^k  = P\cdot\left(\max_{j\in [k]}\bbE\big[\|\bz^j\|^2\big]\right)+Q\leq P \frac{C_1}{1-C_2} +Q\equiv\widehat{P}.
\end{equation*}
Then  $\epsilon_\bz$  in  Theorem~\ref{thm:converge} can be bounded by a  constant as follows:
\begin{align}\label{eq:epsilon_1}
\epsilon_\bz \leq & \frac{1}{2(1\!-\!\beta_1) \sum_{k=1}^t \alpha_k } \bigg(   n(\theta+1) B^2  +   \frac{ \sqrt{n}(\theta + \frac{\widehat{P}}{\theta}) \sum_{j=1}^t \alpha_j^2}{(1\!-\!\beta_1)^2(1\!-\!\beta_2)^{1/2}} \nonumber\\
&  \hspace{2cm} + \frac{\rho_1F^2 \sum_{j=1}^t \alpha_j^2}{\alpha_1(1\!-\!\beta_1)}  + \big( F^2 + \widehat{P}\big) \sum_{j=1}^t \alpha_j^2\bigg)\equiv\epsilon_1.
\end{align}

\begin{remark}
In  Theorems \ref{thm:converge} and \ref{thm:bound_z_t}, we assume $\bz^1=\0$ in the initialization of Algorithm \ref{alg:APriD}, which simplifies the proof. But notice that  we can get the same-order convergence rate  with any $\bz^1\geq\0$ by a similar proof with slightly different $\epsilon_0, \epsilon_\bz, C_1$.  
\end{remark}

With Theorems \ref{thm:converge} and \ref{thm:bound_z_t}, we are ready to show the convergence rate results by the following lemma, whose proof is given in Appendix~\ref{app:3.2}.  
\begin{lemma}\label{lem:3.2} 
Let $\bar{\bx}\in X$ and $\bar{\bz}\geq 0$ be random vectors. If for any $\bx\in X$ and $\bz\geq \0$  that may depend on $(\bar{\bx},\bar{\bz})$, there are two constants  $\epsilon_1,\epsilon_0$ satisfying 
\begin{align}\label{eq:lem3.2}
\bbE \big[f_0(\bar{\bx}) -   f_0(\bx) - \big\langle\bar{\bz}, \bff(\bx)\big\rangle  +  \big\langle\bz, \bff(\bar{\bx})\big\rangle\big] \leq  
 \epsilon_1+\epsilon_0\bbE\big[\big\|\bz\big\|^2\big],
\end{align}
then for any $(\bx^*,\bz^*)$ satisfying KKT conditions,  we have
\begin{align}
 \bbE\big[\mid f_0(\bar{\bx}) - f_0(\bx^*)\mid\big]\leq 2 \epsilon_1 + 9\epsilon_0 \big\|\bz^*\big\|^2, \label{eq:3.2.1}\\
 \bbE\big[\sum_{i=1}^M [f_i(\bar{\bx})]_+ \big]\leq  \epsilon_1+\epsilon_0\big\|1+\bz^*\big\|^2,\label{eq:3.2.2}\\
 \bbE\big[d(\bz^*)-d(\bar{\bz})\big] \leq \frac{3}{2}( \epsilon_1 + 3\epsilon_0 \big\|\bz^*\big\|^2)\label{eq:3.2.3}.
\end{align}
\end{lemma}
 
Below, we specify the parameters and give convergence rate results for two different choices of $\{\alpha_k\}_{k=1}^K$. One is a constant sequence and the other a varying sequence. %which are  commonly used in  the  literature.
 \begin{corollary}[Convergence rate with constant step size]\label{cor:cons_step}
 Given any positive integer $K$, set $\alpha_j = \frac{\alpha }{\sqrt{K}}$ for $j\in [K]$ and $\rho_1 = \frac{\rho}{\sqrt{K}}$ with  $\rho <\frac{ (1\!-\!\beta_1)^4(1\!-\!\beta_2)^{1/2}\theta}{2\alpha\sqrt{n} P }$.
Let $\{(\bx^k,\bz^k)\}$ be the sequence generated from Algorithm \ref{alg:APriD}   with  $\bz^1=\0$, and $\bar{\bx}^K = \frac{ \sum_{j=1}^K (1-\beta_1^{K-j+1}) \bx^j}{\sum_{j=1}^K (1-\beta_1^{K-j+1}) }$ and $\bar{\bz}^K = \frac{ \sum_{j=1}^K (1-\beta_1^{K-j+1})  \bz^j}{\sum_{j=1}^K(1-\beta_1^{K-j+1})}$.  
Define 
\begin{align*}
C_1 =& \frac{2\rho }{ \alpha(1\!-\!\beta_1)^2}\!\left( \! n \theta B^2 \!+\! \frac{\sqrt{n}(\theta^2 + Q) \alpha^2}{(1\!-\!\beta_1)^2(1\!-\!\beta_2)^{1/2}  \theta} \!+\! \frac{ \rho \alpha F^2  }{ (1\!-\!\beta_1)} \! \right) \!+\! \big(2 \!+\!  \frac{2}{(1\!-\!\beta_1)^2} \big)\big\|\bz^*\big\|^2,\\
C_2 =& \frac{2\rho\alpha \sqrt{n} P }{(1\!-\!\beta_1)^4(1\!-\!\beta_2)^{1/2}\theta},
\\
\widehat{P} =& P\frac{C_1}{1-C_2}+Q,\\
\phi=& n(\theta+1) B^2 +  \frac{\sqrt{n}\alpha^2(\theta^2 +   \widehat{P})}{(1\!-\!\beta_1)^2(1\!-\!\beta_2)^{1/2}\theta}   +  \frac{ \alpha \rho F^2  }{(1\!-\!\beta_1)}  +  (F^2 +  \widehat{P})  \alpha^2.
\end{align*}
Then we  have
\begin{align}
 \bbE\big[\mid f_0(\bar{\bx}^K) - f_0(\bx^*)\mid\big]\leq\frac{1}{2(1-\beta_{1})  \alpha\sqrt{K} }  \Big(2\phi + \frac{9(\alpha+\rho)}{\rho} \big\|\bz^*\big\|^2\Big),\label{eq:cor3.1}\\
 \bbE\big[\sum_{i=1}^M [f_i(\bar{\bx}^K)]_+ \big]\leq \frac{1}{2(1-\beta_{1})  \alpha\sqrt{K} } \Big(\phi+\frac{\alpha+\rho}{\rho}\big\|1+\bz^*\big\|^2 \Big),\\
 \bbE\big[d(\bz^*)-d(\bar{\bz}^K)\big] \leq   \frac{3}{4(1-\beta_{1})  \alpha\sqrt{K} }  \Big(\phi +  \frac{3(\alpha+\rho)}{\rho}\big\|\bz^*\big\|^2\Big).\label{eq:cor3.3}
\end{align}
\end{corollary}

\begin{proof} From $\alpha_j = \frac{\alpha}{\sqrt{K}}$ for $j\in [K]$, we have 
$\sum_{j=1}^t\alpha_j^2\leq \sum_{j=1}^K\alpha_j^2 =  \alpha^2$ for all $t\in [K]$.
Hence, the conditions required by Theorem~\ref{thm:bound_z_t} hold, and thus 
$\bbE\big[\big\|\bz^{t+1}\big\|^2\big]\leq \frac{C_1}{1-C_2}$ for $ t \in [K]$ with  the given ${C_1}, C_2$, and  
$\widehat{P}$.

Note  $\sum_{k=1}^K\alpha_k = \alpha\sqrt{K}$ and $\epsilon_\bz\leq\epsilon_1$ in \eqref{eq:epsilon_1}.
By Theorem  \ref{thm:converge},  we have that for any $\bx\in X$ and $\bz\geq \0$,
\begin{align*}
\bbE\big[ f_0(\bar{\bx}^K)-f_0(\bx) -\big\langle\bar{\bz}^K, \bff(\bx)\big\rangle + \big\langle\bz, \bff(\bar{\bx}^K)\big\rangle\big] \leq  \epsilon_1+\epsilon_0\bbE\big[\big\|\bz\big\|^2\big],
\end{align*}
where 
\begin{align*}
\epsilon_1 &=  \frac{1}{2(1\!-\!\beta_1) \alpha\sqrt{K} } \bigg( \! n(\theta+1) B^2 \!+\! \frac{\sqrt{n}(\theta + \frac{ \widehat{P}}{\theta})\alpha^2}{(1\!-\!\beta_1)^2(1\!-\!\beta_2)^{1/2}}  \!+\!  \frac{ \alpha \rho F^2  }{(1\!-\!\beta_1)}  \!+\!  (F^2 + \widehat{P})  \alpha^2  \bigg),\\
 \epsilon_0 &=
 \frac{1}{2(1-\beta_{1}) \alpha\sqrt{K} } \Big(\frac{\alpha}{\rho} + 1   \!\Big).
\end{align*} 
Therefore, by Lemma~\ref{lem:3.2},  we complete the proof.
\end{proof} 

\begin{remark}\label{rm:effect-theta}
The convergence rate results in Corollary~\ref{cor:cons_step} indicate that the parameter $\theta$ can neither be too big or too small.  Let $\alpha$ be fixed and $\rho$ proportional to $\theta$ such that $\rho <\frac{ (1\!-\!\beta_1)^4(1\!-\!\beta_2)^{1/2}\theta}{2\alpha\sqrt{n} P }$. Then $C_2$ is a constant independent of $\theta$, and $C_1$ and $\phi$ are quadratically dependent on $\theta$. However, the second terms in the parenthesis of the right sides of  \eqref{eq:cor3.1}-\eqref{eq:cor3.3} are inversely proportional to $\theta$. Thus the right-hand sides of  \eqref{eq:cor3.1}-\eqref{eq:cor3.3} will approach to infinity if $\theta\to\infty$ or $\theta\to 0$. 
Similarly, $\theta$ can neither be too big or too small in Corollary~\ref{cor:vary_step}, Corollary~\ref{cor:constant_stepL} and Corollary~\ref{cor:vary_stepL}. But notice that if $\bu^k$ in \eqref{eq:u_hat} is uniformly bounded for all $k$, then a too big $\theta$ will not have effect of clipping. Numerically, we observe that the algorithm can still perform well even with a very small $\theta$, even if $\alpha$ and $\rho$ are both relatively large.
\end{remark}

\begin{remark}\label{rem:para-tune} 
Corollary~\ref{cor:cons_step} requires $\rho <\frac{ (1\!-\!\beta_1)^4(1\!-\!\beta_2)^{1/2}\theta}{2\alpha\sqrt{n} P }$, or equivalently $\alpha\rho\leq \frac{ (1\!-\!\beta_1)^4(1\!-\!\beta_2)^{1/2}\theta}{2\sqrt{n} P }$.
In some applications, $P$ is unknown, and even it can be estimated, strictly following the setting may give too conservative stepsizes. In the experiments, we do not follow the condition strictly, but instead we tune 
$\alpha$ and $\rho$. 
In subsection~\ref{example:param}, we test Algorithm \ref{alg:APriD} with different combinations of $(\alpha, \theta, \rho)$, and it turns out that the algorithm can perform well in a wide range of values.
\end{remark}

\begin{corollary}[Convergence rate with varying step size] \label{cor:vary_step}
For all $j \in [K]$, set $\alpha_j = \frac{\alpha}{\sqrt{j+1}\log(j+1)}$. In addition, choose $\rho_1 = \frac{\rho}{\sqrt{2}\log(2)}$ with  $\rho <\frac{ (1\!-\!\beta_1)^4(1\!-\!\beta_2)^{1/2}\theta}{5\alpha\sqrt{n} P }$.
Let $\{(\bx^k,\bz^k)\}$ be the sequence generated from Algorithm \ref{alg:APriD}  with $\bz^1=\0$, 
and $\bar{\bx}^K = \frac{ \sum_{j=1}^K\sum_{k=j}^K \alpha_k  \beta_1^{k-j}   \bx^j}{\sum_{j=1}^K \sum_{k=j}^K \alpha_k  \beta_1^{k-j}  }$, and $\bar{\bz}^K= \frac{  \sum_{j=1}^K\sum_{k=j}^K \alpha_k  \beta_1^{k-j}  \bz^j}{\sum_{j=1}^K \sum_{k=j}^K \alpha_k  \beta_1^{k-j}}$, 
Define
\begin{align*}
C_1 =& \frac{\rho }{ \alpha(1\!-\!\beta_1)^2} 
\! \left( \! 2n \theta B^2 \!+\! \frac{5\sqrt{n} \alpha^2(\theta^2 \!+\! Q) }{(1\!-\!\beta_1)^2(1\!-\!\beta_2)^{1/2}  \theta} \!+\! \frac{ 5 \alpha\rho  F^2 }{ (1\!-\!\beta_1)} \!\right) \!+\! \big(2 \!+\! \frac{2}{(1\!-\!\beta_1)^2} \big)\big\|\bz^*\big\|^2,\\
C_2 =& \frac{5\alpha \rho\sqrt{n} P }{(1\!-\!\beta_1)^4(1-\beta_2)^{1/2}\theta},\\
\widehat{P} =& P\frac{C_1}{1-C_2}+Q,\\
\phi =& n(\theta+1) B^2 +  \frac{2.5\alpha^2\sqrt{n}(\theta^2 + \widehat{P})}{(1\!-\!\beta_1)^2(1\!-\!\beta_2)^{1/2}\theta}  +  \frac{ 2.5\alpha \rho F^2  }{ (1\!-\!\beta_1)}  +  2.5\alpha^2(F^2+ \widehat{P}).
\end{align*}
Then we  have
\begin{align*}
 \bbE\big[\mid f_0(\bar\bx^K) - f_0(\bx^*)\mid \big]\leq 
 \frac{\log(1+K)}{4(1-\beta_{1}) \alpha (\sqrt{K+2}-\sqrt{2})}  \Big(2\phi + \frac{9(\alpha+\rho)}{\rho} \big\|\bz^*\big\|^2\Big),
 \\
 \bbE\big[\sum_{i=1}^M [f_i(\bar{\bx}^K)]_+ \big]\leq  \frac{\log(1+K)}{4(1-\beta_{1}) \alpha (\sqrt{K+2}-\sqrt{2})}  \Big(\phi+\frac{\alpha+\rho}{\rho}\big\|1+\bz^*\big\|^2 \Big),\\
 \bbE\big[d(\bz^*)-d(\bar{\bz}^K)\big] \leq \frac{3\log(1+K)}{8(1-\beta_{1}) \alpha (\sqrt{K+2}-\sqrt{2})}  \Big(\phi +  \frac{3(\alpha+\rho)}{\rho}\big\|\bz^*\big\|^2\Big).
\end{align*}
\end{corollary}

\begin{proof}
Note that for any $t\in[K]$, 
 \begin{align*}
 &\sum_{j=1}^t \alpha_j^2 \leq \sum_{j=1}^\infty \alpha_j^2  \leq  \alpha^2 \sum_{j=1}^\infty \frac{1}{(j+1)(\log(j+1))^2}\\
\leq&  \alpha^2\Big( \frac{1}{2(\log2)^2} + \int_{1}^\infty  \frac{1}{(x+1)(\log(x+1))^2}dx\Big)\\
=& \alpha^2 \Big(\frac{1}{2(\log2)^2}+\frac{1}{\log2}\Big) \leq 2.5 \alpha^2.
 \end{align*}
Hence, by the choice of $\rho_1$, the conditions required by Theorem~\ref{thm:bound_z_t} hold, and thus  $\bbE\big[\big\|\bz^{t+1}\big\|^2\big] \leq \frac{C_1}{1-C_2}$ for all $t\in[K]$ with the given ${C_1}, C_2$, and $\widehat{P}$. Note  $\epsilon_\bz\leq\epsilon_\1$ in \eqref{eq:epsilon_1} and $\sum_{k=1}^K\alpha_k$ satisfies
\[
\sum_{k=1}^K \alpha_k  = \sum_{k=1}^K   \frac{\alpha}{\sqrt{k+1}\log(k+1)} \geq \frac{\alpha}{\log(1+K)} \int_{k=1}^{K+1}   \frac{1}{\sqrt{k+1} } = \frac{2\alpha(\sqrt{K+2}-\sqrt{2})}{\log(1+K)}.
\]
By Theorem  \ref{thm:converge},  we have  for any $\bx\in X$, $\bz\geq \0$,  
\begin{align*}
\bbE\big[ f_0(\bar{\bx}^K)-f_0(\bx) -\big\langle\bar{\bz}^K, \bff(\bx)\big\rangle + \big\langle\bz, \bff(\bar{\bx}^K)\big\rangle\big] \leq  \epsilon_1+\epsilon_0\bbE\big[\big\|\bz\big\|^2\big],
\end{align*}
where 
\begin{align*}
\epsilon_1 &=  \frac{\log(1+K)}{4(1-\beta_{1}) \alpha (\sqrt{K+2}-\sqrt{2})} \bigg(n(\theta+1) B^2 +  \frac{2.5\alpha^2\sqrt{n}(\theta + \frac{\widehat{P}}{\theta})}{(1\!-\!\beta_1)^2(1\!-\!\beta_2)^{1/2}} \\
& \hspace{2cm}+  \frac{ 2.5\alpha \rho F^2  }{ (1\!-\!\beta_1)} +  2.5\alpha^2(F^2 + \widehat{P})  \bigg),\\
\epsilon_0 &=
 \frac{\log(1+K)}{4(1-\beta_{1}) \alpha (\sqrt{K+2}-\sqrt{2})} \Big(\frac{\alpha}{\rho} + 1   \Big).
\end{align*} 
By Lemma~\ref{lem:3.2}, we obtain the desired results and complete the proof.
\end{proof} 
 
\section{Extension  to Stochastic  Minimax Problem}
In this section, we modify the APriD method in Algorithm~\ref{alg:APriD} to solve a stochastic convex-concave minimax problem in the form of
\begin{equation}
\min_{\bx\in X} \, \max_{\bz\in Z}\, \{\hL(\bx,\bz) = \bbE_{\xi}[L(\bx,\bz;\xi)]\}. \label{eq:problemL}
\end{equation}
Similar to existing works (e.g., \cite{nemirovski2009robust, juditsky2011solving, chen2014optimal, zhao2019optimal}), we assume that $X \subseteq \bbR^n $  and $Z  \subseteq \bbR^m $ in \eqref{eq:problemL} are compact convex sets, $\xi$ is a random vector, and $\hL(\bx,\bz)$ is convex in $\bx\in X$ and concave in $\bz\in Z$. 
By the minimax theorem \cite{neumann1928theorie}, the strong duality holds for \eqref{eq:problemL}, namely, 
\begin{equation}\label{eq:str-dual}
\min_{\bx\in  X} \max_{\bz\in Z}\, \hL(\bx,\bz)=  \max_{\bz\in Z}\,\min_{\bx\in X}\, \hL(\bx,\bz).
\end{equation}

In Algorithm~\ref{alg:APriD}, we perform an adaptive SGM update  to  the primal variable $\bx$ in \eqref{eq:m_update}-\eqref{eq:x_update}, but  a   vanilla SGM update to the dual variable $\bz$. This is because the Lagrangian function of \eqref{eq:problem} has a simple linear dependence on the dual variable. As demonstrated in \cite{xu2020primal}, an adaptive primal-dual SGM performs similarly as well as its non-adaptive counterpart for a linearly constrained problem.  
In \eqref{eq:problemL}, $\hL$ can have complex dependence on  both $\bx$ and $\bz$. Hence, we modify Algorithm~\ref{alg:APriD} to solve \eqref{eq:problemL} by performing adaptive updates to both $\bx$ and  $\bz$. The modified algorithm is named as APriAD, and its pseudocode is given in Algorithm~\ref{alg:APriAD}. Similar to Algorithm~\ref{alg:APriD}, we assume a stochastic first-order oracle, which can return an unbiased stochastic subgradient of $\hL$ at any inquiry point $(\bx,\bz)$.

\begin{algorithm}[htbp]
\caption{Adaptive primal-dual stochastic gradient (APriAD) method for \eqref{eq:problemL}}
\label{alg:APriAD}
    \begin{algorithmic}[1]
    \State \textbf{Initialization:} choose  $\bx^1\in X, \bz^1 \in Z, \bm^0 = [\bm^0_\bx, \bm^0_\bz]  = \0, \bv^0 = [\bv^0_\bx, \bv^0_\bz] = \0, \widehat{\bv}^0 = [\widehat{\bv}^0_\bx, \widehat{\bv}^0_\bz] = \0$;
    \State \textbf{Parameter setting:}  set the maximum number $K$ of iterations;  choose $\beta_1, \beta_2 \in (0,1)$, $\theta>0$, non-increasing step sizes $\{\alpha_k\}_{k= 1}^K$ and $\{\rho_k\}_{k= 1}^K$;
      \ForAll {$k = 1,2,...,K$}
      \State Call the oracle to return a stochastic subgradient $\bg^k = (\bu^k,\bw^k)$ of $\hL$  at  $(\bx^k,\bz^k)$.
       \State Update the primal variables $\bx$ and dual variables $\bz$ by 
              \begin{align} 
               \bm^k   &= \beta_{1}\bm^{k-1}  + (1-\beta_{1})\bg^k,  \label{eq:Lm_update} \\
              \widehat{\bg}^k& = ( \widehat{\bu}^k, \widehat{\bw}^k)=\bigg(\frac{\bu^k }{\max\{1, \frac{\|\bu^k\|}{\theta}\}}, \frac{\bw^k}{\max\{1, \frac{\|\bw^k\|}{\theta}\}}\bigg),\\
               \bv^k &= \beta_2 \bv^{k-1}  + (1-\beta_2)(\widehat{\bg}^k)^2,   \label{eq:Lv_update1}\\ 
               \widehat{\bv}^k  & = \max\{\widehat{\bv}^{k-1}, \bv^k\},\label{eq:Lv_update2} \\ 
              \bx^{k+1}  &= \proj_{X,(\widehat{\bv}^k_\bx)^{1/2}}(\bx^k  - \alpha_k \bm^k_\bx/(\widehat{\bv}^k_\bx)^{1/2}),\label{eq:Lx_update} \\
              \bz^{k+1} &= \proj_{Z,(\widehat{\bv}^k_\bz)^{1/2}}(\bz^k + \rho_k \bm^k_\bz/(\widehat{\bv}^k_\bz)^{1/2}). \label{eq:Lz_update} 
              \end{align}  
              \EndFor  
    \end{algorithmic}
\end{algorithm}   

Below we analyze the convergence of Algorithm~\ref{alg:APriAD}.
Throughout our analysis in this section, we  make the following two assumptions.
\begin{assumption}\label{assu:compactL}
$X$ and $Z$ are both compact convex sets, i.e., there exist  constants $B_\bx,B_\bz$ such that 
\[
\|\bx_1-\bx_2\|_\infty \leq B_\bx,\forall\,\bx_1,\bx_2\in X;  \quad  \|\bz_1-\bz_2\|_\infty \leq B_\bz, \forall\, \bz_1,\bz_2\in Z. 
\]
\end{assumption} 

\begin{assumption}\label{assu:boundL}
The stochastic subgradient $(\bu^k,\bw^k)$ of $\hL(\bx^k,\bz^k)$ is unbiased and bounded for all $k\in[K]$, i.e., there are constants $M_\bx$ and $M_\bz$ such that for any $k\in[K]$,
\[\bbE\big[\bu^k\mid \hH^k\big] =  \tilde{\nabla}_\bx \hL (\bx^k,\bz^k) \in \partial_\bx \hL(\bx^k,\bz^k),  \quad  \bbE\big[\|\bu^k\|^2\big] \leq M_\bx^2,  \]
\[ \bbE\big[\bw^k\mid \hH^k\big] = \tilde{\nabla}_\bz \hL (\bx^k,\bz^k)\in \partial_\bz \hL(\bx^k,\bz^k),  \quad  \bbE\big[\|\bw^k\|^2\big] \leq M_\bz^2.\]
\end{assumption}

Different from Assumption~\ref{assu:bound}, here we assume uniform bounds on $\bbE\big[\|\bu^k\|^2\big]$ and $ \bbE\big[\|\bw^k\|^2\big]$ as $X$ and $Z$ are both compact.
By the assumption, we have that for any $k\in[K]$,
\begin{equation*}
\max_{j\in [k]}  \bbE\big[\|\bu^j\|^2\big] \leq M_\bx^2, \quad \max_{j\in [k]}  \bbE\big[\|\bw^j\|^2\big] \leq M_\bz^2.
\end{equation*}
Similarly  to Lemma~\ref{lem:bound_v}, we have the bounds  
 \begin{align*} 
   &\bbE\big[ \big\|(\widehat{\bv}^k_\bx)^{1/2}\big\|_1\big] \leq n\theta, \quad  \bbE\big[ \big\|\bm^{k}_\bx\big\|_{(\widehat{\bv}^k_\bx)^{-{1/2}}}^2\big] \leq   \frac{\sqrt{n}(\theta^2 +  M_\bx^2)}{(1-\beta_2)^{1/2}\theta}, \\
    &\bbE\big[ \big\|(\widehat{\bv}^k_\bz)^{1/2}\big\|_1\big] \leq m\theta, \quad  \bbE \big[\big\|\bm^{k}_\bz\big\|_{(\widehat{\bv}^k_\bz)^{-{1/2}}}^2\big] \leq \frac{\sqrt{m}(\theta^2 + M_\bz^2)}{(1-\beta_2)^{1/2}\theta}.
\end{align*}  

By the convexity-concavity of $\hL(\bx,\bz)$, we can upper bound $ \hL (\bx^j,\bz) -\hL (\bx,\bz^j)$ as follows.
\begin{lemma} \label{lem:L_L}
For any $j \in[K]$, $\bx\in X, \bz\in Z$, we have
\begin{align} 
   \hL (\bx^j,\bz) -\hL (\bx,\bz^j)   
\leq & \big\langle\bx^{j} \!-\! \bx, \bu^{j} \big\rangle - \big\langle\bz^{j} \!-\! \bz, \bw^{j} \big\rangle + \big\langle\bz^{j} \!-\! \bz, \bw^{j}-\tilde{\nabla}_\bz\hL (\bx^j,\bz^j) \big\rangle \nonumber \\
& -  \big\langle\bx^{j} \!-\! \bx, \bu^{j}-\tilde{\nabla}_\bx \hL (\bx^j,\bz^j) \big\rangle,\label{eq:L_L}
\end{align}
where $\tilde{\nabla}_\bx \hL (\bx^j,\bz^j) = \bbE\big[\bu^j\mid\hH^j\big] \in \partial_\bx \hL(\bx^j,\bz^j)$ and $ \tilde{\nabla}_\bz \hL (\bx^j,\bz^j) = \bbE\big[\bw^j\mid\hH^j\big] \in \partial_\bz \hL(\bx^j,\bz^j)$.
\end{lemma}
\begin{proof}
For any $j\in[K]$, we have 
\begin{align}
\big\langle\bx^{j} - \bx, \bu^{j} \big\rangle &=  \big\langle\bx^{j} - \bx, \tilde{\nabla}_\bx \hL (\bx^j,\bz^j) \big\rangle+\big\langle\bx^{j} - \bx, \bu^{j}-\tilde{\nabla}_\bx \hL (\bx^j,\bz^j) \big\rangle, \label{eq:12_1}\\
\big\langle\bz^{j} - \bz, \bw^{j} \big\rangle &=  \big\langle\bz^{j} - \bz, \tilde{\nabla}_\bz \hL (\bx^j,\bz^j) \big\rangle+\big\langle\bz^{j} - \bz, \bw^{j}-\tilde{\nabla}_\bz \hL (\bx^j,\bz^j) \big\rangle.\label{eq:12_2}
\end{align}
Because $\hL (\bx,\bz)$ is convex in $\bx$ and concave in $\bz$, we have
\begin{align*}
\big\langle \bx^j-\bx, \tilde{\nabla}_\bx \hL (\bx^j,\bz^j) \big\rangle &\geq \hL (\bx^j,\bz^j) -\hL (\bx,\bz^j),\\
\big\langle \bz^j-\bz, \tilde{\nabla}_\bz \hL (\bx^j,\bz^j) \big\rangle &\leq \hL (\bx^j,\bz^j)-\hL (\bx^j,\bz).
\end{align*}
Negating the second one  of the above  two inequalities and adding to the first one give  
\begin{align*}
& \big\langle \bx^j-\bx, \tilde{\nabla}_\bx \hL (\bx^j,\bz^j) \big\rangle-\big\langle \bz^j-\bz, \tilde{\nabla}_\bz \hL (\bx^j,\bz^j) \big\rangle \geq\hL (\bx^j,\bz) -\hL (\bx,\bz^j).
\end{align*}
Replacing the left  two terms of the  above inequality  by  \eqref{eq:12_1} and \eqref{eq:12_2}, we obtain the desired result by rearranging terms.  
\end{proof} 

For the first two terms in the right side of \eqref{eq:L_L}, we have the following lemma by the same arguments as those in the proof of  Lemma~\ref{lem:primal_x}.
\begin{lemma}   \label{lem:sum_uLbbE}
For any $t\in[K], \bx\in X, \bz\in Z$,
\begin{align*} 
  (1-\beta_{1}) \sum_{j=1}^t \bbE\big[\big\langle\bx^{j} - \bx, \bu^{j} \big\rangle\big] \sum_{k=j}^t \alpha_k  \beta_1^{k-j}
&\leq \frac{n \theta B^2_\bx}{2} +  
 \frac{\sqrt{n}(\theta^2 +   M_\bx^2) \sum_{j=1}^t  \alpha_j^2}{2(1\!-\!\beta_1)^2(1\!-\!\beta_2)^{1/2}\theta},\\
(1-\beta_{1}) \sum_{j=1}^t \bbE\big[\big\langle\bz^{j} - \bz, \bw^{j} \big\rangle\big] \sum_{k=j}^t \rho_k  \beta_1^{k-j}
&\geq -\frac{m \theta  B^2_\bz}{2}    - 
 \frac{\sqrt{m}(\theta^2 +   M_\bz^2) \sum_{j=1}^t  \rho_j^2}{2(1\!-\!\beta_1)^2(1\!-\!\beta_2)^{1/2}\theta}.   
\end{align*}
\end{lemma}
The second inequality above holds reversely  because of the concavity of $\hL$ about $\bz$.

For the last two terms in the right side of \eqref{eq:L_L},  by essentially the same arguments as those in the proof of Lemma~\ref{lem:cov}, we have the following lemma.
\begin{lemma}\label{lem:covL} For any deterministic or stochastic vector $(\bx,\bz)$ with $\bx\in X$ and $\bz\in Z$, it holds for any positive number sequence $\{\gamma_j\}_{j= 1}^K$ and $t\in[K]$ that 
\begin{align*}
&\sum_{j=1}^t \gamma_j\bbE\left[\big\langle\bz^{j}-\bz,  \bw^j -  \tilde{\nabla}_\bz \hL (\bx^j,\bz^j)\big\rangle - \big\langle\bx^{j} - \bx, \bu^{j}-\tilde{\nabla}_\bx \hL (\bx^j,\bz^j) \big\rangle\right] \nonumber \\
 \leq &  \frac{1}{2}\Big( nB_\bx^2+ m B_\bz^2 +   ( M_\bx^2 +M_\bz^2 ) \sum_{j=1}^t \gamma_j^2\Big). 
\end{align*} 
where $ \tilde{\nabla}_\bx \hL (\bx^j,\bz^j) = \bbE\big[\bu^j\mid \hH^j\big]  $ and $ \tilde{\nabla}_\bz \hL (\bx^j,\bz^j) = \bbE\big[\bw^j\mid \hH^j\big]$. 
\end{lemma} 

Note  that the bounds in  Lemmas~\ref{lem:sum_uLbbE} and \ref{lem:covL}  do not depend on $\bx$ and $\bz$.  
Below, we use the established lemmas to show the main convergence rate result of Algorithm~\ref{alg:APriAD}. Following \cite{nemirovski2009robust}, we adopt the expected primal-dual gap to measure 
the quality of a solution $(\bar{\bx},\bar{\bz})\in X\times Z$ for \eqref{eq:str-dual}, namely, 
\[
\epsilon(\bar{\bx},\bar{\bz}) = \bbE\Big[ \max_{\bz\in Z}\, \hL(\bar{\bx},\bz)  - \min_{\bx\in X} \,\hL(\bx,\bar{\bz})\Big].
\]

\begin{theorem}%[Convergence]
\label{thm:convergeL}
Under Assumptions~\ref{assu:compactL} and \ref{assu:boundL}, let $\{(\bx^j,\bz^j)\}$ be generated from Algorithm~\ref{alg:APriAD}. Suppose $\alpha_k = c \rho_k, \forall\, k$ for some constant $c>0$.
Then for any  $t\in[K]$, it holds
\begin{align*}
& \epsilon(\bar{\bx}^t,\bar{\bz}^t)   \leq   
\frac{1}{2(1-\beta_{1}) \sum_{j=1}^t \alpha_j} \Bigg( \theta(n B_{\mathbf{x}}^2 + c m B_{\mathbf{z}}^2) + (n B_{\mathbf{x}}^2 + m B_{\mathbf{z}}^2)  
\\
&\hspace{0.2cm} +  (M_\bx^2+M_\bz^2) \sum_{j=1}^t \alpha_j^2+  \frac{ \sqrt{n}(\theta^2 + M_\bx^2)\sum_{j=1}^t  \alpha_j^2+ c\sqrt{m}(\theta^2 + M_\bz^2)\sum_{j=1}^t  \rho_j^2}{(1\!-\!\beta_1)^2(1\!-\!\beta_2)^{1/2}\theta}  \Bigg).
 \end{align*}
 where $\bar{\bx}^t = \frac{ \sum_{j=1}^t\sum_{k=j}^t \alpha_k  \beta_1^{k-j}   \bx^j}{\sum_{j=1}^t \sum_{k=j}^t \alpha_k  \beta_1^{k-j}  }$, and $\bar{\bz}^t = \frac{  \sum_{j=1}^t\sum_{k=j}^t \alpha_k  \beta_1^{k-j}  \bz^j}{\sum_{j=1}^t \sum_{k=j}^t \alpha_k \beta_1^{k-j}}$.
\end{theorem} 

\begin{proof}
By the convexity-concavity of $\hL$ and  Lemma~\ref{lem:L_L} , we have that for any $\bx\in X$ and $\bz\in  Z$, 
\begin{align*}
&  \hL (\bar{\bx}^t,\bz) -\hL (\bx,\bar{\bz}^t)   
\leq \sum_{j=1}^t\left(\left( \hL (\bx^j,\bz) -\hL (\bx,\bz^j) \right)  \frac{  \sum_{k=j}^t\alpha_k  \beta_1^{k-j}}{\sum_{j=1}^t \sum_{k=j}^t \alpha_k  \beta_1^{k-j} }\right)\nonumber\\
\overset{\eqref{eq:L_L}}\leq & 
 \frac{1}{ \sum_{j=1}^t \sum_{k=j}^t \alpha_k  \beta_1^{k-j} }
 \bigg( \sum_{j=1}^t\left( \big\langle\bx^{j} - \bx, \bu^{j} \big\rangle - \big\langle\bz^{j} - \bz, \bw^{j} \big\rangle \right) \sum_{k=j}^t \alpha_k  \beta_1^{k-j}  \\
 & +\sum_{j=1}^t\left( \big\langle\bz^{j} \!-\! \bz, \bw^{j}\!-\!\tilde{\nabla}_\bz\hL (\bx^j,\bz^j) \big\rangle \!-\! \big\langle\bx^{j} \!-\! \bx, \bu^{j}\!-\!\tilde{\nabla}_\bx \hL (\bx^j,\bz^j) \big\rangle \right) \sum_{k=j}^t \alpha_k  \beta_1^{k-j}\bigg).  
\end{align*}
Let $\bx\in \mbox{arg}\min_{\bx\in X} \,\hL (\bx,\bar{\bz}^t)$ and $\bz\in  \mbox{arg}\max_{\bz\in Z}\hL (\bar{\bx}^t,\bz)$, and  take the expectation on the above inequality.
By Lemma  \ref{lem:sum_uLbbE} and Lemma~\ref{lem:covL} with $\gamma_j=\sum_{k=j}^t \alpha_k  \beta_1^{k-j}(1\!-\!\beta_1)\leq \alpha_j$  for  the numerator, and $ \sum_{j=1}^t \sum_{k=j}^t \alpha_k  \beta_1^{k-j} \geq  \sum_{j=1}^t  \alpha_j $ from  \eqref{eq:beta_ineq} for the denominator, we have
\begin{align*}
& \epsilon(\bar{\bx}^t,\bar{\bz}^t)  = \bbE\Big[ \max_{\bz\in Z}\,\hL (\bar{\bx}^t,\bz) -\min_{\bx\in X} \,\hL (\bx,\bar{\bz}^t)\Big] \\
\leq 
& \frac{1}{(1-\beta_{1}) \sum_{j=1}^t \alpha_j } \bigg(\frac{\theta (n B^2_\bx+  c m B^2_\bz)}{2} +\frac{1}{2}\Big(  n B_\bx^2+m B_\bz^2 +  ( M_\bx^2+M_\bz^2) \sum_{j=1}^t \alpha_j^2\Big) \\
&\hspace{2cm}+ \frac{\sqrt{n}(\theta^2 +   M_\bx^2) \sum_{j=1}^t\alpha_j^2+c \sqrt{m}(\theta^2 + M_\bz^2) \sum_{j=1}^t\rho_j^2}{2(1\!-\!\beta_1)^2(1\!-\!\beta_2)^{1/2}\theta}\bigg) .
\end{align*} 
Simplifying the above inequality gives the desired result.
\end{proof} 

Below, we specify the choices of $\{\alpha_k\}_{k=1}^K$  and $\{\rho_k\}_{k=1}^K$  and obtain sublinear convergence of Algorithm~\ref{alg:APriAD}.  Corollary~\ref{cor:constant_stepL} can be proved by essentially the same arguments as those in the proof of  Corollary~\ref{cor:cons_step}, and the proof of Corollary~\ref{cor:vary_stepL} is given in Appendix~\ref{app:vary_stepL}.

 \begin{corollary} [Convergence rate with constant step size]\label{cor:constant_stepL}
Given any positive integer $K$, set $\alpha_j = \frac{\alpha}{\sqrt{K}}$ and $\rho_j = \frac{\rho}{\sqrt{K}}$ for all $j\in [K]$ and some positive constants $\alpha, \rho$. Let $\{(\bx^k,\bz^k)\}$ be the sequence generated from Algorithm~\ref{alg:APriAD}. Let $\bar{\bx}^K = \frac{ \sum_{j=1}^K (1-\beta_1^{K-j+1}) \bx^j}{\sum_{j=1}^K (1-\beta_1^{K-j+1}) }$ and $\bar{\bz}^K = \frac{ \sum_{j=1}^K (1-\beta_1^{K-j+1})  \bz^j}{\sum_{j=1}^K(1-\beta_1^{K-j+1})}$.  
Then  
 \begin{align*}
\epsilon(\bar{\bx}^K,\bar{\bz}^K)  \leq  & 
\frac{1}{2(1-\beta_{1}) \alpha\sqrt{K}} \bigg(\theta(n B_{\mathbf{x}}^2 + \frac{m\alpha}{\rho} B_{\mathbf{z}}^2) + (n B_{\mathbf{x}}^2 + m B_{\mathbf{z}}^2) \\
& \hspace{0.8cm} +  \alpha^2(M_\bx^2+M_\bz^2)  + 
 \frac{ \alpha^2 \sqrt{n}(\theta^2 + M_\bx^2)+ \alpha\rho \sqrt{m}(\theta^2 + M_\bz^2)}{(1\!-\!\beta_1)^2(1\!-\!\beta_2)^{1/2}\theta}
 \bigg).
\end{align*} 
 \end{corollary}

\begin{corollary} [Convergence rate with varying step size]  \label{cor:vary_stepL}   Set $\alpha_j = \frac{\alpha}{\sqrt{j+1}}$  and $\rho_j = \frac{\rho}{\sqrt{j+1}}$ for all $j\in[K]$ and some positive constants $\alpha, \rho$. Let $\{(\bx^k,\bz^k)\}$ be generated from Algorithm \ref{alg:APriAD}. Let $\bar{\bx}^K = \frac{ \sum_{j=1}^K\sum_{k=j}^K \alpha_k  \beta_1^{k-j}   \bx^j}{\sum_{j=1}^K \sum_{k=j}^K \alpha_k  \beta_1^{k-j}  }$ and $\bar{\bz}^K= \frac{  \sum_{j=1}^K\sum_{k=j}^K \alpha_k  \beta_1^{k-j}  \bz^j}{\sum_{j=1}^K \sum_{k=j}^K \alpha_k  \beta_1^{k-j}}$ for $K\geq 2$. Then 
\begin{align*}
\epsilon(\bar{\bx}^K,\bar{\bz}^K)  \leq  & 
\frac{\log(1+K)}{4(1-\beta_{1}) \alpha(\sqrt{K+2}-\sqrt{2})}\bigg(\theta(n B_{\mathbf{x}}^2 + \frac{m\alpha}{\rho} B_{\mathbf{z}}^2) + (n B_{\mathbf{x}}^2 + m B_{\mathbf{z}}^2)  \\
& \hspace{0.8cm} +  \alpha^2(M_\bx^2+M_\bz^2) +  \frac{ \alpha^2 \sqrt{n}(\theta^2 + M_\bx^2)+ \alpha\rho \sqrt{m}(\theta^2 + M_\bz^2)}{(1\!-\!\beta_1)^2(1\!-\!\beta_2)^{1/2}\theta}
 \bigg).
\end{align*}
\end{corollary}

 \section{Numerical Experiments}
In this  section, we  compare the numerical performance of our proposed APriD method in Algorithm~\ref{alg:APriD} to CSA  \cite{lan2020algorithms} and MSA \cite{nemirovski2009robust}. The latter two methods have been reviewed in section~\ref{sec:literature}. In all our experiments,  we take $\beta_1 = 0.9, \beta_2 = 0.99,  \theta = 10$ in APriD;  $s=1$ and  $\eta_k =  \eta = 0.04$  in  CSA. 
In \cite{lan2020algorithms}, the  output of CSA is the  weighted  average of $\bx^t$  over $t\in \hB^k = \{t \in [K] \mid\widehat{G}_t\leq  \eta_t  \}$. 
Note that $\hB^k$ may be empty for a small $k$. Hence, we also computed the weighted average of $\bx^t$ over all $t\in [k]$. We call the result of CSA by the former weighted average as CSA1 and the latter as CSA2.  

In \cite{nemirovski2009robust},  the step sizes of MSA for updating $\bx$ and $\bz$ are both equal to $\gamma_k$, as in  \eqref{eq:MSA}.  In  our experiments, we chose different step sizes for the $\bx$ and $\bz$ updates in order to have better performance. 
More specifically, we did the updates:
\[
 \bx^{k+1}  =  \proj_{X}(\bx^{k} -\alpha_k \bu^k) \quad  \mbox{and}\quad  \bz^{k+1} = \proj_{Z}(\bz^k+\rho_k \bw^k ),
\] 
where $\alpha_k$ and $\rho_k$ are the step sizes. We used the same $\alpha_k$ and $\rho_k$ for MSA and APriD. Hence, MSA can be viewed as a non-adaptive counterpart of APriD. We tried different pairs of $(\alpha_k, \rho_k)$ for MSA. It turned out that the best  pair  for MSA was also the best for APriD. With the step size 
$\gamma_k$ in CSA, we denotes the step sizes in the three algorithms as $(\alpha_k, \rho_k, \gamma_k)$.

Three problems were tested in our experiments. The first one is NPC in a finite-sum form;  the second one is QCQP whose  objective and  constraint are both in an expectation form; the third one is QCQP with scenario approximation (i.e., with many quadratic constraints). For the third problem, we also compared to PDSG-adp in \cite{xu2020primal}, for which the update is given in \eqref{eq:xu_x_update}.  
All experiments were run on MATLAB installed on a MacBook Pro with one 2.9 GHz Dual-Core Intel Core i5 processor and 16 GB memory. 

\subsection{Neyman-Pearson Classification Problem}\label{example:NP}
In this subsection, we compare the algorithms on solving instances of
NPC \eqref{eq:NPproblem}.
We take the linear classifier $h(\bx;\ba) =  \bx^\top \ba$ and the convex surrogate $\varphi(z) = \log(1+e^z)$. This way,  \eqref{eq:NPproblem} reduces to
\begin{align}
\min_{\bx} \,&\, \bbE[\log(1+e^{-\bx^\top  \ba})\mid b=+1], \nonumber \\
 \st  \,&\,\bbE[ \log(1+e^{\bx^\top \ba})\mid b=-1]-c\leq 0.\label{eq:NP1} 
\end{align} 
Given a training data set with $n^+$ positive-class samples $\{\ba^+_i\}_{i=1}^{n^+}$ and $n^-$ negative-class samples $\{\ba^-_i\}_{i=1}^{n^-}$,
we can obtain and solve a scenario approximation of \eqref{eq:NP1}, namely,
\begin{align} 
\min_{\bx}\,&\,  f_0(\bx) \equiv \frac{1}{n^+}\sum_{i=1}^{n^+} \log(1+e^{-\bx^\top  \ba_i^+}), \nonumber \\
 \st \,&\, f_1(\bx) \equiv  \frac{1}{n^-}\sum_{i=1}^{n^-} \log(1+e^{\bx^\top  \ba_i^-})- \widehat{c}\leq 0. \label{eq:NP2}
\end{align}  
According to \cite{rigollet2011neyman}, we set $\widehat{c} = c - \frac{\kappa}{\sqrt{n^-}}$, where $\kappa$ is a positive constant  in order to ensure that the feasible solution of \eqref{eq:NP2} is also feasible for \eqref{eq:NP1} in a given high probability.  

\begin{table}[htbp]
\begin{center}
\begin{minipage}{0.96\textwidth}
\caption{Characteristics of data  sets and algorithm parameters. For the data  sets, $d$ is the number of features  of each sample; $n^+$ and $n^-$ are respectively the numbers of positive and negative samples. $\widehat c$  controls the level of the  false-positive error  in \eqref{eq:NP2}.
$J^0$ and $J^1$  are the batch sizes for evaluating the stochastic gradients of the objective and  the constraint functions by the three methods;  $J^g$ is the batch size for obtaining a stochastic estimation of the constraint function value in CSA. $K$ is the total number  of  iterations. 
%APriD and MSA use constant step size $\frac{\alpha}{\sqrt{K}}$ for the primal update and $\frac{\rho}{\sqrt{K}}$ for the dual update; CSA uses a constant step size $\frac{\gamma}{\sqrt{K}}$.
$(\alpha_k, \rho_k, \gamma_k) = (\frac{\alpha}{\sqrt{K}},   \frac{\rho}{\sqrt{K}},  \frac{\gamma}{\sqrt{K}}),\forall\, k\le K$ are  constant step sizes.
}  \label{tab:data_set}
\begin{tabular}{ccccccc}
\hline
data  set& $d$ & $(n^+, n^-)$ & $\widehat{c}$ &$(J^0,J^1,J^g)$ &  $K$  & $(\alpha,\rho,\gamma)$ \\
\hline
 { \verb|spambase|} \cite{Dua:2019}&57&(1813, 2788)&$-\log(0.7)$ &$(10,10,100)$&$10^5$&
%$\multirow{3}{*}{(10,1,10)}$&$\multirow{3}{*}{1e5} $\\
$ (10,1,10)$\\
 { \verb|madelon|}  \cite{guyon2004result}&500&(1300, 1300)&$-\log(0.6)$&$(10,10,100)$&$10^5$&$ (10,1,10)$  \\
 { \verb|gisette|} \cite{guyon2004result}&2000&(3500, 3500)&$-\log(0.6)$&$(30,30,100)$&$10^5$&$ (10,1,10)$\\
\hline
\end{tabular}
\end{minipage}
\end{center}
\end{table}

We use three data sets. The information of the data sets  and the algorithm parameters are given in Table~\ref{tab:data_set}. 
Before feeding the data sets into the methods, we preprocess the data sets. We first 
normalize the data sets feature-wisely to have mean $0$ and standard deviation $1$ and then scale each sample to have unit 2-norm.   

We apply a deterministic method iALM \cite{xu2019iteration} to compute the ``optimal''  solution $\xopt$.  The selected $\widehat{c}$ makes sure $\xopt$ is feasible.
The results, in terms of iteration and time (in seconds), by all compared methods are shown  in Figure~\ref{fig:NP}.
The left two columns of Figure~\ref{fig:NP} are the objective error at the output solutions.  The right two columns show the value of the constraint at the output solutions.
\begin{figure}[!hbtp]
  \begin{center}   
   \includegraphics[width=\textwidth]{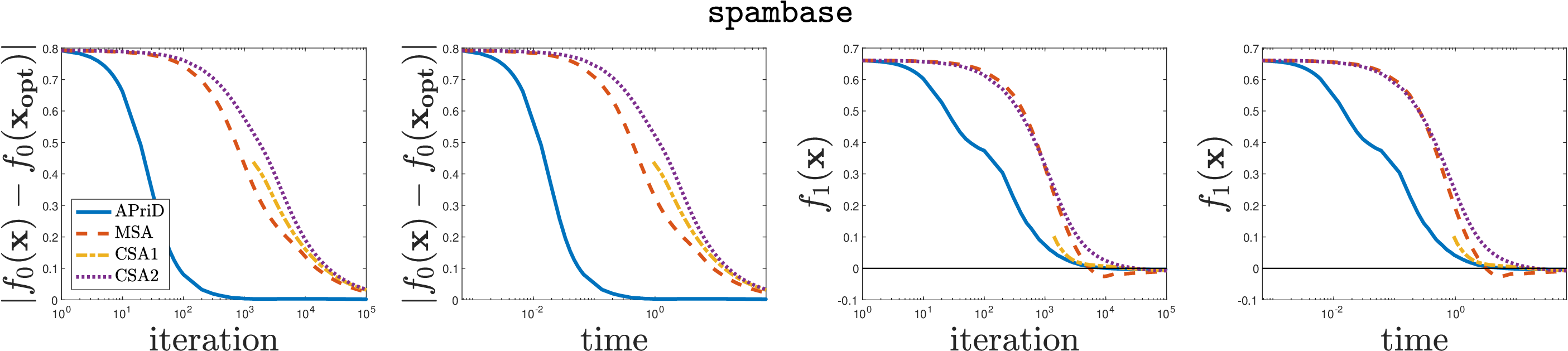}    
    \includegraphics[width=\textwidth]{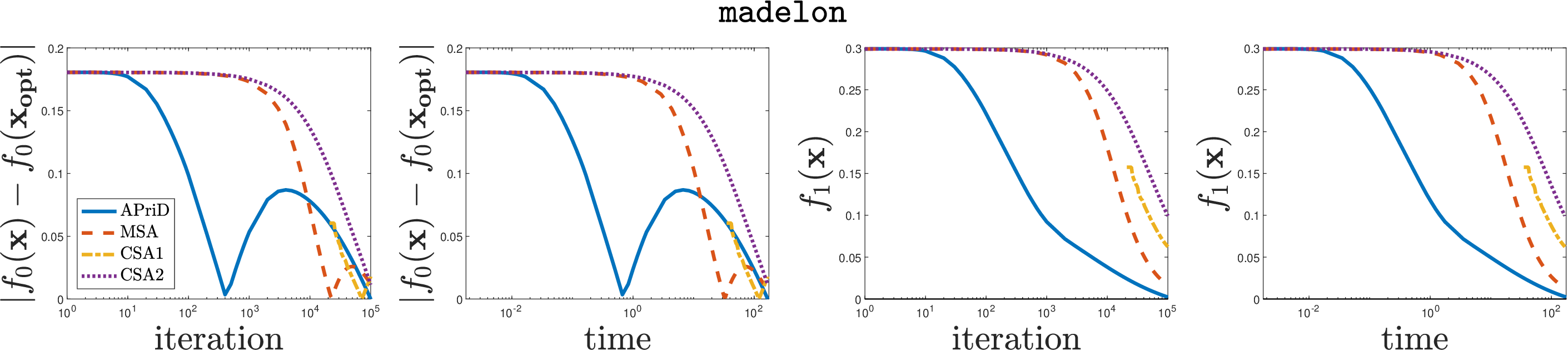}   
   \includegraphics[width=\textwidth]{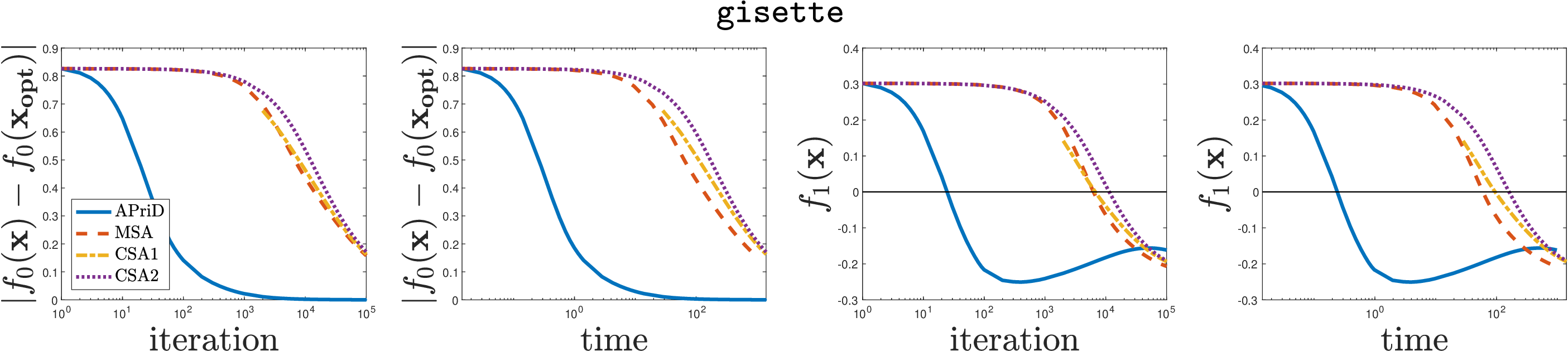}  
 \caption{ The objective error (Left two columns) and constraint value (Right two columns) by three compared methods on solving Example \ref{example:NP} with three data sets:
   Spambase(Top);  
 Madelon(Middle);  Gisette(Bottom). The curve with legend ``0'' represents the zero-error curve.} \label{fig:NP}
      \end{center}
\end{figure}    
We can see from Figure \ref{fig:NP} that APriD performs the best on the data sets \verb|spambase| and \verb|gisette|, in terms of either objective error or feasibility. For the data set \verb|madelon|, APriD is the fastest to achieve feasibility. 
We also note the lack of results for CSA1 at the beginning iterations because $\hB$ is empty. When CSA1 has  results, it is better than CSA2 since the former only takes ergodic mean on \quotes{good} solutions but the later one takes ergodic mean on \quotes{all} solutions.

\subsection{QCQP in Expectation Form }\label{example:expqcqp} 
In  this subsection,  we conduct experiments on the QCQP in an expectation form:  
% \begin{equation} 
% \min_{\bx\in X} f_0(\bx)\equiv \bbE[F(\bx,\xi)] \quad \st\quad f_1(\bx)\equiv\bbE_\xi[G(\bx,\xi)] \leq 0.  \label{eq:expQCQP}
% \end{equation}
\begin{align} 
\min_{\bx\in X} \,&\, f_0(\bx)\equiv \bbE[\frac{1}{2} \|H_{\xi}\bx-\bc_{\xi}\|^2],\nonumber \\ 
\st\,&\, f_1(\bx)\equiv\bbE[\frac{1}{2} \bx^\top Q_\xi \bx+\ba_\xi^\top \bx - b_\xi] \leq 0.  \label{eq:expQCQP}
\end{align}
Here, we set $X = [-10,10]^n$, $\xi$ is a random variable, 
%$F(\bx,\xi) =  \frac{1}{2} \|H_{\xi}\x-\bc_{\xi}\|^2$ and $G(\bx,\xi) = \frac{1}{2} \bx^\top Q_\xi \bx+\ba_\xi^\top \bx - b_\xi$.
$H_\xi\in \bbR^{p\times n}$ and $\bc_\xi\in \bbR^{p}$ are randomly generated, and their components are generated by standard Gaussian distribution and then normalized. The entries of $\ba_\xi\in \bbR^{n}$ are also generated by standard Gaussian distribution and then normalized; $Q_\xi\in \bbR^{n\times n}$ is a randomly generated symmetric positive  semidefinite matrix with unit 2-norm; $b_\xi$ follows a uniform  distribution  on the open interval $(0.1, 1.1)$. 

While running the algorithms, we generate $H_\xi, \bc_\xi,  Q_\xi, \ba_\xi, b_\xi$ based on the above distribution once needed for function  evaluation or gradient evaluation. Hence, the  function value  and  gradient direction are both unbiased estimations. At a weighted iterate $\bx$, we generate another $10^5$ samples to evaluate the objective value and the constraint function value. Also, we obtain the ``optimal'' solution $\xopt$ by using CVX \cite{cvx,gb08} to solve a sample approximation problem with the generated $10^5$ samples.

\begin{figure}[!hbpt]
  \begin{center}  
  \includegraphics[width=1.\textwidth]{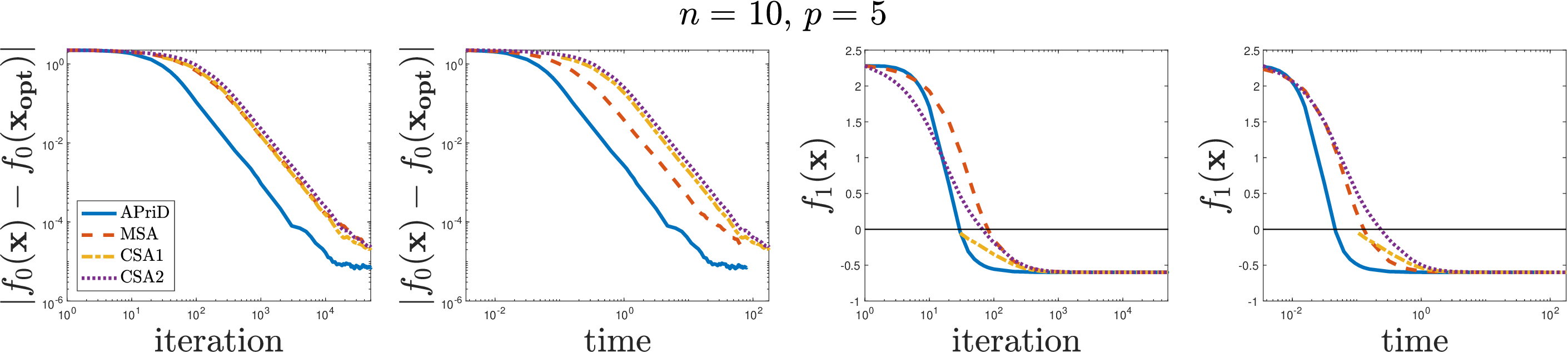}    
  \includegraphics[width=1.\textwidth]{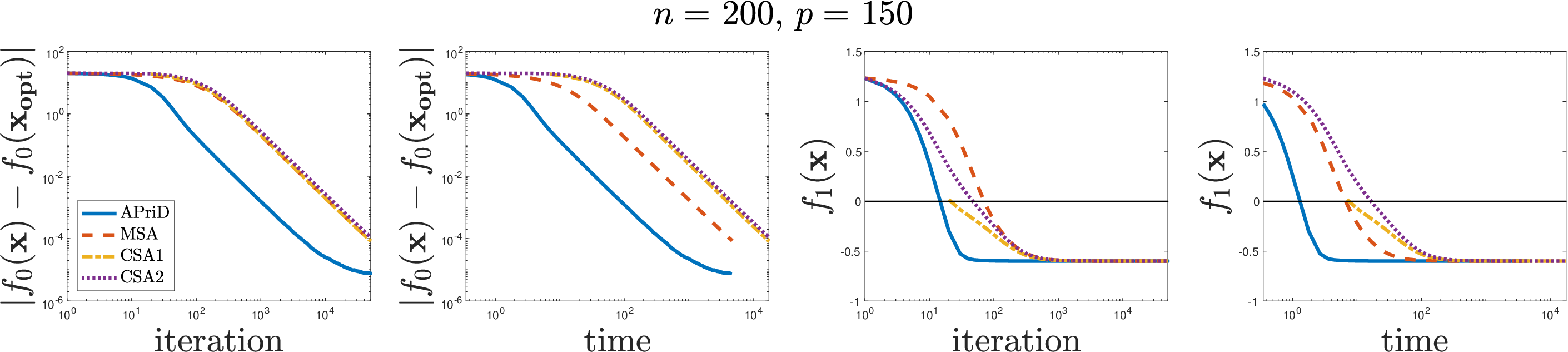}     
 \caption{The objective error (Left two columns) and constraint value (Right two columns) by three methods on solving  QCQP  instances of \eqref{eq:expQCQP}.   
 } \label{fig:exp_qcqp}
      \end{center}
\end{figure}   
We test the compared algorithms on QCQP instances of size $(n,p)=(10,5)$ and $(n,p)=(200,  150)$. In both instances, we run $K = 5\times 10^4$ iterations, and we set batch size $J=10$ for obtaining stochastic gradients of both objective and constraint functions in all the three methods and $J^g=100$ for obtaining a stochastic estimation of the constraint value in CSA. 
Step sizes are set to  $(\alpha_k, \rho_k, \gamma_k) = (\frac{10}{\sqrt{K}},   \frac{10}{\sqrt{K}},  \frac{\sqrt{10}}{\sqrt{K}})$ for all $k$.
The results  in terms of iteration and time (in seconds) are shown in Figure \ref{fig:exp_qcqp}.  
From the results, we see again that APriD significantly outperforms over other two compared methods.  

\subsection{Finite-sum structured QCQP with many constraints}\label{example:scen_qcqp} In this subsection, we  test the algorithms on the QCQP with a finite-sum objective and many constraints: 
\begin{align}\label{eq:qcqp-finite}
\min_{\bx\in X}\,&\, f_0 (\bx) \equiv  \frac{1}{2N}\sum_{i=1}^N \|H_i\bx-\bc_i\|^2,\nonumber\\
\st \,&\, f_j(\bx)\equiv\frac{1}{2} \bx^\top Q_j\bx+\ba_j^\top \bx -  b_j\le 0,\, j\in[M].
\end{align}
We set $X = [-10,10]^n$ in the experiment. $H_i, \bc_i$ for $i\in [N]$ and  $Q_j, \ba_j, b_j$ for $j\in [M]$ are independently generated from  the  same distribution as $H_\xi,\bc_\xi,Q_\xi,\ba_\xi,b_\xi$ in section~\ref{example:expqcqp}.  
Two different-size QCQP instances are generated: one with $(n,p) = (10, 5)$ and the other with $(n,p)=(200, 150)$. We set $N=M=10^4$ in both instances. Besides CSA and MSA, we also compare APriD with the adaptive method in \cite{xu2020primal}, called PDSG-adp. The update of PDSG-adp is shown in \eqref{eq:xu_x_update}. For each update of the compared methods, we randomly select $10$ component functions of the objective and $10$ constraint functions for evaluating stochastic gradients, and for CSA, we randomly pick $100$ constraint functions to obtain a stochastic estimation of $g(\bx)\equiv\sum_{j=1}^M[f_j(\bx)]_+$. 
In both instances, we  run $K=5\times 10^4$ iterations with step sizes $(\alpha_k, \rho_k, \gamma_k) = (\frac{10}{\sqrt{K}}, \frac{\sqrt{10}}{\sqrt{K}},  \frac{10}{\sqrt{K}})$ for all $k$. 
We run PDSG-adp to $K=5\times 10^4$ iterations with $\alpha_k=\frac{20}{\sqrt{K}}$, $\rho_k=\frac{\sqrt{10}}{\sqrt{K}},\forall\, k$ and $\eta=0.1$.   

For the smaller-size instance, we obtain the optimal  solution $\xopt$ by  CVX \cite{cvx,gb08}; for the larger one, we use as an estimated optimal solution  $\bx_{\mathbf{opt}}$ that is feasible and has the smallest objective value among all iterates  from APriD, CSA, MSA and PDSG-adp.  
 We report the objective error, the averaged constraint violation measured by $\frac{1}{M}\sum_{j=1}^M[f_j(\bx)]_+$, and the maximum constraint violation measured by $\max_{j\in [M]}[f_j(\bx)]_+$. The results in terms of iteration and time (in seconds) are shown in Figure  \ref{fig:scen_qcqp}. From the results, we see again that APriD outperforms over the other methods in terms of any of the three measures we use.
\begin{figure}[!hbtp]
  \begin{center}   
   \includegraphics[width=0.9\textwidth]{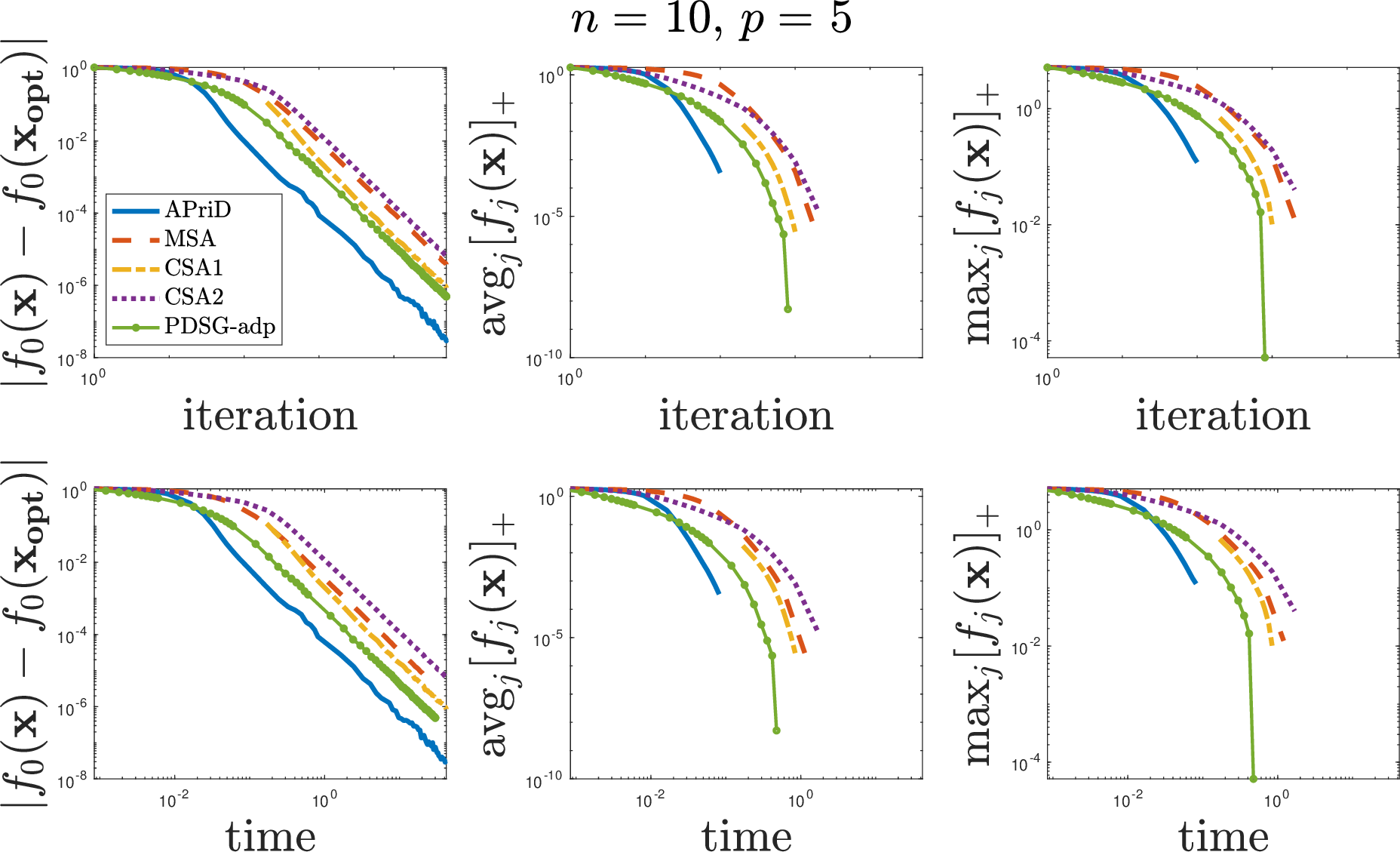}     
  \includegraphics[width=0.9\textwidth]{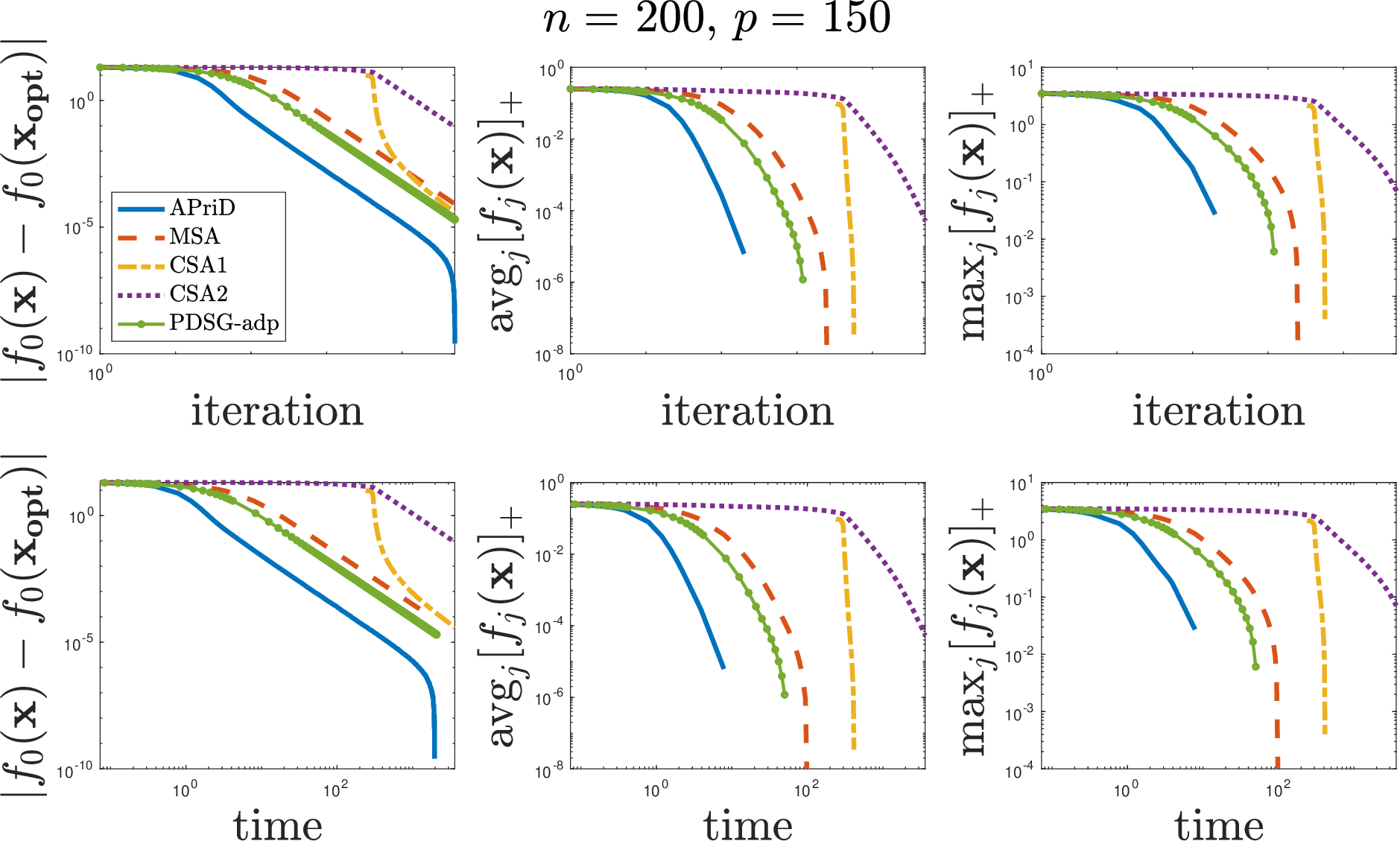}     
 \caption{The objective error (Left), averaged constraint violation (Middle), maximum constraint violation (Right) by  four methods on solving QCQP instances of \eqref{eq:qcqp-finite}.  Rows 1 and 3 are with respect to iteration; rows 2 and 4 are with respect to time (in seconds).
 } \label{fig:scen_qcqp}
      \end{center}
\end{figure}  

\subsection{Computing time comparison}\label{example:time}
In Table \ref{tab:compare_time}, we compare the total running time (in seconds) of all the tested methods in subsections~\ref{example:NP}-\ref{example:scen_qcqp}. From the table, we see that although APriD needs extra computation in \eqref{eq:m_update}-\eqref{eq:x_update}, it takes similar amount of time (and thus has similar per-iteration cost) as MSA in all examples and PDSG-adp for the example in subsection~\ref{example:scen_qcqp}. CSA always takes more time than the other methods because of the extra estimation of $g(\bx^k)$ in each iteration.

\begin{table}[htbp]
\begin{center}
\begin{minipage}{0.85\textwidth}
\caption{Total running time (in seconds) of all the methods that are tested in subsections~\ref{example:NP}-\ref{example:scen_qcqp}. PDSG-adp is only applied to QCQP \eqref{eq:qcqp-finite}.}
\label{tab:compare_time}
\begin{tabular}{cccccc}
\hline
Example & data or size & APriD & MSA & CSA & PDSG-adp \\
\hline
\multirow{3}{*}{NPC \eqref{eq:NP2} }&  \verb|spambase| & 60.3 & 58.8 & 67.1 & -  \\
&\verb|madelon| & 175.5 & 144.6 & 167.4 & - \\
&\verb|gisette| & 983.8 & 954.8 & 1416.6 &- \\  \hline
 \multirow{3}{*}{QCQP \eqref{eq:expQCQP} } &(10,5) & 78.8 & 79.6 & 177.1 &-\\
&(200,150) & 4451.9 & 4722.6 & 17592.7 &- \\ \hline 
 \multirow{3}{*}{QCQP \eqref{eq:qcqp-finite} } &(10,5) & 40.7 & 30.0 & 42.2 & 29.7  \\  
&(200,150) & 1964.2 & 1966.0 & 3638.1 & 2087.2 \\
\hline
\end{tabular}
\end{minipage}
\end{center}
\end{table}

\subsection{Effect of hyper-parameters on algorithm performance}\label{example:param}
In this subsection, we test how the choices of $\alpha, \rho$ (in the constant step size $(\alpha_k, \rho_k)=(\frac{\alpha}{K}, \frac{\rho}{K})$ for all $k\le K$) and $\theta$ affect the performance of APriD. For all tests in this subsection, we apply APriD to solve instances of NPC \eqref{eq:NP2} with the \verb|spambase| and \verb|gisette| data sets and QCQP instances \eqref{eq:expQCQP} with $(n,p)=(10, 5)$ and $(200, 150)$. It turns out that APriD can perform reasonably well for a wide range of values of the hyper-parameters.

\begin{figure}[!hbtp]
  \begin{center}   
   \includegraphics[width=0.47\textwidth]{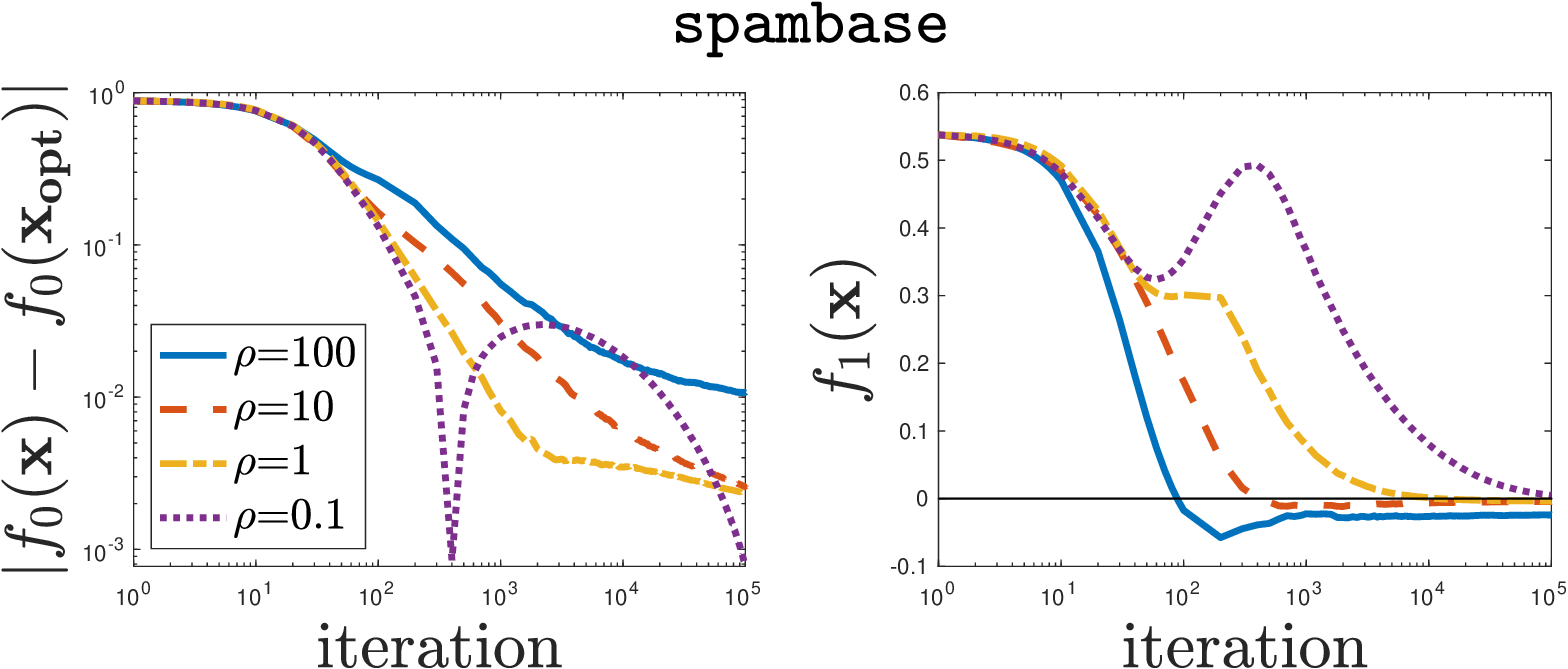}      
  \includegraphics[width=0.47\textwidth]{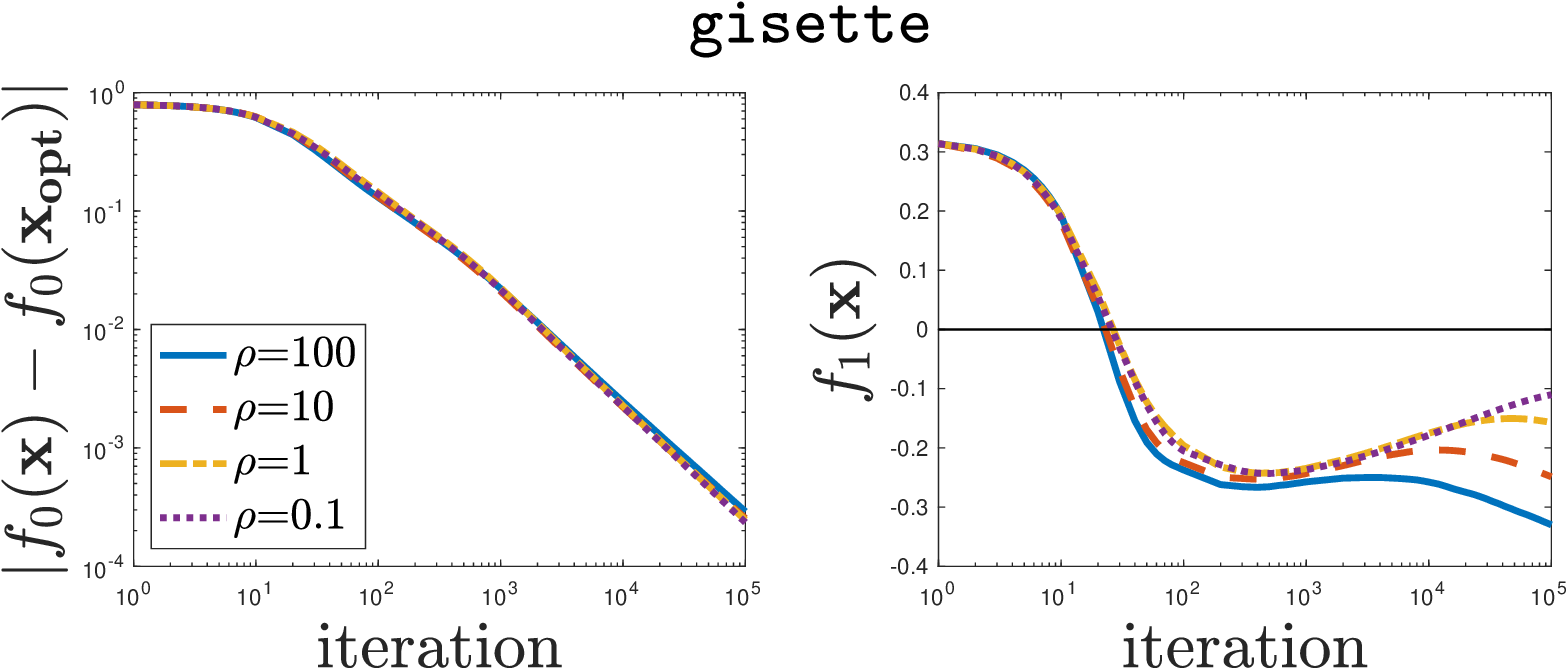}     
   \includegraphics[width=0.47\textwidth]{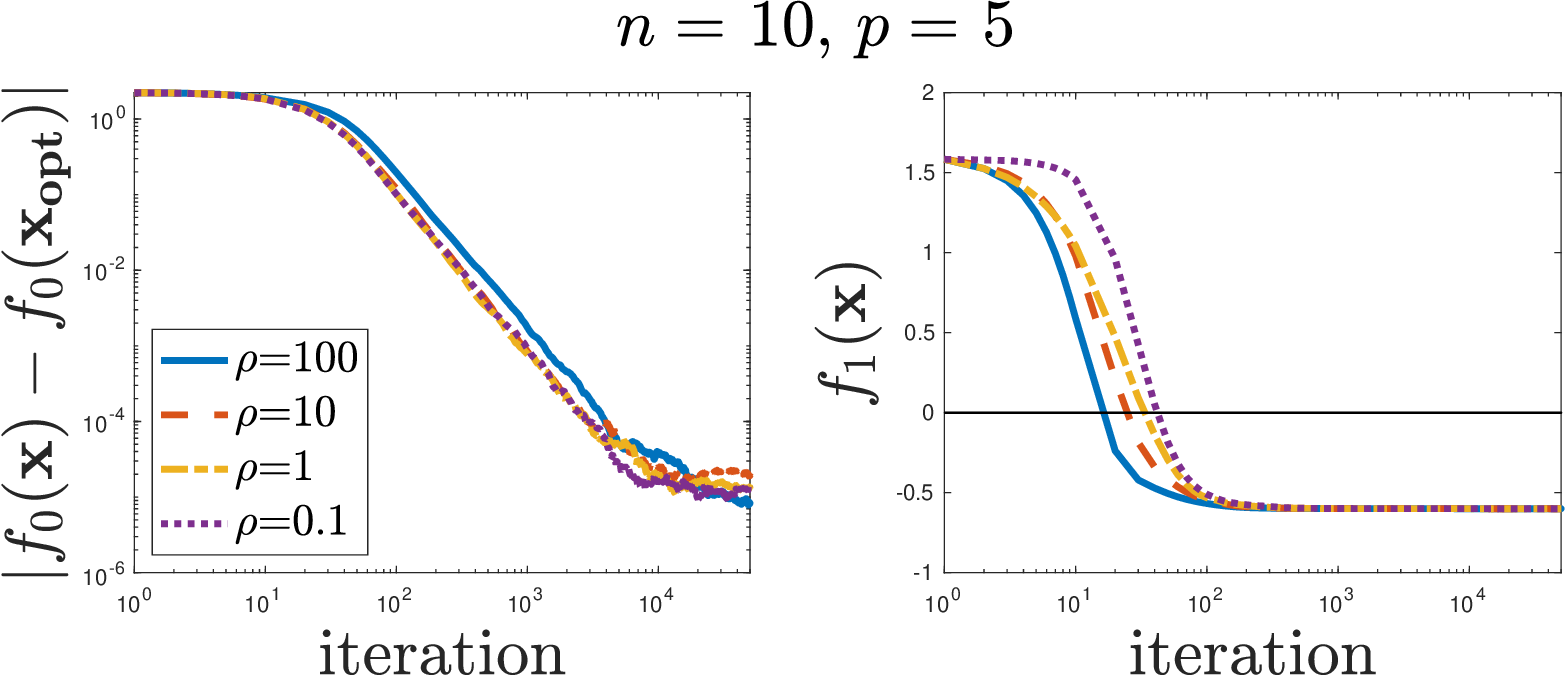}      
  \includegraphics[width=0.47\textwidth]{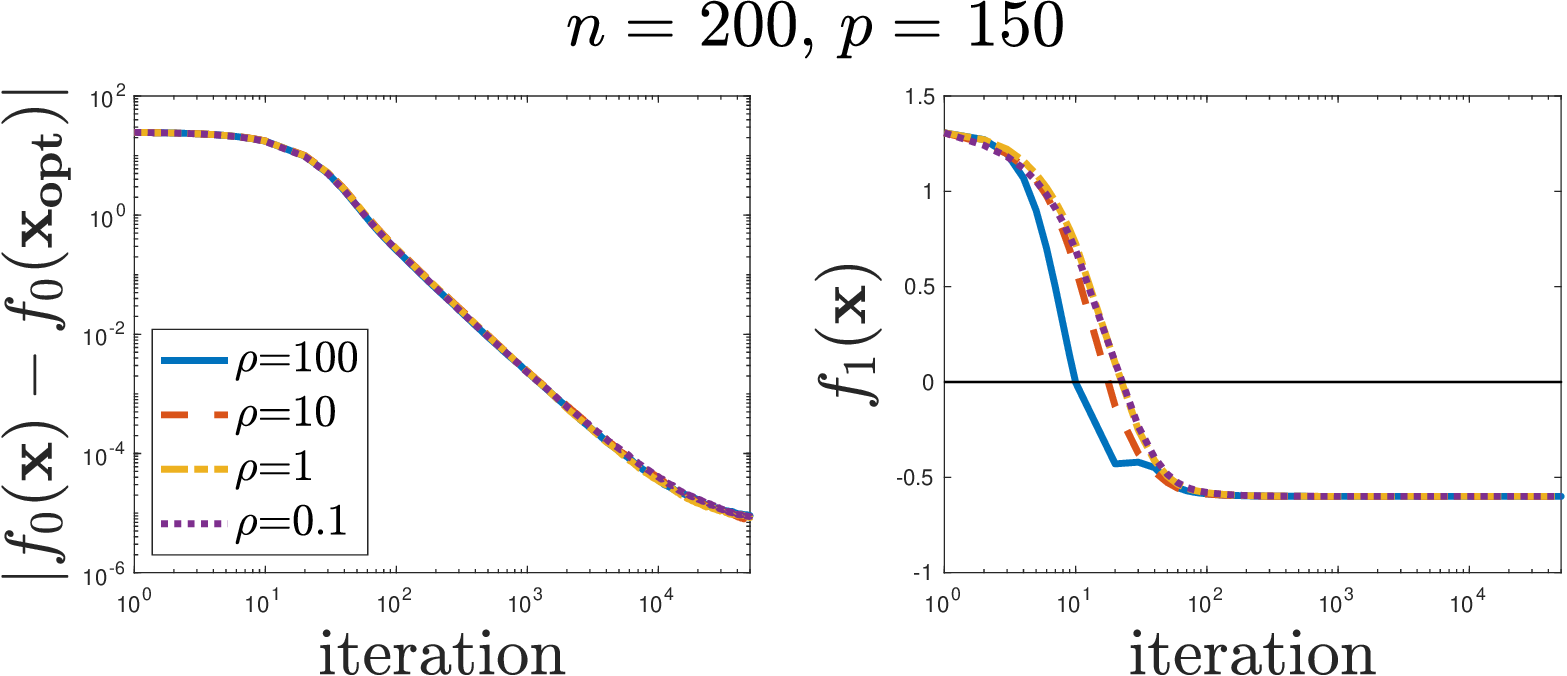}   
 \caption{The objective error and constraint violation  by APriD with $\theta=10, \alpha=10$ and different values of $\rho$  on solving instances of NPC \eqref{eq:NP2} with data sets spambase (TopLeft) and gisette (TopRight) and on solving instances of QCQP \eqref{eq:expQCQP} with $(n,p)=(10, 5)$ (BottomLeft)  and $(n,p)=(200, 150)$ (BottomRight).
  } \label{fig:fixed_theta_alpha}
      \end{center}
\end{figure} 

\begin{figure}[!hbtp]
  \begin{center}   
   \includegraphics[width=0.47\textwidth]{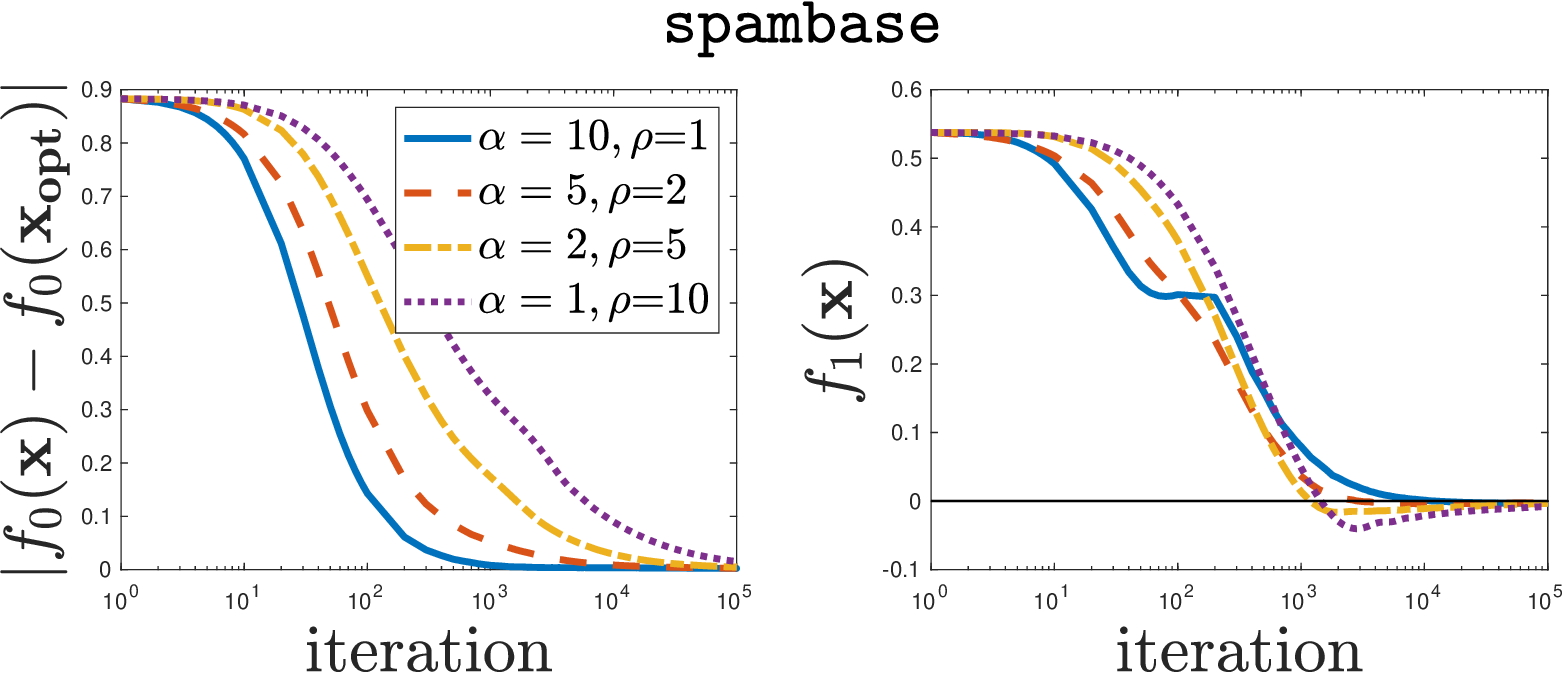}      
  \includegraphics[width=0.47\textwidth]{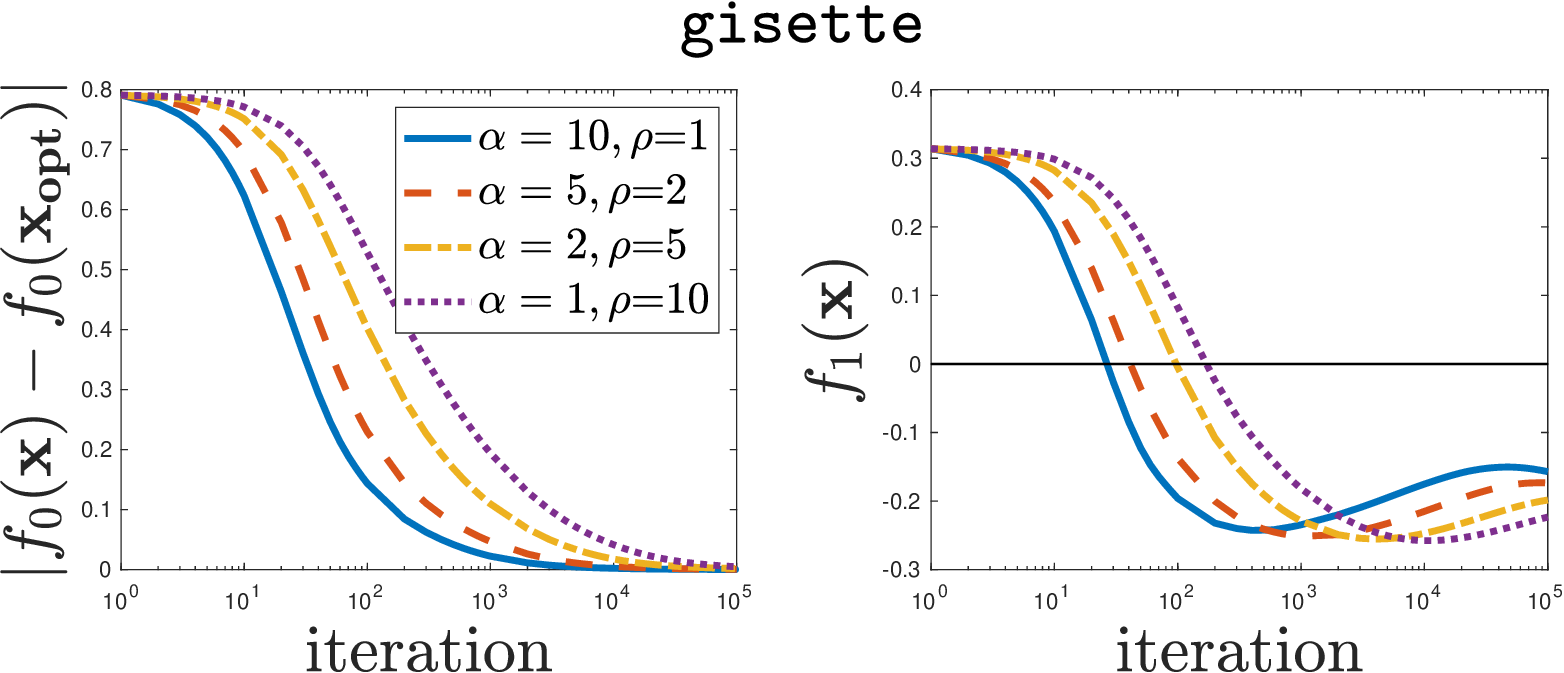}     
   \includegraphics[width=0.47\textwidth]{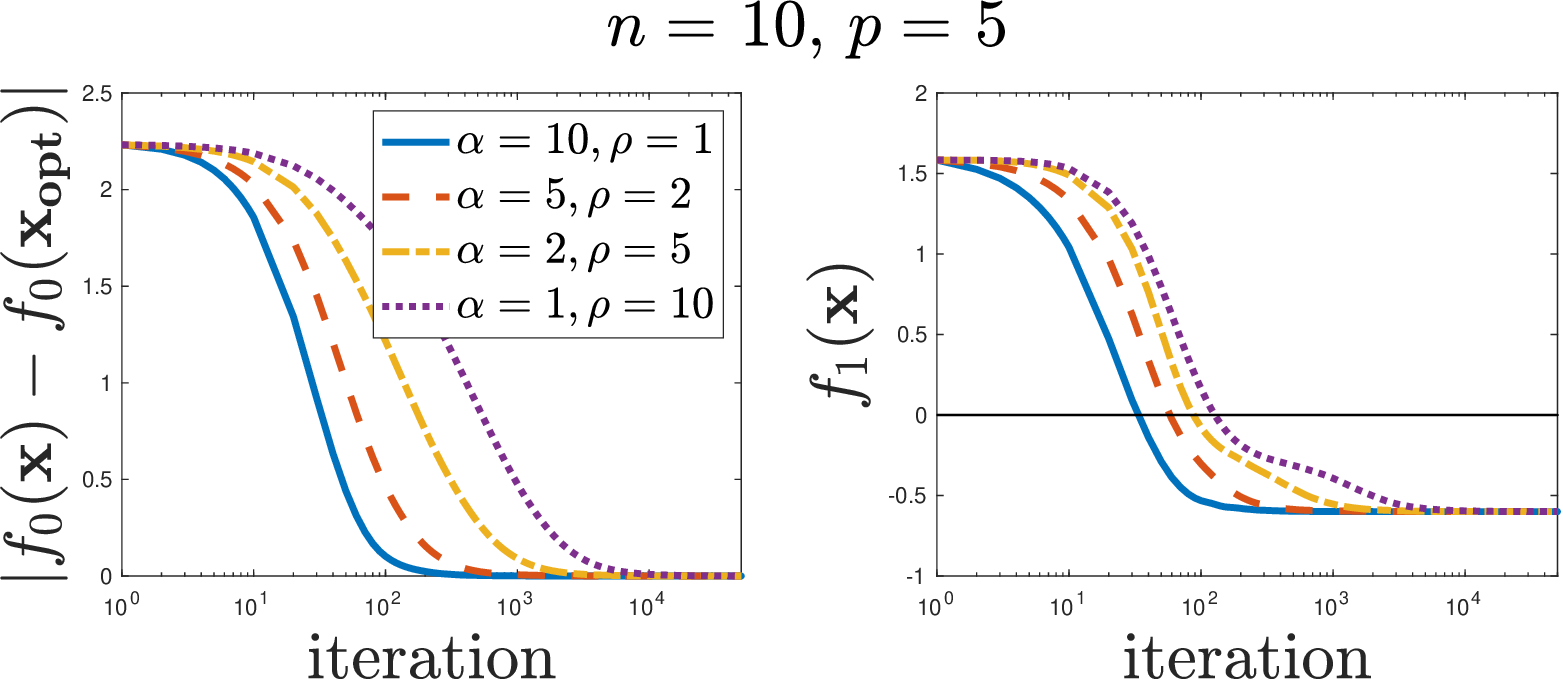}      
  \includegraphics[width=0.47\textwidth]{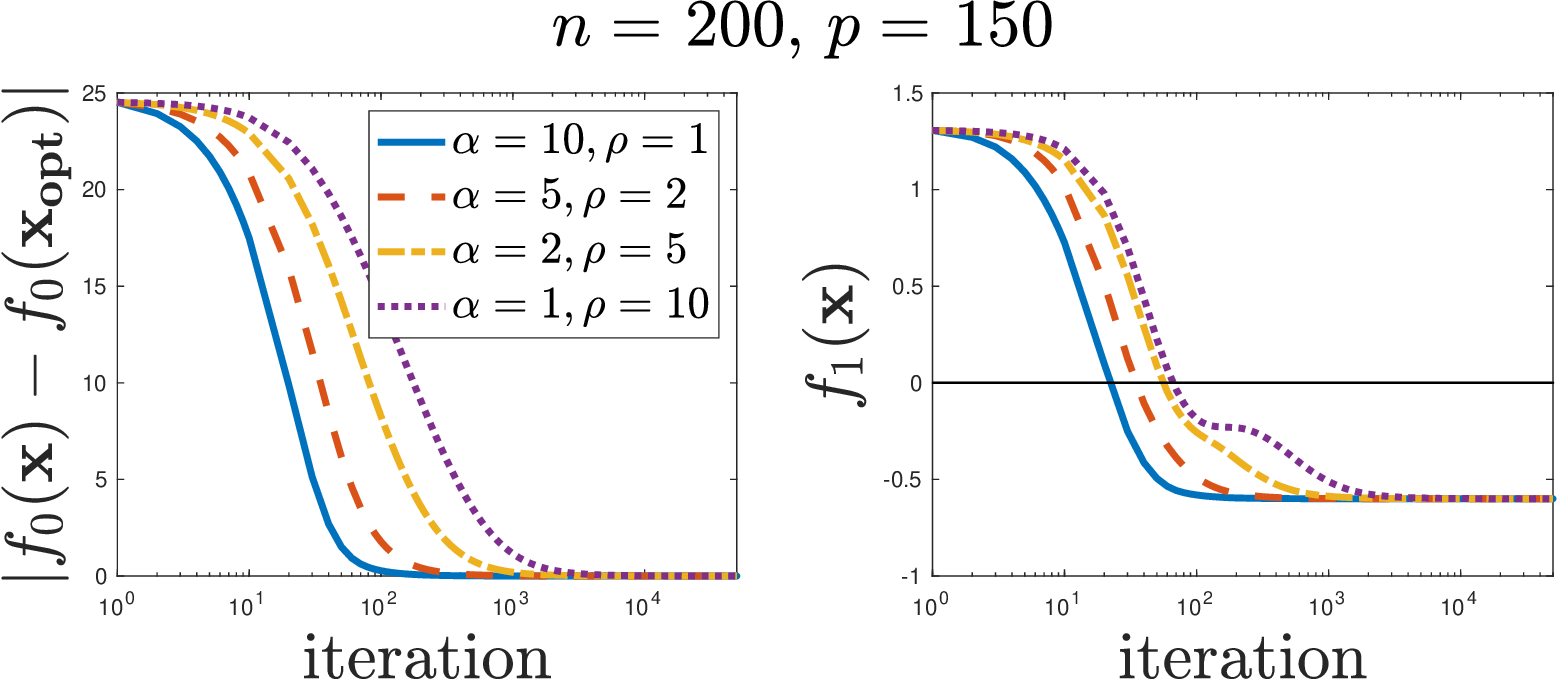}   
 \caption{The objective error and constraint violation by APriD with different pairs of $(\alpha,\rho)$  and a fixed $\theta=10$, on solving instances of NPC \eqref{eq:NP2} with data sets spambase (TopLeft) and gisette (TopRight) and on solving instances of QCQP  \eqref{eq:expQCQP} with $(n,p)=(10, 5)$  (BottomLeft)  and $(n,p)=(200, 150)$ (BottomRight).
 } \label{fig:fixed_theta}
      \end{center}
\end{figure}   

First, we test the effect of $\rho$ by fixing $\theta=10$ and $\alpha=10$. 
From the results in Figure \ref{fig:fixed_theta_alpha}, we see that the algorithm performs similarly well with different values of $\rho$. The difference is most obvious for the instance of  NPC on the \verb|spambase| data set. Zooming in details of the curves, we can observe that with the biggest $\rho=100$, the objective error decreases slowest but the constraint value decreases fastest, and the convergence behavior with the smallest $\rho=0.1$ is exactly the opposite. 

Second, we fix $\theta=10$ and test the effect of $(\alpha, \rho)$. For simplicity, we fix the product $\alpha\rho=10$ and test the algorithm with different pairs of $(\alpha, \rho)$.  The results are shown in Figure \ref{fig:fixed_theta}. It turns out that the algorithm with a larger $\alpha$ tends to converge faster in terms of both the objective error and the constraint violation. Nevertheless, the influence by the choice of $(\alpha, \rho)$ is not severe, and the algorithm with all four different pairs of $(\alpha, \rho)$ can perform reasonably well.

Finally, we test the effect of $\theta$ by fixing $\alpha=10$ and $\rho=10$. We vary $\theta\in\{10^{-4}, 10^{-3}, 10^{-2}, 10^{-1}, 1, 10\}$. The results are shown in Figure~\ref{fig:vary_theta}. For the instances of NPC \eqref{eq:NP2}, there is almost no difference for $\theta\in\{0.1, 1, 10\}$. This is probably because these values of $\theta$ do not trigger the clipping. For the \verb|spambase| data set, the best results appear to be given by $\theta=10^{-2}$ and $\theta=10^{-3}$. The algorithm can still perform well with $\theta=10^{-4}$, but the corresponding feasibility curve is less smooth. Similar observations are made to the \verb|gisette| data set. For QCQP instances of \eqref{eq:expQCQP},  the performance of the algorithm is more affected by the value of $\theta$. The algorithm performs better as $\theta$ decreases from 10 to $10^{-2}$, in terms of both suboptimality and infeasibility. However, for the case of $n=200$ and $p=150$, the infeasibility increases rapidly in the beginning iterations. These observations match with our discussion in Remark~\ref{rm:effect-theta}, i.e., the best value of $\theta$ should not be extremely small.    
\begin{figure}[!hbtp]
  \begin{center}   
   \includegraphics[width=0.47\textwidth]{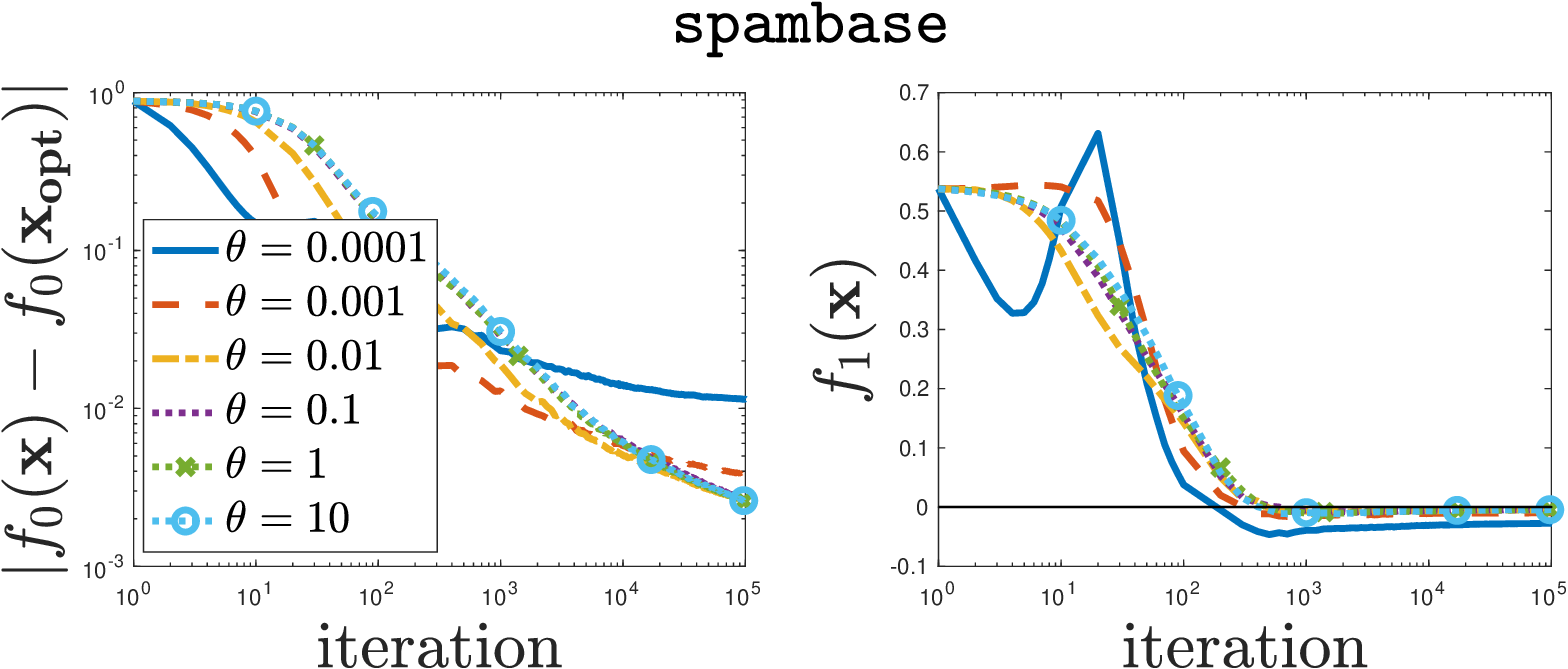}      
  \includegraphics[width=0.47\textwidth]{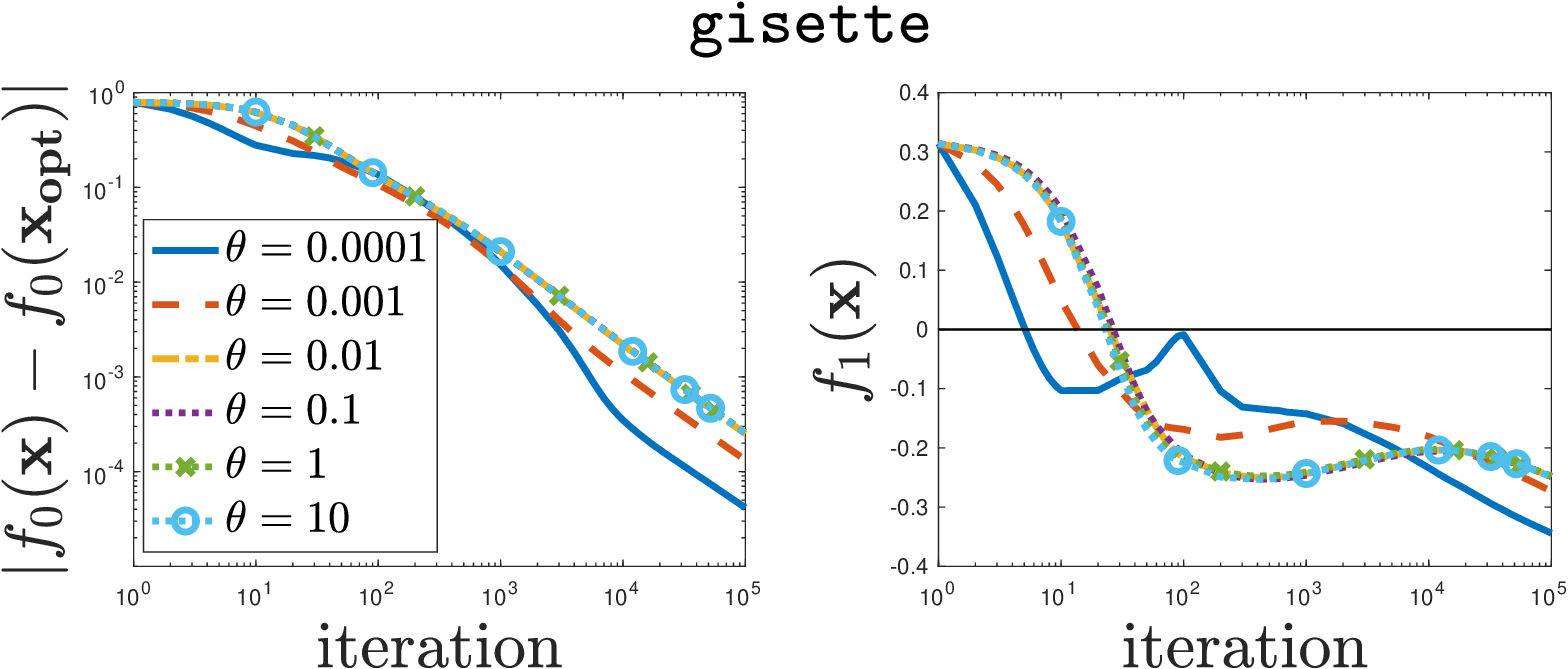}     
   \includegraphics[width=0.47\textwidth]{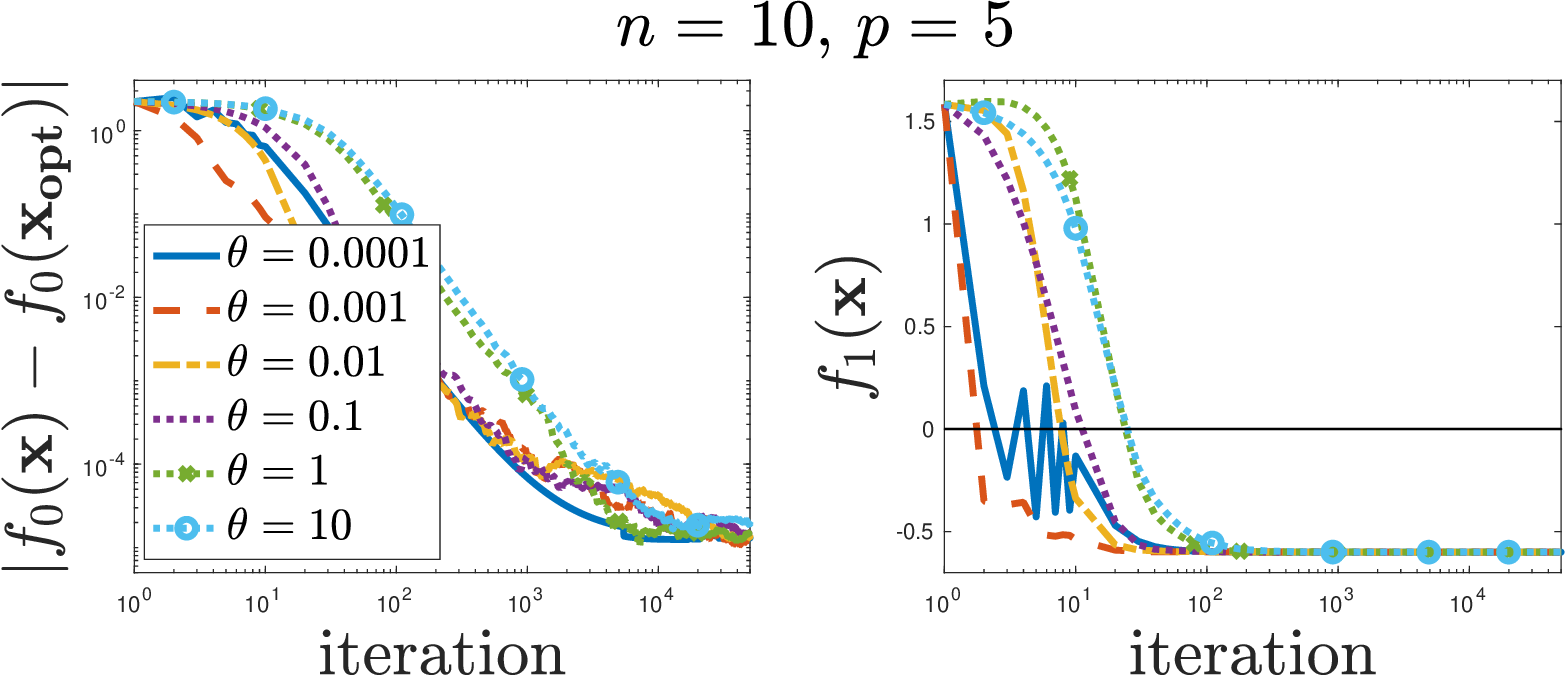}      
  \includegraphics[width=0.47\textwidth]{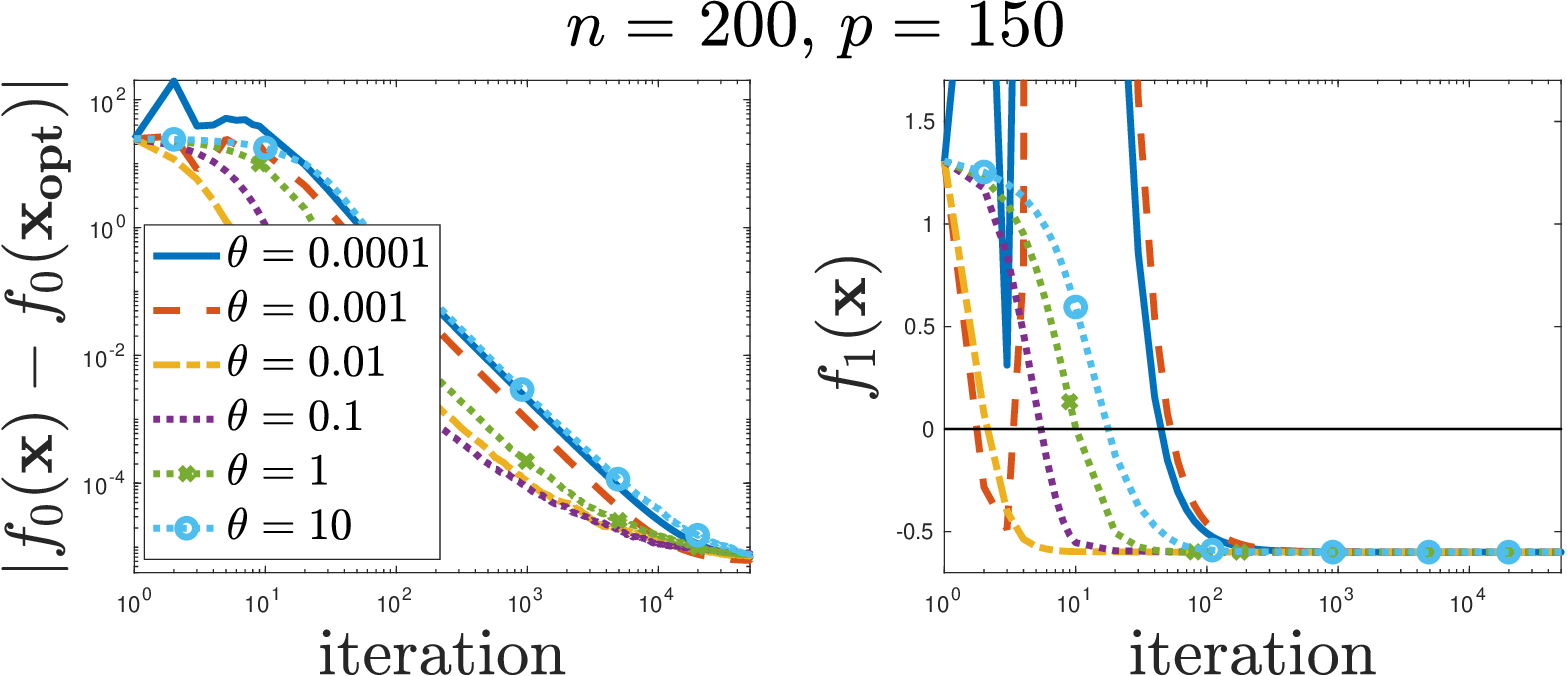}   
 \caption{The objective error and constraint violation by APriD with  $\alpha=10, \rho=10$ and $\theta\in\{10^{-4}, 10^{-3}, 10^{-2}, 10^{-1}, 1, 10\}$ on solving instances of NPC \eqref{eq:NP2} with data sets spambase (TopLeft) and gisette (TopRight) and on solving instances of QCQP \eqref{eq:expQCQP} with $(n,p)=(10, 5)$  (BottomLeft)  and $(n,p)=(200, 150)$ (BottomRight).
 } \label{fig:vary_theta}
      \end{center}
\end{figure}   

 \section{Conclusions}  
 We have proposed an adaptive primal-dual stochastic gradient  method (SGM) for solving expectation-constrained convex stochastic programming.  The method is designed based on the Lagrangian function. At each iteration, it first inquires an unbiased stochastic estimation of the subgradient of the Lagrangian function, and then it performs an adaptive SGM update to the primal variables  and a vanilla SGM step to the dual variables. The method has also been extended with a modification to  solve stochastic convex-concave minimax problems.
For both methods, we have established the convergence rate of $O(1/\sqrt{k})$, where $k$ is the number of inquiries of the stochastic subgradient. Numerical experiments on  three examples demonstrate its superior practical performance over two state-of-the-art methods.  

\section*{Acknowledgements} The authors would like to thank three anonymous reviewers for their valuable comments and suggestions to improve the quality of the paper and also for the careful testing on our codes. The authors are partly supported by the NSF award 2053493 and the RPI-IBM AIRC faculty fund.

\appendix

\section{Proof of Lemma~\ref{lem:rho}}\label{app:rho}
\begin{proof} 
First, consider the case of non-constant primal step size.
By $\eta_1 = \sum_{i=1}^K \alpha_i\beta_1^{i-1}$ and the $\eta$-update in  \eqref{eq:rho_from_eta_update}, we have $\eta_k = \sum_{i=k}^K \alpha_i  \beta_1^{i-k}, k\in[K]$, and thus the $\rho$-update becomes 
\begin{equation*}
   \rho_k = \frac{\rho_{k-1}}{\beta_1 + \frac{\alpha_{k-1} }{\sum_{i=k}^K \alpha_i  \beta_1^{i-k}} }.  
\end{equation*} 
By the above  equation, we have for $2\leq  j\leq K$,
 \begin{align*} 
&\rho_j  = \frac{\rho_{j-1}}{\beta_1 + \frac{\alpha_{j-1} }{\sum_{k=j}^K \alpha_k  \beta_1^{k-j}} }  \leq \frac{\rho_{j-1}}{\beta_1 + \frac{\alpha_{j-1}}{\alpha_{j-1}\sum_{k=j}^K \beta_1^{k-j}} }  
 \leq  \frac{\rho_{j-1}}{\beta_1+ \frac{1}{\sum_{k=j}^\infty \beta_1^{k-j}} } %= \frac{\rho_{j-1}}{\beta_1 \!+\!  (1- \beta_1) } 
 = \rho_{j-1},
\end{align*}
where the inequality follows from the non-increasing monotonicity of $\{\alpha_j\}_{j= 1}^K$  and $\beta_1\in(0,1)$. Hence, $\{\rho_j\}_{j= 1}^K$ is a non-increasing sequence. Using \eqref{eq:rho_from_eta_update}  again, we have for $2\leq j\leq t\leq K$,
 \[\rho_j   
 =  \frac{\rho_{j-1}}{\beta_1 + \frac{\alpha_{j-1} }{\sum_{k=j}^K \alpha_k  \beta_1^{k-j}} } \geq \frac{\rho_{j-1}}{\beta_1 + \frac{\alpha_{j-1} }{\sum_{k=j}^t \alpha_k  \beta_1^{k-j}} }= \frac{\sum_{k=j}^t \alpha_k  \beta_1^{k-j}}{\sum_{k=j-1}^t \alpha_k  \beta_1^{k-(j-1)}}\rho_{j-1},\]
 which clearly implies the inequality in \eqref{eq:rho_1}.
 
 For $j=1$, \eqref{eq:rho_2}  holds because $\beta\in(0,1)$.  To show it holds for  $2\leq j\leq K$, we rewrite $\rho_j$ and obtain
\begin{align*}
&\rho_j = \frac{\rho_{j-1}}{\beta_1 + \frac{\alpha_{j-1} }{\sum_{k=j}^K \alpha_k  \beta_1^{k-j}} } = \frac{\sum_{k=j}^K \alpha_k  \beta_1^{k-j}}{\sum_{k=j-1}^K \alpha_k  \beta_1^{k-(j-1)}} \rho_{j-1}\\
=& \frac{\sum_{k=j}^K \alpha_k  \beta_1^{k-j}}{\sum_{k=j-1}^K \alpha_k  \beta_1^{k-(j-1)}} \times \frac{\sum_{k=j-1}^K \alpha_k  \beta_1^{k-(j-1)}}{\sum_{k=j-2}^K \alpha_k  \beta_1^{k-(j-2)}} \times \cdots \times \frac{\sum_{k=2}^K \alpha_k  \beta_1^{k-2}}{\sum_{k=1}^K \alpha_k  \beta_1^{k-1}} \rho_{1} \\
=&  \frac{\sum_{k=j}^K \alpha_k  \beta_1^{k-j}}{\sum_{k=1}^K \alpha_k  \beta_1^{k-1}} \rho_{1} 
\leq  \frac{  \frac{\alpha_j}{(1\!-\!\beta_1)}}{\sum_{k=1}^K \alpha_k  \beta_1^{k-1} } \rho_{1} \leq  
\frac{  \frac{\alpha_j}{(1\!-\!\beta_1)}}{\alpha_1} \rho_{1}
=\frac{ \rho_{1} \alpha_j}{\alpha_1(1\!-\!\beta_1)},
\end{align*}
where the third equation recursively applies the second equation, and the inequalities hold by the two inequalities in \eqref{eq:beta_ineq}. %\comm{Add more steps to show how \eqref{eq:beta_ineq} is used.}

Now, consider the case of constant primal step size, i.e.,  
%Finally, if
$\alpha_j = \alpha_1$ for all $j\in[K]$. % then  we
We can prove $\eta_j = \frac{\alpha_1}{1-\beta_1}$ for all $j\in[K]$  by the induction, and thus  
\[
\rho_j = \frac{\rho_{j-1}}{\beta_1 + \frac{\alpha_1 }{\alpha_1/(1\!-\!\beta_1)}} =   \frac{\rho_{j-1}}{\beta_1 + (1\!-\!\beta_1)}  =  \rho_{j-1},\forall \, j\ge2,
\]
which completes the proof.
%Thus $ \{\rho_j \}_{j\geq 1} $ are an unvarying constant in  this case.
\end{proof} 

\section{Proof of Lemma~\ref{lem:3.1}}\label{app:3.1}
\begin{proof}
As $(\bx^*,\bz^*)$ satisfies the KKT conditions in  Assumption~\ref{assu:kkt}, there are $\tilde{\nabla}f_i(\bx^*), \forall i\in [M]$ such that 
\[
-\sum_{i=1}^{M} z_i^*\tilde{\nabla}f_i(\bx^*) \in \partial f_0(\bx^*) + \hN_X(\bx^*).
\]
From the convexity of $f_0$ and $X$, it follows that $ f_0(\bx) - f_0(\bx^*)\geq - \big\langle\sum_{i=1}^{M} z_i^*\tilde{\nabla}f_i(\bx^*), \bx-\bx^*\big \rangle, \forall \bx\in X$.
Since  $f_i$ is convex for each $i\in [M]$, we have  $f_i(\bx) - f_i(\bx^*) \geq \langle \tilde{\nabla}f_i(\bx^*),\bx-\bx^*\rangle$. 
Noticing $\bz^*\geq \0$, we have  for any $ \bx\in X$,
\[ f_0(\bx) - f_0(\bx^*)\geq - \sum_{i=1}^{M} z_i^*\langle\tilde{\nabla}f_i(\bx^*), \bx-\bx^* \rangle \geq -\sum_{i=1}^M  z_i^*(f_i(\bx)-f_i(\bx^*)).\]
Because $z_i^*f_i(\bx^*)=0$ for all $i\in [M]$, we obtain \eqref{eq:2.41}.

Furthermore, for  any $\bz\geq\0$, we  have $\langle \bz, \bff(\bx^*)\rangle \leq 0$ from  $f_i(\bx^*)\leq 0$,  $\forall i \in [M]$. Hence, combining with \eqref{eq:2.41}, we have the inequality in \eqref{eq:2.4}.
\end{proof} 

\section{Proof of Lemma~\ref{lem:bound_v}}\label{app:bound_v}
\begin{proof}  
For $\widehat{\bu}^k$ given in \eqref{eq:u_hat},  we  have  $\|\widehat{\bu}^k\| \leq \theta $ and thus each coordinate of $\widehat{\bu}^k$  is also less than $\theta$, i.e.  $-\theta  \1 \leq \widehat{\bu}^k \leq \theta \1$. 
Recursively  rewriting the updates in  \eqref{eq:m_update}, \eqref{eq:v_update1} and  \eqref{eq:v_update2} gives 
\begin{align}
&\bm^k = \beta_{1}\bm^{k-1} + (1-\beta_{1})\bu^k= (1-\beta_{1})\sum_{j=1}^{k} \beta_{1}^{k-j} \bu^j, \label{eq:m_recursive} \\
&\bv^k = \beta_2 \bv^{k-1} + (1-\beta_2)(\widehat\bu^k)^2=(1-\beta_2)\sum_{j=1}^k \beta_2^{k-j} (\widehat\bu^j)^2,\label{eq:v_recursive} \\
&\widehat{\bv}^k =  \max\{\widehat{\bv}^{k-1}, \bv^k\} = \max_{j\in [k]}  \bv^j,\label{eq:vhat_recursive}
\end{align}
here, $\bm^0=\bv^0=\widehat{\bv}^0=\0$. 
By \eqref{eq:v_recursive} and $-\theta  \1 \leq\widehat{\bu}^k \leq \theta  \1$, we  have  $ \bv^k  \leq \theta^2 (1-\beta_2)\sum_{j=1}^k \beta_2^{k-j} \1  \leq  \theta^2 \1$.  By \eqref{eq:vhat_recursive}, we  further have 
$ \widehat{\bv}^k   \leq \theta^2 \1$. Thus $ \bbE \big[\|(\widehat{\bv}^k)^{1/2}\|_1\big] \leq n\theta$ holds.

Notice $  \bbE \big[\|\bm^{k}\|_{(\widehat\bv^k)^{-{1/2}}}^2\big]  = \bbE\big[\|\frac{\bm^{k}}{(\widehat\bv^k)^{{1/4}}}\|^2\big]$. We can lower bound $\bv^k$  by  keeping only the last term in  \eqref{eq:v_recursive}  since  $(\widehat\bu^{j})^2\geq \0$, i.e. $ \bv^k \geq (1-\beta_2)(\widehat\bu^k)^2 $.  By \eqref{eq:vhat_recursive}, we also have 
$\widehat{\bv}^k   \geq  (1-\beta_2)  \max_{j\in [k]} (\widehat\bu^j)^2.$
 Plugging the inequality and  \eqref{eq:m_recursive} into $  \bbE \big[\|\bm^{k}\|_{(\widehat\bv^k)^{-{1/2}}}^2\big]  $ gives
 \begin{align}
 &\bbE  \big[ \|\bm^{k}\|_{( \widehat{\bv}^k )^{-{1/2}}}^2\big] \leq \bbE\left[ \left\|  \frac{\bm^k}{{\big((1-\beta_2) \max_{j\in [k]}  (\widehat\bu^j)^2\big)^{1/4}}} \right\|^2\right]\nonumber\\ =&\frac{(1\!-\!\beta_1)^2}{(1-\beta_2)^{1/2}}  \bbE\left[ \left\| \frac{\sum_{j=1}^{k} \beta_{1}^{k-j} \bu^j }{\max_{j\in [k]} \mid\widehat\bu^j\mid^{1/2}}\right\|^2\right]. \label{eq:311}
\end{align} 
Then we bound $\left\|\frac{\sum_{j=1}^{k} \beta_{1}^{k-j} \bu^j }{\max_{j\in [k]} \mid\widehat\bu^j\mid^{1/2}} \right\|^2$ by the Cauchy-Schwarz inequality.
\begin{align}
&\left\|\frac{\sum_{j=1}^{k} \beta_{1}^{k-j} \bu^j }{\max_{j\in [k]} \mid\widehat\bu^j\mid^{1/2}}\right\|^2 \!=\! \left\|\sum_{j=1}^k (\beta_1^{k-j})^{1/2}  \frac{(\beta_{1}^{k-j})^{1/2} \bu^j }{\max_{j\in [k]} \mid\widehat\bu^j\mid^{1/2}}\right\|^2 \nonumber\\
\leq&  \bigg( \sum_{j=1}^{k} \Big((\beta_{1}^{k-j})^{1/2}\Big)^2  \bigg)
\bigg(\sum_{j=1}^{k}\left\|\frac{(\beta_{1}^{k-j})^{1/2}  \bu^j }{\max_{j\in [k]} \mid\widehat\bu^j\mid^{1/2}}\right\|^2\bigg) \nonumber\\
=& \Big( \sum_{j=1}^{k}  \beta_{1}^{k-j}  \Big)
\bigg(\sum_{j=1}^{k}\beta_{1}^{k-j}\left\|\frac{ \bu^j }{\max_{j\in [k]} \mid\widehat\bu^j\mid^{1/2}}\right\|^2\bigg)\nonumber\\
\leq & \Big( \sum_{j=1}^{k}  \beta_{1}^{k-j}  \Big)
\bigg(\sum_{j=1}^{k}\beta_{1}^{k-j}\left\|\frac{ \bu^j }{\mid\widehat\bu^j\mid^{1/2}}\right\|^2\bigg),
\label{eq:312}
\end{align}
where we  use $\max_{j\in [k]} \mid\widehat\bu^j\mid^{1/2}\geq  \mid\widehat\bu^j\mid^{1/2}$ for  $j\in [k]$ in the second inequality.
For  $\left\|\frac{ \bu^j }{\mid\widehat\bu^j\mid^{1/2}}\right\|^2$, we have $\widehat\bu^j$ given in \eqref{eq:u_hat} and notice $\max\big\{1, \frac{\|\bu^k\|}{\theta}\big\}$  is a scalar.   
\begin{align*}
&\left\|\frac{ \bu^j }{\mid\widehat\bu^j\mid^{1/2}}\right\|^2=
\bigg\|\frac{ \bu^j }{ \mid\frac{\bu^j}{\max\big\{1, \frac{\|\bu^j\|}{\theta}\big\}}\mid^{1/2}}\bigg\|^2 =   \max\big\{1, \frac{\|\bu^j\|}{\theta}\big\}  \left\|  \bu^j  \right\|_1 \nonumber\\
\leq&  \max\big\{1, \frac{\|\bu^j\|}{\theta}\big\}  \sqrt{n}\big\|  \bu^j  \big\|
= \begin{cases}  
 \sqrt{n} \big\|\bu^j\big\|  , & \mbox{if } \|\bu^j\| \leq \theta, \\ 
  \frac{\sqrt{n}\big\|\bu^j\big\|^2}{\theta} , & \mbox{if }  \|\bu^j\|> \theta.
  \end{cases}
\end{align*}  
So we get $\left\|\frac{ \bu^j }{\mid\widehat\bu^j\mid^{1/2}}\right\|^2\leq  \sqrt{n} \left(\theta + \frac{\|\bu^j \|^2}{\theta}\right)$.
Plug the inequality back to \eqref{eq:312}, and then \eqref{eq:312}  back to  \eqref{eq:311}. 
\begin{align*}
 &\bbE \big[\|\bm^{k}\|_{(\widehat\bv^k)^{-{1/2}}}^2\big]  \leq \frac{(1\!-\!\beta_1)^2}{(1-\beta_2)^{1/2}}  \bbE\left[   \Big( \sum_{j=1}^{k}  \beta_{1}^{k-j}  \Big)
\Bigg(\sum_{j=1}^{k}\beta_{1}^{k-j} 
 \sqrt{n} \bigg(\theta + \frac{\|\bu^j \|^2}{\theta}\bigg)
\Bigg)\right] \\
= & \frac{ \sqrt{n} (1\!-\!\beta_1)^2}{(1-\beta_2)^{1/2}}   \Big( \sum_{j=1}^{k}  \beta_{1}^{k-j}  \Big)
\Bigg(\sum_{j=1}^{k}\beta_{1}^{k-j} 
\bigg(\theta + \frac{ \bbE\left[ \|\bu^j \|^2\right]}{\theta}\bigg)
\Bigg)\\
\leq&  \frac{ \sqrt{n} (1\!-\!\beta_1)^2}{(1-\beta_2)^{1/2}}   \Big( \sum_{j=1}^{k}  \beta_{1}^{k-j}  \Big)^2\max_{j\in[k]} 
\bigg(\theta + \frac{ \bbE\left[ \|\bu^j \|^2\right]}{\theta}\bigg) \\
\leq& \frac{\sqrt{n} \Big(\theta + \frac{\max_{j\in[k]} \bbE \left[ \|\bu^j \|^2 \right]}{\theta}\Big)}{(1-\beta_2)^{1/2}},
\end{align*} 
where we  have used  $ \sum_{j=1}^{k} \beta_{1}^{k-j}\leq \sum_{j=1}^{\infty} \beta_{1}^{j}\leq \frac{1}{1-\beta_1}$ for $\beta_1\in(0,1)$ in the  last inequality. 
With \eqref{eq:P_hat}, the proof  is finished. 
\end{proof} 

\section{Proof of Lemma~\ref{lem:primal_x}}\label{app:primal_x}
\begin{proof}
 From the projection  \eqref{eq:x_update} in the primal variable update, we have for $k\in  [K]$ and $ \forall \bx\in X$,
\begin{align}\label{eq:x1} 
0&\geq \left\langle \bx^{k+1}-\bx, \bx^{k+1}-\Big(\bx^k - \alpha_k \bm^k/(\widehat{\bv}^k)^{1/2}\Big)\right\rangle_{(\widehat{\bv}^k)^{1/2}}\nonumber\\
&= \left\langle \bx^{k+1}-\bx, (\widehat{\bv}^k)^{1/2}\big(\bx^{k+1}- \bx^k\big) + \alpha_k \bm^k \right\rangle.
\end{align}
The first term of the right side equals to
\begin{align}
&\left\langle \bx^{k+1}-\bx, (\widehat{\bv}^k)^{1/2}\Big(\bx^{k+1} - \bx^k\Big) \right\rangle \nonumber\\
=& \frac{1}{2} \left( \big\| \bx^{k+1}-\bx \big\|^2_{(\widehat{\bv}^k)^{1/2}} - \big\|\bx^k - \bx\big\|^2_{(\widehat{\bv}^k)^{1/2}} + \big\|\bx^{k+1} - \bx^k \big\|^2_{(\widehat{\bv}^k)^{1/2}} \right).  \label{eq:x_abc}
\end{align}
Recursively rewrite $ \big\langle \bx^{k+1}-\bx,  \bm^k \big\rangle $ with the update  \eqref{eq:m_update} 
\begin{align}
 &\big\langle \bx^{k+1}-\bx,  \bm^k \big\rangle \\
=& \big\langle \bx^{k+1} -\bx^k, \bm^k \big\rangle + (1-\beta_{1})\big\langle\bx^k - \bx, \bu^k \big\rangle + \beta_{1}\big\langle\bx^k - \bx,\bm^{k-1} \big\rangle \nonumber \\
=&  \big\langle \bx^{k+1} -\bx^k, \bm^k \big\rangle + (1-\beta_{1})\big\langle\bx^k - \bx, \bu^k \big\rangle  \nonumber\\
&  +  \beta_{1}\left( \big\langle \bx^{k} -\bx^{k-1}, \bm^{k-1} \big\rangle + (1-\beta_{1})\big\langle\bx^{k-1} - \bx, \bu^{k-1} \big\rangle\right) \nonumber\\
& ....\nonumber\\
& + \beta_{1}^{k-1}\left(\big\langle \bx^{2} -\bx^{1}, \bm^{1} \big\rangle + (1-\beta_{1})\big\langle\bx^{1} - \bx, \bu^{1} \big\rangle\right) + \beta_{1}^k\big\langle\bx^{1} - \bx,\bm^{0}\big\rangle \nonumber\\ 
=& \sum_{j=1}^k \beta_1^{k-j} \Big(\big\langle \bx^{j+1} -\bx^{j}, \bm^{j} \big\rangle + (1-\beta_{1})\big\langle\bx^{j} - \bx, \bu^{j}\big\rangle \Big), \label{eq:m_ite}
\end{align}
where the  second equation recursively applied the first equation and 
 the last term $\beta_{1}^k\big\langle\bx^{1} - \bx,\bm^{0}\big\rangle$ vanishes because $\bm^0 = \0$. Plugging equations \eqref{eq:x_abc} and \eqref{eq:m_ite} into the inequality \eqref{eq:x1} gives 
\begin{align}
&\alpha_k \sum_{j=1}^k \beta_1^{k-j}  \Big(\big\langle \bx^{j+1} -\bx^{j}, \bm^{j} \big\rangle + (1-\beta_{1})\big\langle\bx^{j} - \bx, \bu^{j} \big\rangle \Big) \nonumber\\
\leq & \frac{1}{2} \left( -\big\| \bx^{k+1}-\bx\big\|^2_{(\widehat{\bv}^k)^{1/2}} +\big\|\bx^k - \bx\big\|^2_{(\widehat{\bv}^k)^{1/2}} - \big\|\bx^{k+1} - \bx^k \big\|^2_{(\widehat{\bv}^k)^{1/2}}\right).\label{eq:primal_x_before_sum}
\end{align}
Sum the above inequality \eqref{eq:primal_x_before_sum} for $k=1$ to $t$. About the left side, we have
\begin{align}
&\sum_{k=1}^t\alpha_k \sum_{j=1}^k \beta_1^{k-j}  \Big(\big\langle \bx^{j+1} -\bx^{j}, \bm^{j} \big\rangle\Big) = \sum_{j=1}^t  \Big(\big\langle \bx^{j+1} -\bx^{j}, \bm^{j} \big\rangle\Big)\sum_{k=j}^t \alpha_k  \beta_1^{k-j}\nonumber\\
\geq&\sum_{j=1}^t\Big(
  -\frac{\big\|\bx^{j+1} -\bx^{j}\big\|^2_{(\widehat{\bv}^j)^{1/2}}}{2\sum_{k=j}^t \alpha_k  \beta_1^{k-j}}-\frac{\sum_{k=j}^t \alpha_k  \beta_1^{k-j}}{2}\big\|\bm^{j}\big\|^2_{(\widehat{\bv}^j)^{-{1/2}}}  \Big) \sum_{k=j}^t \alpha_k  \beta_1^{k-j}\nonumber\\
\overset{\eqref{eq:beta_ineq}}\geq&\sum_{j=1}^t\Big(
  -\frac{\big\|\bx^{j+1} -\bx^{j}\big\|^2_{(\widehat{\bv}^j)^{1/2}}}{2}-\frac{\alpha_j^2\big\|\bm^{j}\big\|^2_{(\widehat{\bv}^j)^{-{1/2}}}}{2(1\!-\!\beta_1)^2} \Big).\label{eq:primal_sum_left}
\end{align} 
About the right side  of the sum of  the inequality \eqref{eq:primal_x_before_sum}, by $(\widehat{\bv}^{k})^{1/2} \geq (\widehat{\bv}^{k-1})^{1/2}\geq \0$, $k \in [t]$ since the iteration \eqref{eq:v_update2}, and Assumption~\ref{assu:compact}, we have  
\begin{align}
& \sum_{k=1}^{t} \frac{1}{2} \left( -\big\| \bx^{k+1}-\bx \big\|^2_{(\widehat{\bv}^k)^{1/2}} + \big\|\bx^k - \bx\big\|^2_{(\widehat{\bv}^k)^{1/2}}\right)\nonumber\\
=&  \frac{1}{2} \left(-\big\| \bx^{t+1}-\bx \big\|^2_{(\widehat{\bv}^t)^{1/2}} + \sum_{k=2}^{t} \big\| \bx^{k}-\bx\big\|^2_{(\widehat{\bv}^k)^{1/2}-(\widehat{\bv}^{k-1})^{1/2}} + \big\|\bx^1 - \bx\big\|^2_{(\widehat{\bv}^1)^{1/2}}\right)\nonumber\\
\leq &  \frac{B^2}{2} \left( \sum_{k=2}^{t} \big\|(\widehat{\bv}^k)^{1/2}-(\widehat{\bv}^{k-1})^{1/2}\big\|_1 +\big\|(\widehat{\bv}^1)^{1/2}\big\|_1\right) = \frac{B^2}{2} \big\|(\widehat{\bv}^t)^{1/2}\big\|_1.\label{eq:primal_sum_right}
\end{align}
Thus with inequalities \eqref{eq:primal_sum_left} and \eqref{eq:primal_sum_right}, the sum of the inequality \eqref{eq:primal_x_before_sum} for $k=1$ to $t$ becomes
\begin{align*}
&\sum_{k=1}^t \! \alpha_k \! \sum_{j=1}^k \! \beta_1^{k-j} (1\!-\!\beta_{1})\big\langle\bx^{j} \!-\! \bx, \bu^{j} \big\rangle  \!+\! 
\sum_{j=1}^t\!\Big(\!
 -\!\frac{\big\|\bx^{j+1} \!-\!\bx^{j}\big\|^2_{(\widehat{\bv}^j)^{1/2}}}{2}\!-\!\frac{\alpha_j^2\big\|\bm^{j}\big\|^2_{(\widehat{\bv}^j)^{\!-\!{1/2}}}}{2(1\!-\!\beta_1)^2} \Big)\\
\leq& \frac{B^2}{2} \big\|(\widehat{\bv}^t)^{1/2}\big\|_1 -   \frac{1}{2}\sum_{k=1}^t 
 \big\|\bx^{k+1} \!-\! \bx^k \big\|^2_{(\widehat{\bv}^k)^{1/2}}.
\end{align*}
Eliminating the term $\big\|\bx^{j+1} -\bx^{j}\big\|_{(\widehat{\bv}^j)^{1/2}}$ on  both sides and exchanging the order of sums in the first term give 
\begin{align*}%\label{eq:primal_x}
  (1-\beta_{1}) \sum_{j=1}^t \big\langle\bx^{j} - \bx, \bu^{j} \big\rangle \sum_{k=j}^t \alpha_k  \beta_1^{k-j}
\leq \frac{B^2}{2} \big\|(\widehat{\bv}^t)^{1/2}\big\|_1 +  \frac{\sum_{j=1}^t \alpha_j^2\big\|\bm^{j}\big\|^2_{(\widehat{\bv}^j)^{-1/2}}}{2(1\!-\!\beta_1)^2} . 
\end{align*}
Then  take the expectation on  the above inequality.  With the  bounds given in Lemma~\ref{lem:bound_v}, we have
\begin{align*}%\label{eq:primal_x_bbE}
 & (1-\beta_{1}) \sum_{j=1}^t \bbE\left[\big\langle\bx^{j} - \bx, \bu^{j} \big\rangle\right] \sum_{k=j}^t \alpha_k  \beta_1^{k-j}
\leq \frac{n\theta B^2}{2}   +  \frac{\sum_{j=1}^t \alpha_j^2  \frac{\sqrt{n}(\theta + \frac{\widehat{P}_\bz^j}{\theta})}{(1-\beta_2)^{1/2}} }{2(1\!-\!\beta_1)^2}\\
\leq&  \frac{n\theta B^2}{2}   +  \frac{ \sqrt{n}(\theta + \frac{\widehat{P}_\bz^t}{\theta}) \sum_{j=1}^t \alpha_j^2 }{2(1\!-\!\beta_1)^2(1\!-\!\beta_2)^{1/2}},
\end{align*}
where  the last inequality holds because we notice $\widehat{P}_\bz^k$ defined in  \eqref{eq:P_hat} is nondecreasing with respect to $k$. 
\end{proof} 

\section{Proof of Lemma~\ref{lem:dual_z}}\label{app:dual_z}
\begin{proof}
For the  dual variable is projected to the  positive region in the update \eqref{eq:z_update}, it follows that for any $\bz\geq 0$, $j\in  [K]$,
\begin{align*}
&\Big\langle\bz^{j+1}-\bz, \bz^{j+1}-\big(\bz^j + {\rho_j}\bw^j\big) \Big\rangle  \leq 0. 
\end{align*}
It could be rewritten as 
\begin{align}\label{eq:z1}
 &\big\langle\bz^{j+1}-\bz, \bz^{j+1}-\bz^j\big \rangle  \leq  \big\langle\bz^{j+1}-\bz, {\rho_j}\bw^j\big\rangle =  \big\langle\bz^{j+1}-\bz^j, {\rho_j}\bw^j\big\rangle + \big\langle\bz^{j}-\bz, {\rho_j}\bw^j\big\rangle. 
\end{align}
For each term of the above inequality \eqref{eq:z1}, we have 
\begin{align*}
\big\langle\bz^{j+1}-\bz, \bz^{j+1}-\bz^j \big\rangle &= \frac{1}{2}\Big( \big\|\bz^{j+1}-\bz^j\big\|^2+\big\|\bz^{j+1} - \bz\big\|^2 - \big\|\bz^j-\bz\big\|^2\Big),\\
\big\langle\bz^{j+1}-\bz^j, \bw^j\big\rangle &\leq \frac{1}{2{\rho_j}}\big\|\bz^{j+1}-\bz^j\big\|^2 + \frac{{\rho_j}}{2} \big\|\bw^j\big\|^2, \\
\big\langle\bz^{j}-\bz,  \bw^j\big\rangle & = \big\langle\bz^{j}-\bz,  \bw^j - \bff(\bx^j)\big\rangle + \big\langle\bz^{j}-\bz,   \bff(\bx^j)\big\rangle. 
\end{align*}
Plugging the above three terms into the inequality \eqref{eq:z1} and eliminating $\big\|\bz^{j+1}-\bz^j\big\|^2$ give
\[\frac{1}{2\rho_j}\Big(\big\|\bz^{j+1} - \bz\big\|^2 \!-\! \big\|\bz^j-\bz\big\|^2\Big) \leq  \frac{{\rho_j}}{2} \big\|\bw^j\big\|^2 \!+\! \big\langle\bz^{j}-\bz,  \bw^j - \bff(\bx^j)\big\rangle \!+\! \big\langle\bz^{j}-\bz,   \bff(\bx^j)\big\rangle.\]
Rearranging the above inequality gives the inequality \eqref{eq:dual_z}.
\end{proof} 

\section{Proof of Lemma~\ref{lem:sum_u}}\label{app:sum_u}
\begin{proof}
For any $j\in  [K]$, we have 
\begin{equation}\label{eq:xx}
\big\langle\bx^{j} - \bx, \bu^{j} \big\rangle =  \big\langle\bx^{j} - \bx, \tilde{\nabla}_\bx \hL (\bx^j,\bz^j) \big\rangle+\big\langle\bx^{j} - \bx, \bu^{j}-\tilde{\nabla}_\bx \hL (\bx^j,\bz^j) \big\rangle.
\end{equation}
Here $ \tilde{\nabla}_\bx \hL (\bx^j,\bz^j)  = \bbE\big[\bu^j\mid\hH^j\big]  \in \partial_\bx \hL(\bx^j,\bz^j)$ according  to Assumption~\ref{assu:bound}.
By the convexity of $f_i(\bx), i=0,1,..., M$, we know $\hL (\bx,\bz)$ is convex with respect to $\bx$ and 
\begin{align*}
& \big\langle \bx^j \!-\!\bx, \tilde{\nabla}_\bx \hL (\bx^j,\bz^j) \big\rangle \!\geq\! \hL (\bx^j,\bz^j) \!-\! \hL (\bx,\bz^j) 
\!=\! f_0(\bx^j) \!- \! f_0(\bx) \!+\! \big\langle \bz^j, \bff(\bx^j)\big\rangle\!-\!\big\langle \bz^j,\bff(\bx)\big\rangle.
\end{align*}
Plug the  lower  bound of  $\langle \bz^j, \bff(\bx^j)\rangle$ given in Lemma~\ref{lem:dual_z}  to the above inequality.
\begin{align*}
&\big\langle \bx^j-\bx, \tilde{\nabla}_\bx \hL (\bx^j,\bz^j) \big\rangle \\
\geq & f_0(\bx^j)-f_0(\bx) -  \big\langle\bz^j, \bff(\bx)\big\rangle +  \big\langle\bz, \bff(\bx^j)\big\rangle  + \frac{1}{2\rho_j}\Big(\big\|\bz^{j+1} - \bz\big\|^2 - \big\|\bz^j-\bz\big\|^2\Big) \\
&- \frac{{\rho_j}}{2} \big\|\bw^j\big\|^2 -\big\langle\bz^{j}-\bz,  \bw^j - \bff(\bx^j)\big\rangle.
\end{align*}
Summarizing \eqref{eq:xx} with weights $ \sum_{k=j}^t \alpha_k\beta_1^{k-j}$ for $j\in[t]$, and plugging the above inequality  give
\begin{align}
&\sum_{j=1}^t \big\langle\bx^{j} - \bx, \bu^{j} \big\rangle \sum_{k=j}^t \alpha_k\beta_1^{k-j}\nonumber\\
\geq& \sum_{j=1}^t \bigg(f_0(\bx^j)-f_0(\bx) - \big\langle\bz^j, \bff(\bx)\big\rangle + \big\langle\bz, \bff(\bx^j)\big\rangle +\frac{1}{2\rho_j}\Big(\big\|\bz^{j+1} - \bz\big\|^2 -\big\|\bz^j-\bz\big\|^2\Big) \nonumber\\
&  - \frac{\rho_j}{2} \big\|\bw^j\big\|^2-\big\langle\bz^{j}-\bz,  \bw^j - \bff(\bx^j)\big\rangle   +\big\langle\bx^{j} - \bx, \bu^{j}-\tilde{\nabla}_\bx \hL (\bx^j,\bz^j) \big\rangle \bigg) \sum_{k=j}^t \alpha_k\beta_1^{k-j}.\label{eq:sum_u0}
\end{align}
Summation of the term about $\big\|\bz^{j+1} - \bz\big\|^2 - \big\|\bz^j-\bz\big\|^2$ can be lower bounded:
\begin{align}
&\sum_{j=1}^t  \frac{1}{2\rho_j}\Big(\big\|\bz^{j+1} - \bz\big\|^2 - \big\|\bz^j-\bz\big\|^2\Big)  \sum_{k=j}^t \alpha_k\beta_1^{k-j}\nonumber\\
  = &  -\frac{ \sum_{k=1}^t \alpha_k  \beta_1^{k-1} }{2\rho_1}\big\|\bz^1-\bz\big\|^2 + \frac{\alpha_t }{2\rho_t}\big\|\bz^{t+1}-\bz\big\|^2 \nonumber\\
  & + \sum_{j=2}^t\left( \frac{\sum_{k=j-1}^t \alpha_k  \beta_1^{k-(j-1)}}{2\rho_{j-1}} - \frac{\sum_{k=j}^t \alpha_k  \beta_1^{k-j}}{2\rho_j}\right) \big\|\bz^j-\bz\big\|^2 \nonumber \\
\overset{\eqref{eq:rho_1} }\geq &  -\frac{ \sum_{k=1}^t \alpha_k  \beta_1^{k-1} }{2\rho_1}\big\|\bz^1- \bz\big\|^2 + \frac{\alpha_t }{2\rho_t}\big\|\bz^{t+1}-\bz\big\|^2 \nonumber\\ \overset{\eqref{eq:beta_ineq}}\geq&  -\frac{ \alpha_1\big\| \bz^1-\bz\big\|^2}{2\rho_1(1-\beta_{1})}  + \frac{\alpha_t }{2\rho_t}\big\|\bz^{t+1}-\bz\big\|^2.  \label{eq:sum_z1}
\end{align} 
Summation of $\big\|\bw^j\big\|^2$ can also be lower bounded
\begin{align}
\!-\! \sum_{j=1}^t   \frac{{\rho_j}}{2} \big\|\bw^j\big\|^2  \sum_{k=j}^t \alpha_k\beta_1^{k-j} 
\!\overset{\eqref{eq:beta_ineq}}\geq \!  -\! \sum_{j=1}^t   \frac{{\rho_j}\alpha_j\big\|\bw^j\big\|^2}{2(1\!-\!\beta_1)}
\!\overset{ \eqref{eq:rho_2}}\geq\! - \! \frac{ \rho_1}{2\alpha_1(1\!-\!\beta_1)^2} \sum_{j=1}^t \alpha_j^2 \|\bw^j\|^2.
 \label{eq:bound_w}
\end{align}  
Plugging the above two inequalities \eqref{eq:sum_z1} and \eqref{eq:bound_w} into the inequality \eqref{eq:sum_u0} gives
\begin{align*} 
& \sum_{j=1}^t \big\langle\bx^{j} - \bx, \bu^{j}\big\rangle \sum_{k=j}^t \alpha_k\beta_1^{k-j}\nonumber\\ 
\geq&  \sum_{j=1}^t \Big(f_0(\bx^j)-f_0(\bx) - \big\langle\bz^j, \bff(\bx)\big\rangle  +  \big\langle\bz, \bff(\bx^j)\big\rangle\Big) \sum_{k=j}^t \alpha_k\beta_1^{k-j} \\
& -\frac{ \alpha_1\big\|\bz^1- \bz\big\|^2 }{2\rho_1 (1-\beta_{1})} 
+ \frac{\alpha_t }{2\rho_t}\big\|\bz^{t+1}-\bz\big\|^2  -  \frac{ \rho_1 \sum_{j=1}^t \alpha_j^2 \|\bw^j\|^2}{2\alpha_1(1\!-\!\beta_1)^2} 
\\
& +\sum_{j=1}^t \Big(\!-\!\big\langle\bz^{j}\!-\!\bz,  \bw^j \!-\! \bff(\bx^j)\big\rangle +\big\langle\bx^{j} \!-\! \bx, \bu^{j}\!-\!\tilde{\nabla}_\bx \hL (\bx^j,\bz^j) \big\rangle  \Big) \sum_{k=j}^t \alpha_k\beta_1^{k-j}.
\end{align*}   
Then taking the  expectation on  the above  inequality  and  using  Assumption~\ref{assu:bound}  to bound $\bbE\big[\big\|\bw^j\big\|^2\big]$ give the result \eqref{eq:sum_u}. 
\end{proof} 

\section{Proof of Lemma~\ref{lem:cov}}\label{app:cov}
\begin{proof}
If $(\bx,\bz)$ are deterministic, we can prove  \eqref{eq:stochastic_x_z2} by the conditional expectation and Assumption~\ref{assu:bound}, i.e.
\begin{align*}
 & \bbE \left[\big\langle\bz^{j}-\bz,  \bw^j - \bff(\bx^j)\big\rangle\right] = \bbE\left[\bbE\big[\big\langle\bz^{j}-\bz,  \bw^j - \bff(\bx^j)\big\rangle \mid \hH^k\big]\right]\\
&\hspace{0.2cm} = \bbE\left[\Big\langle\bz^{j}-\bz,  \bbE[\bw^j - \bff(\bx^j) \mid \hH^k]\Big\rangle\right]= \bbE\left[\big\langle\bz^{j}-\bz, \0\big\rangle\right] =0,\\
&\bbE \left[\big\langle  \bx^{j} - \bx, \bu^{j}-\tilde{\nabla}_\bx \hL (\bx^j,\bz^j) \big\rangle\right] = \bbE\left[\bbE\Big[ \langle\bx^{j} - \bx, \bu^{j}-\tilde{\nabla}_\bx \hL (\bx^j,\bz^j) \rangle \mid \hH^k\Big] \right] \\
&\hspace{0.2cm} = \bbE\left[\Big\langle\bx^{j} - \bx, \bbE\big[ \bu^{j}-\tilde{\nabla}_\bx \hL (\bx^j,\bz^j) \mid \hH^k\big] \Big\rangle \right] = \bbE\left[ \langle\bx^{j} - \bx,\0\rangle  \right] =0.
\end{align*}

Then we prove the stochastic case through considering the left two terms of \eqref{eq:stochastic_x_z}, separately, in a similar way.
Let $\tilde{\bz}^1 = \bz^1, \tilde{\bz}^{j+1} = \tilde{\bz}^j - \gamma_j( \bw^j - \bff(\bx^j))$, then $ \bz^j-\tilde{\bz}^j$ is known given $\hH^k$ and  we have $\bbE \big[ \big\langle \bz^j-\tilde{\bz}^j, \bw^j - \bff(\bx^j) \big\rangle\big] = 0$  like the above deterministic case.  Thus  we have 
\begin{align}
  &\sum_{j=1}^t \gamma_j \bbE \big[ \big\langle \bz^{j}-\bz,  \bw^j - \bff(\bx^j)\big\rangle\big] \nonumber\\
  =& \sum_{j=1}^t \gamma_j\bbE \big[ \big\langle \tilde{\bz}^{j}-\bz,  \bw^j - \bff(\bx^j)\big\rangle\big]  =\sum_{j=1}^t \bbE \big[ \big\langle \tilde{\bz}^{j}-\bz,   \tilde{\bz}^j-\tilde{\bz}^{j+1}  \big\rangle\big]  \nonumber\\
  =& \sum_{j=1}^t\frac{1}{2}\bbE\left[ \big\|\tilde{\bz}^{j}-\bz\big\|^2 + \big\|\tilde{\bz}^j-\tilde{\bz}^{j+1}\big\|^2 - \big\|\bz- \tilde{\bz}^{j+1}  \big\|^2\right] \nonumber \\
  =& \frac{1}{2}\bigg(\bbE\big[ \big\|\tilde{\bz}^{1}-\bz\big\|^2\big] + \sum_{j=1}^t\bbE\big[ \big\|\tilde{\bz}^j-\tilde{\bz}^{j+1}\big\|^2\big]  - \bbE\big[ \big\|\bz- \tilde{\bz}^{t+1} \big\|^2\big] \bigg) \nonumber \\
 \leq  & \frac{1}{2}\bigg(\bbE\big[ \big\| \bz^1-\bz\big\|^2\big]  + \sum_{j=1}^t \gamma_j^2\bbE\big[ \big\| \bw^j - \bff(\bx^j)\big\|^2\big] \bigg) \nonumber\\
 \leq &  \frac{1}{2}\bigg(\bbE\big[ \big\|  \bz^1-\bz\big\|^2\big]  + \sum_{j=1}^t \gamma_j^2 \bbE\big[ \big\|\bw^j\big\|^2\big] \bigg) 
 \leq   \frac{1}{2}\bigg(\bbE\big[ \big\|  \bz^1-\bz\big\|^2\big]  + F^2 \sum_{j=1}^t \gamma_j^2 \bigg),\label{eq:sum_z}
\end{align} 
where the first inequality holds because we drop the nonpositive term and $\tilde{\bz}^1 = \bz^1$; the  second inequality holds  because for any random  vector $\bw$,  $\bbE\big[\big\|\bw-\bbE[\bw]\big\|^2\big]\leq\bbE\big[\big\|\bw\big\|^2\big]$, and  here $\bbE\big[\bw^j\mid\hH^j\big] = \bff(\bx^j)$ for $j\in[t]$;  and the  last  inequality holds by  Assumption~\ref{assu:bound}.  

For the summation  of  $\gamma_j\bbE\big[ \big\langle \bx^j-\bx, \bu^j -\tilde{\nabla}_\bx \hL (\bx^j,\bz^j)\big\rangle\big]$,   
let $\tilde{\bx}^1 = \bx^1, \tilde{\bx}^{j+1} = \tilde{\bx}^j + \gamma_j\big(\bu^j -\tilde{\nabla}_\bx \hL (\bx^j,\bz^j)\big)$, then $\bbE\big[\big\langle \bx^j-\tilde{\bx}^j, \bu^j -\tilde{\nabla}_\bx \hL (\bx^j,\bz^j) \big]\big\rangle = 0$ and
\begin{align}
  &- \sum_{j=1}^t \gamma_j\bbE\big[\big\langle \bx^j-\bx, \bu^j -\tilde{\nabla}_\bx \hL (\bx^j,\bz^j)\big\rangle\big] \nonumber \\
 =&  - \sum_{j=1}^t \gamma_j\bbE\big[\big\langle\tilde{\bx}^j-\bx, \bu^j -\tilde{\nabla}_\bx \hL (\bx^j,\bz^j)\big\rangle\big] 
 = \sum_{j=1}^t \bbE\big[\big\langle \tilde{\bx}^{j}-\bx,   \tilde{\bx}^j-\tilde{\bx}^{j+1}  \big\rangle\big] \nonumber\\
 =&\sum_{j=1}^t \frac{1}{2}  \bbE \left[ \big\|\tilde{\bx}^{j}-\bx\big\|^2 + \big\|\tilde{\bx}^j-\tilde{\bx}^{j+1}\big\|^2 - \big\|\bx- \tilde{\bx}^{j+1}  \big\|^2\right] \nonumber\\
 =& \frac{1}{2}\bigg(\bbE \big[\big\|\tilde{\bx}^{1}-\bx\big\|^2\big] + \sum_{j=1}^t\bbE\big[\big\|\tilde{\bx}^j-\tilde{\bx}^{j+1}\big\|^2\big] - \bbE\big[\big\|\bx- \tilde{\bx}^{t+1} \big\|^2\big]\bigg)\nonumber \\
\leq&  \frac{1}{2}\bigg(n B^2 + \sum_{j=1}^t \gamma_j^2 \bbE\big[\big\|\bu^j -\tilde{\nabla}_\bx \hL (\bx^j,\bz^j)\big\|^2\big]\bigg)
   \leq  \frac{1}{2}\bigg(n B^2 + \sum_{j=1}^t \gamma_j^2 \bbE\big[\big\|\bu^j\big\|^2\big]\bigg)\nonumber\\
 \leq  & \frac{1}{2}\bigg(n B^2 +( \max_{j\in [t]} \bbE\big[\big\|\bu^j\big\|^2\big]) \sum_{j=1}^t \gamma_j^2 \bigg)  \leq  \frac{1}{2}\bigg(n B^2 +  \widehat{P}_\bz^t \sum_{j=1}^t \gamma_j^2\bigg),\label{eq:sum_x}
\end{align} 
where the first inequality holds because  we drop the nonpositive term and \eqref{eq:nB2}; the  second inequality holds  because  $\bbE\big[\bu^j\mid\hH^j\big]= \tilde{\nabla}_\bx \hL (\bx^j,\bz^j)$ for $j\in[t]$;  and the  last two inequalities holds by  Assumption~\ref{assu:bound}  and  \eqref{eq:P_hat}.  
  
Adding \eqref{eq:sum_z} and \eqref{eq:sum_x} gives the result \eqref{eq:stochastic_x_z}.
\end{proof}

\section{Proof of Lemma~\ref{lem:3.2}}\label{app:3.2}
\begin{proof} 
Let  $\bx = \bx^*$ in \eqref{eq:lem3.2}, we have  $\bff(\bx^*)\leq \0$ and thus
\begin{equation}\label{eq:3.5}
\bbE \big[f_0(\bar{\bx}) -   f_0(\bx^*)   +  \big\langle\bz, \bff(\bar{\bx})\big\rangle\big] \leq \epsilon_1+\epsilon_0\bbE\big[\big\|\bz\big\|^2\big].
\end{equation}
Since $f_j(\bar\bx)\leq [f_j(\bar\bx)]_+$ and $\bz^*\geq 0$,  we have from \eqref{eq:2.41} that 
\begin{equation}\label{eq:3.6}
-\sum_{j=1}^M z_j^* [f_j(\bar{\bx})]_+ \leq f_0(\bar\bx)-f_0(\bx^*).
\end{equation}
Substituting \eqref{eq:3.6} into \eqref{eq:3.5} with $\bz$ given by  $z_j = 1 + z_j^*$ if $f_j(\bar\bx)>0$ and $z_j  = 0$ otherwise for any $j\in [M]$ gives
\begin{equation*} 
-\bbE\Big[\sum_{j=1}^M z_j^* [f_j(\bar{\bx})]_+\Big] +\bbE\Big[\sum_{j=1}^M (1+z_j^*) [f_j(\bar{\bx})]_+\Big] \leq  \epsilon_1+\epsilon_0  \big\|1+\bz^*\big\|^2.
\end{equation*}
Simplifying the above inequality gives \eqref{eq:3.2.2}.

Letting  $z_j = 3 z_j^*$ if $f_j(\bar\bx)>0$ and $z_j  = 0$ otherwise for any $j\in [M]$ in \eqref{eq:3.5} and adding \eqref{eq:3.6} together gives 
\begin{equation}\label{eq:3.7}
\bbE\Big[\sum_{j=1}^M z_j^* [f_j(\bar{\bx})]_+\Big]\leq \frac{ \epsilon_1}{2}+\frac{9\epsilon_0}{2}\big\|\bz^*\big\|^2.
\end{equation}
Hence,  by  the above inequality and  \eqref{eq:3.6}, we obtain 
\begin{equation}
 -\bbE\big[f_0(\bar\bx)-f_0(\bx^*)\big]\leq  \frac{ \epsilon_1}{2}+\frac{9\epsilon_0}{2} \big\|\bz^*\big\|^2. \label{eq:3.2.7}
\end{equation}
Thus $\bbE\big[[f_0(\bar\bx)-f_0(\bx^*)]_-\big]\leq \frac{\epsilon_1}{2}+\frac{9\epsilon_0}{2}\big\|\bz^*\big\|^2.$
In addition,  from \eqref{eq:3.5} with  $\bz = 0$, it  follows $\bbE\big[f_0(\bar\bx)-f_0(\bx^*)\big]\leq  \epsilon_1$. Since  $\mid a\mid = a + 2[a]_-$  for any real number $a$, we have 
\[
\bbE \big[\mid f_0(\bar\bx)-f_0(\bx^*)\mid\big] \!=\! \bbE\big[f_0(\bar\bx)-f_0(\bx^*)\big]+2\bbE\big[[f_0(\bar\bx)-f_0(\bx^*)]_-\big]\leq 2 \epsilon_1+9\epsilon_0\big\|\bz^*\big\|^2,
\]
which gives  \eqref{eq:3.2.1}.

Furthermore, let  $\bz  = 0$ in \eqref{eq:lem3.2} and take $\widehat\bx\in \mbox{argmin}_{\bx\in X}  f_0(\bx) + \big\langle \bar{\bz},\bff(\bx)\big\rangle$. By equation \eqref{eq:dual_function}, we have $\bbE f_0(\bar\bx)\leq  \bbE d(\bar{\bz}) + \epsilon_1$, which together with  \eqref{eq:strong_dual}  gives 
\begin{equation}
 \bbE\big[d(\bz^*)-d(\bar{\bz})\big] \leq \bbE\big[f_0(\bx^*)-f_0(\bar{\bx})\big]+  \epsilon_1.
\end{equation}
Combining the  above inequality with  \eqref{eq:3.2.7} gives 
\eqref{eq:3.2.3}.
\end{proof} 

\section{Proof of Corollary~\ref{cor:vary_stepL}}\label{app:vary_stepL}
\begin{proof} 
For  the step sizes $\{\alpha_j\}_{j=1}^K$, we have 
\begin{align*}
&\sum_{j=1}^{K} \alpha_j =\alpha \sum_{j=1}^{K} \frac{1}{\sqrt{j+1}}\geq \alpha \int_{s = 2}^{K+2} \frac{1}{\sqrt{s}} \mbox{d}s = 2\alpha(\sqrt{K+2}-\sqrt{2}),\\
&\sum_{j=1}^{K} \alpha_j^2 =\alpha^2 \sum_{j=1}^{K} \frac{1}{j+1}\leq \alpha^2 \int_{s = 1}^{K+1} \frac{1}{s} \mbox{d}s =  \alpha^2\log{(K+1)}.
\end{align*}
Similarly, for $\{\rho_j\}_{j=1}^K$, it  holds 
$ \sum_{j=1}^{K} \rho_j^2\leq  \rho^2\log{(K+1)}$.  
Plug these bounds to the result of Theorem~\ref{thm:convergeL} and note $\log(K+1)\geq 1$ for $K\geq 2$.
We finish the proof.
\end{proof}

\bibliography{adaptive_primal-dual_final}{}% common bib file
\bibliographystyle{plain} 
\end{document}